\documentclass[10pt,a4paper,american]{amsart}
\usepackage[left=3.5cm, right=3.5cm]{geometry}
\usepackage[american]{babel}
\usepackage{amsthm}
\usepackage{amsbsy}
\usepackage{amstext}

\usepackage[scr=boondoxo]{mathalpha}

\usepackage{amsfonts,amsmath,amsthm,latexsym,amssymb,marvosym,esint,xfrac,bbm,amssymb}
\usepackage{graphics, epsfig}
\usepackage{color}
\usepackage{tikz}
\tikzstyle{mybox} = [draw=black, very thick, rectangle, rounded corners, inner ysep=5pt, inner xsep=5pt]
\usepackage{pgfplots}
\usepackage[shortlabels]{enumitem}
\usepackage{appendix}
\usepackage{amsfonts}
\usepackage{mathcomp}
\usepackage{dsfont}
\usepackage{latexsym}
\usepackage{amssymb}
\usepackage{esint}
\usepackage{stmaryrd}
\usepackage{comment}
\usepackage[colorlinks,pdfpagelabels,pdfstartview = FitH,bookmarksopen = true,bookmarksnumbered = true,linkcolor = blue,plainpages = false, hypertexnames = true,citecolor = red,pagebackref=true]{hyperref}
\usepackage{cite}
\usepackage{etoolbox}
\usepackage{skull}

\newtheorem{theo}{Theorem}[section]
\newtheorem{cor}[theo]{Corollary}
\newtheorem{prop}[theo]{Proposition}
\newtheorem{lem}[theo]{Lemma}
\theoremstyle{definition}
\newtheorem{defin}[theo]{Definition}
\newtheorem*{lem*}{Lemma}
\newtheorem{rem}[theo]{Remark}

\newtheorem*{cor*}{Corollary}
\newtheorem*{theo*}{Theorem}

\newcommand{\norm}[1]{\lVert#1\rVert}
\newcommand{\abs}[1]{\left|#1\right|}
\DeclareMathOperator*{\esssup}{ess\,sup}

\DeclareMathOperator*{\supp}{supp}

\def\Ex#1{[\![#1]\!]_h}
\def\Exx#1{[\![#1]\!]_{\bar h}}

\def\XXint#1#2#3{{\setbox0=\hbox{$#1{#2#3}{\int}$ }
\vcenter{\hbox{$#2#3$ }}\kern-.6\wd0}}

\def\XXiint#1#2#3{{\setbox0=\hbox{$#1{#2#3}{\iint}$ }
\vcenter{\hbox{$#2#3$ }}\kern-.55\wd0}}

\renewcommand{\d}{\:\:\!\!\mathrm{d}}
\newcommand{\N}{\ensuremath{\mathbb{N}}}
\newcommand{\R}{\ensuremath{\mathbb{R}}}

\renewcommand{\b}{\mathfrak{b}}

\numberwithin{equation}{section}
\allowdisplaybreaks

 \makeatletter
\patchcmd{\@setaddresses}{\indent}{\noindent}{}{}
\patchcmd{\@setaddresses}{\indent}{\noindent}{}{}
\patchcmd{\@setaddresses}{\indent}{\noindent}{}{}
\patchcmd{\@setaddresses}{\indent}{\noindent}{}{}
\makeatother
 
\pgfplotsset{compat=1.18}

\begin{document}
\renewcommand{\refname}{References}
\renewcommand{\abstractname}{Abstract}

\title[Existence, comparison principle and uniqueness]{Existence, comparison principle and uniqueness for fully nonlinear anisotropic evolution equations}
\date{\today}

\author[A. Nastasi, E. Pe\~na Ayala, M. Vestberg]{Antonella Nastasi, Emiliano Pe\~na Ayala, Matias Vestberg}

\address{Antonella Nastasi\\
Department of Engineering, University of Palermo, Viale delle Scienze, 90128, Palermo, Italy\\
 ORCID ID: 0000-0003-1589-2235}
\email{antonella.nastasi@unipa.it}

\address{Emiliano Pe\~na Ayala\\
Department of Mathematics, Uppsala University,
P.~O.~Box 480, 751 06, Uppsala, Sweden}
\email{emiliano.pena@math.uu.se}

\address{Matias Vestberg\\
Department of Mathematics, Uppsala University,
P.~O.~Box 480, 751 06, Uppsala, Sweden}
\email{matias.vestberg@math.uu.se}

\begin{abstract}
We prove the existence of solutions to the Cauchy-Dirichlet problem associated with a class of fully nonlinear anisotropic evolution equations. We prove a comparison principle and conclude the uniqueness of solutions. All results are obtained under a closeness assumption on the exponents which guarantees that a certain power of the solution has a gradient.
\end{abstract}
\maketitle

\setcounter{tocdepth}{1}

\begin{center}
\begin{minipage}{12cm}
  \small
  \tableofcontents
\end{minipage}
\end{center}

\vskip 0.5cm \noindent 
 
\section{Introduction}\label{sec: intro}
\noindent  This paper proves existence, uniqueness and a comparison principle for solutions to the Cauchy-Dirichlet problem associated with a class of fully nonlinear anisotropic evolution equations of the form
\begin{align}\label{eq:diffusion}
\partial_t u   - \sum^N_{j=1} \partial_j\big(a_j(x,t,u)|\partial_j u^{m_j}|^{p_j-2}\partial_j u^{m_j}\big) = f \quad \text{ in } \quad \Omega_T:=\Omega\times (0,T),
\end{align}
where we use the shorthand notation $\partial_j u^{m_j}:= \partial_j(u^{m_j})$. Here $\Omega \subset \R^N$ is open and bounded, $p_j \in (1,\infty)$ and 
\begin{align*}
m:= \min\{m_j\,|\, j\in \{1,\dots, N\}\} \geq 1.
\end{align*}
We assume that there is a constant $\Lambda>0$ such that the coefficient functions $a_j$ satisfy
\begin{align}\label{cond:unif_ellipt}
 \Lambda^{-1} \leq a_j(x,t,u) \leq \Lambda, \quad (x,t,u)\in \Omega\times(0,T)\times \R, \quad j \in \{1,\dots,N\}.
\end{align}
We also require that each coefficient function $a_j$ is Lipschitz continuous in $u$ for fixed $(x,t)$, that is, there is a constant $c>0$ such that
\begin{align}\label{cond:lip_cont}
 |a_j(x,t,u) - a_j(x,t,v)| \leq c|u - v|.
\end{align}
For the right-hand side we assume
\begin{align}\label{cond:f}
 f \in L^{\sigma \bar p^\prime}(\Omega_T; [0,\infty)), 
\end{align}
where $\bar p$ is defined as 
\begin{align}\label{def:overline_p}
    \bar p: = \bigg( \frac{1}{N} \sum_{j = 1}^N \frac{1}{p_j}\bigg)^{-1},
\end{align}
$\bar p^\prime$ is its corresponding Hölder conjugate exponent, and $\sigma$ satisfies
\begin{align}\label{cond:lower_bound_sigma}
    \sigma >  1 + \tfrac{N}{\bar p}.
\end{align}
\noindent In the isotropic case $p_j \equiv p$, the assumption regarding $f$ coincides with the assumption in \cite{BoeDiVe}.
We consider initial values of the form
\begin{align}\label{cond:u_0} 
    u_0 \in L^\infty(\Omega), \quad u_0 \geq 0. 
\end{align}

\noindent In order to prove existence, we need to impose the following condition which expresses the fact that the exponents $m_j$ are not allowed to differ too much from each other:
\begin{equation}\label{cond:m_j-closeness}
    m_j < p_j' m.
\end{equation}
 In fact, this condition is intimately tied to the very definition of solutions to the Cauchy-Dirichlet problem adopted in this work. Under the assumption \eqref{cond:m_j-closeness} it turns out that for any solution $u$ to \eqref{eq:diffusion}, the function $u^m$ has a spatial gradient. This observation motivates a formulation of the Cauchy-Dirichlet problem which is consistent with that of the special case $m_j \equiv m$, which was considered earlier for example in \cite{Ve}. For the precise definition we refer to Section \ref{sec:setting}. The existence of the gradient of $u^m$ is a surprising regularity result which is not apparent from the natural definition of weak solutions, in which one requires each function $u^{m_j}$ to have a weak derivative only in the $x_j$-coordinate direction. In connection to this result, we also prove the time-continuity of solutions on the time interval into a dual space and into $L^{m+1}_{\textnormal{loc}}(\Omega)$.

  We consider boundary values of the form
\begin{align}\label{cond:g}
    &g \in L^\infty(\Omega_T) \cap  L^{\bf p}(0,T; W^{1, {\bf p}}(\Omega)), \quad \partial_t g \in L^2(\Omega_T),
    \\
    \notag &g \geq \varepsilon_0 >0, \textnormal{ or } g\equiv 0,
\end{align}
and refer to Section \ref{sec:setting} for the definition of the anisotropic Sobolev Space $L^{\bf p}(0,T; W^{1, {\bf p}}(\Omega))$. The condition on the second row is a subtle point related to the choice of a test function which allows us to use weak compactness of Sobolev spaces as a part of the strategy for proving existence (see Lemma \ref{lem:grad_bdd_in_Lp} for details).  

To put our contribution into context we remark that previous works regarding the existence of solutions to the Cauchy-Dirichlet problem for doubly nonlinear anisotropic parabolic equations focused on the special case $m_j\equiv m$. This is the case for example in the work of Sango \cite{Sa} where existence for the model case with vanishing boundary condition and a fairly general right-hand side was proved via a semidiscretization method. Degtyarev and Tedeev in \cite{DeTe} prove the existence of solutions to the model case of the Cauchy problem in $\mathbb{R}^N \times (0, T)$, including an additional $u$-dependent term and a Radon measure as initial data. Also bounded spatial domains and vanishing boundary data on $\Omega \times (0, T)$ are considered in the same setting. Vestberg  treats the case of general structure conditions and boundary data in \cite{Ve} using Galerkin's method, again assuming $m_j\equiv m$.

For the anisotropic porous medium equation, i.e. allowing different $m_j$ but requiring $p_j \equiv 2$, a very singular solution was constructed by V\'azquez in \cite{Va}. However, to our knowledge existence of solutions to the Cauchy-Dirichlet problem in the fully doubly nonlinear case with both different $p_j$ and $m_j$ has not been established previously. Our aim in this article is to start to close this gap.

Our method for proving existence draws inspiration from B\"ogelein-Dietrich-Vestberg \cite{BoeDiVe} as well as Sturm \cite{St}, where the existence of solutions for certain doubly nonlinear isotropic problems is proved. The idea is to approximate our Cauchy-Dirichlet problem with a suitable sequence of anisotropic problems without double nonlinearity and to find a solution in the limit by establishing a sufficiently strong rate of convergence.

Another novelty of this study is the comparison principle (Theorem \ref{thm:comparison_principle}) for weak sub- and super-solutions to \eqref{eq:diffusion}. In simple words, the comparison principle states that a sub-solution $u$ and a super-solution $v$ which satisfy $u \leq v$ on the parabolic boundary of the domain, must have the same property in the whole domain. Despite its seeming simplicity, there are many open questions related to the comparison principle for doubly nonlinear equations; indeed only some special cases have been treated so far. The main difficulties are due to the lack of a weak time derivative and, in particular, at points where the solution is close to zero. Thus, the comparison principle is still far from being fully understood even for doubly nonlinear isotropic problems. 

In this article we prove a comparison principle for solutions to \eqref{eq:diffusion} inspired by the work B\"ogelein-Strunk \cite{BoeStru}, where a comparison principle for the isotropic model case of doubly nonlinear equations was obtained under the assumption that the larger
(super)solution has a positive lower bound on the lateral boundary. Our result is obtained under similar assumptions. The double nonlinearity and the anisotropy involving the exponents $m_j$ and $p_j$, as well as the dependence of the coefficients on the solution itself, make the generalization of the comparison principle to our setting nontrivial. Using the comparison principle, we are able to prove a uniqueness result (Theorem \ref{thm:uniqueness}) for solutions to \eqref{eq:diffusion}.

In this work we consider the range $m_j\geq 1$ for the exponents, but it may be possible to extend our results to the case where the smallest exponent satisfies $m<1$ by means of a change of variables. Namely, the corresponding equation for $v:= u^m$ has the property that the smallest exponent is exactly $1$, however with the additional difficulty of a new exponent $1/m$ appearing in the term with the time derivative. Extending the results of the current paper to this setting would therefore correspond to treating also the case $m<1$. Due to the already high level of technicality in the present work, we chose to leave any further investigation of this case to future publications.

Before concluding the introduction, some historical notes are in order. Doubly nonlinear isotropic equations were first introduced in the 1960s by Lions \cite{Lio} and Kalashnikov \cite{Ka}. The term \textit{doubly nonlinear} refers to the presence of nonlinearity in both the elliptic and parabolic components of the equations. Alternatively, as is the case for the equations we study, both nonlinearities can appear in the elliptic part. Doubly nonlinear equations have a wide range of applications in various physical contexts, including the flow of non-homogeneous non-Newtonian fluids and the simultaneous movement in surface channels of underground water.
Further applications can be found, for example, in the book by Antontsev-D\'iaz-Shmarev \cite{AnDiSh}. 

The existence theory for isotropic doubly nonlinear equations has a long history. Lions \cite{Lio} considered the model case for $p\geq 2$. Raviart \cite{Raviart} and later Bamberger \cite{Ba} studied the model case for the signed orthotropic equation under different assumptions on the source term. Alt-Luckhaus \cite{AlLu} first considered general structure conditions. Later works on existence for various classes of doubly nonlinear isotropic equations include Bernis \cite{Be}, Ivanov \cite{Iv}, Ishige \cite{Ish}, Ivanov-Mkrtychan-J\"ager \cite{IvMkJae} and Laptev \cite{La, La2, La3}.

Anisotropic elliptic problems were introduced in the 1980s by Giaquinta \cite{Gi} and Marcellini \cite{Ma}. Existence of solutions in $\R^n$ in the elliptic and parabolic case without double nonlinearity was proved by Bendahmane and Karlsen in \cite{BeKa}. They also allow advection and lower order terms. 

The literature on comparison principles for various diffusion equations is vast. The parabolic $p$-Laplace equation was treated by DiBenedetto \cite{DiBe}. The model case for anisotropic equations without double nonlinearity was considered by Ciani \cite{Cia}. For comparison principles for the porous medium equation, we refer to Avelin-Lukkari \cite{AvLu}, Kinnunen-Lindqvist-Lukkari\cite{KiLiLu} and Li-Peletier\cite{LiPe}.  

Comparison principles for isotropic doubly non-linear equations were obtained by Alt-Luckhaus \cite{AlLu}, Bamberger \cite{Ba} and D\'iaz \cite{Di}. In these papers, some additional hypotheses on the time derivative are required, which are quite restrictive and are generally difficult to verify. 
Ivanov \cite{Iv} and Ivanov-Mkrtychan-Jäger \cite{IvMkJae} avoid assumptions on the time derivative  and allow time-dependent boundary data, but require bounded and strictly positive sub- and super-solutions. For Trudinger's equation, a special case of doubly nonlinear diffusion, a comparison principle was established by Lindgren-Lindqvist \cite{LiLi} under the hypothesis that either the sub- or super-solution is strictly positive for $p\geq 2$, and with some additional assumption in the case $p\in (1,2)$.

The dependency on time of the boundary data is not assumed in the approach proposed by Otto in \cite{Ot}, which allowed him to prove a comparison principle for weak sub- and super-solutions with time independent lateral boundary data for a class of doubly nonlinear equations. His approach also does not require any extra regularity hypotheses on the sub- and super-solutions. 
Following a similar approach, Bögelein-Duzaar-Gianazza-Liao-Scheven \cite{BoeDuGiaLiSche} proved a comparison principle for nonnegative solutions of general doubly nonlinear isotropic equations involving time-independent vector fields, with one of the solutions vanishing on the lateral boundary. Similar techniques were later adopted by Vestberg \cite{Ve} to obtain a comparison principle under similar assumptions for doubly nonlinear anisotropic equations in the case $m_j \equiv m$. In the model case for anisotropic diffusion
without double nonlinearity, some comparison principles were proved in Ciani-Mosconi-Vespri \cite{CiaMoVe}.


The article is structured as follows. In Section \ref{sec:setting}, we present the general setting and state the main results: the existence theorem (Theorem \ref{thm:existence}), the comparison principle (Theorem \ref{thm:comparison_principle}) and the uniqueness result (Theorem \ref{thm:uniqueness}).  In Section \ref{sec:preliminaries}, we introduce basic tools needed for the arguments, such as Sobolev embeddings, mollifications in time, as well as various algebraic identities and estimates. We also prove a comparison principle for anisotropic equations without double nonlinearity, which is used later in the proof of the existence of solutions. In Section \ref{sec:existence_grad} we establish the existence of the weak gradient of $u^m$.
The proof of the existence of solutions is presented in Section \ref{sec:existence}. Finally, in Section \ref{sec:comparison} we prove the comparison principle (Theorem \ref{thm:comparison_principle}).

\vspace{2mm}
\noindent{\bf Acknowledgments.} This work was supported by the Wallenberg AI, Autonomous Systems and Software Program (WASP) funded by the Knut and Alice Wallenberg Foundation. A. Nastasi is a member of the Gruppo Nazionale per l’Analisi Matematica, la Probabilit\`{a} e le loro Applicazioni (GNAMPA) of the Istituto Nazionale di Alta Matematica (INdAM). Part of this material was obtained when M. Vestberg visited University of Palermo in March - April 2025. The visit was supported by the Centre of Advanced Studies of the University of Palermo.


\section{Setting and main results}\label{sec:setting}
In this section we define the concept of weak solutions to the Cauchy-Dirichlet problem on $\Omega_T := \Omega \times (0,T)$. We also present our main results. We start by presenting the anisotropic Sobolev spaces that appear in the definitions.
 Given a vector ${\bf p} = (p_1,\dots, p_N)$ with $p_j > 1$ we set 
\begin{align*}
W^{1, {\bf p}}(\Omega) &:= \{ v \in W^{1,1}(\Omega)\,|\, \partial_j v \in L^{p_j}(\Omega),\, j \in \{1,\dots, N\}\}.
\end{align*}
When considering solutions to boundary value problems, it is natural to work with the space 
\begin{align*}
 \overline W^{1, {\bf p}}_{\textnormal{o}}(\Omega) := \overline{C^\infty_{\textnormal{o}}(\Omega)} \subset W^{1, {\bf p}}(\Omega),
\end{align*}
i.e. the closure of $C^\infty_{\textnormal{o}}(\Omega)$ in $W^{1, {\bf p}}(\Omega)$ w.r.t. the norm $u\mapsto \norm{u}_{L^1(\Omega)} + \sum^N_{j=1} \norm{\partial_j u}_{L^{p_j}(\Omega)}$. 
As we are working with evolution equations we will also need the following space involving time:
\begin{align*}
 L^{\bf p}(0,T; W^{1, {\bf p}}(\Omega)) &:= \{ v \in L^1(0,T;W^{1,1}(\Omega)) \,|\, \partial_j v \in L^{p_j}(\Omega_T), \, j \in \{1,\dots, N\}\}.
\end{align*}
Aided by the inclusion $i : \overline W^{1, {\bf p}}_{\textnormal{o}}(\Omega) \hookrightarrow L^1(\Omega)$, the continuous linear maps $\partial_k: \overline W^{1, {\bf p}}_{\textnormal{o}}(\Omega) \to L^{p_k}(\Omega)$ and the space $E$  consisting of equivalence classes of all measurable functions $(0,T) \to \overline W^{1, {\bf p}}_{\textnormal{o}}(\Omega)$, we may now introduce the Banach space 
\begin{align*}
 L^{\bf p}(0,T; \overline W^{1, {\bf p}}_{\textnormal{o}}(\Omega)) = \{ u \in E \,|\, i \circ u \in L^1(0,T; L^1(\Omega)), \, \partial_k \circ u \in L^{p_k}(0, T; L^{p_k}(\Omega)),& 
 \\
  k \in \{1,\dots, N\} \}&,
\end{align*}
which represents functions taking the value zero on the lateral boundary $\partial \Omega \times (0,T)$. This is a rather abstract looking definition relying on the Bochner integral, but as one might expect, the space is also isomorphic to a space of functions defined on $\Omega_T$.
\begin{lem}\label{lem:isomorphic-spaces}
 The space $L^{\bf p}(0,T; \overline W^{1, {\bf p}}_{\textnormal{o}}(\Omega))$ is isometrically isomorphic to the closure of $C^\infty_{\textnormal{o}}(\Omega_T)$  in $\{ f \in L^1(\Omega_T)\,| \, \partial_k f \in L^{p_k}(\Omega_T)\}$ with respect to the norm
 \begin{align*}
  v \mapsto \norm{v}_{L^1(\Omega_T)} + \sum^N_{j=1} \norm{\partial_j v}_{L^{p_j}(\Omega_T)}.
 \end{align*}
\end{lem}
\noindent For the proof we refer to \cite{Ve}.

\vspace{2mm}
Given the form of the equation, it seems natural to consider weak solutions of \eqref{eq:diffusion} residing in a set of the form:
\begin{align*}
 V^{\bf p, \bf m}_q := \{ u \in L^q(\Omega_T)\,|\, \partial_j u^{m_j} \in L^{p_j}(\Omega_T)\},
\end{align*}
where the exponent $q$ at least needs to satisfy
\begin{align*}
 q \geq \max  \{m_j\,|\, j\in \{1,\dots, N\}\},
\end{align*}
so that each function $u^{m_j}$ is integrable, which makes it meaningful to investigate the existence of weak derivatives of these functions. In principle, one could also consider local integrability instead.  
The set $V^{\bf p, \bf m}_q$ is closed under scalar multiplication, but not under addition, and hence it is not a vector space. It is however a complete separable metric space when endowed with the metric
\begin{align*}
 d(u,v) := \norm{u - v}_{L^q(\Omega_T)} 
 +
 \sum_{j = 1}^N \norm{\partial_j u^{m_j} - \partial_j v^{m_j}}_{L^{p_j}(\Omega_T)}.
\end{align*}
See Appendix \ref{app:completeness} for the proof of completeness. Separability can be deduced by embedding $V^{\bf p, \bf m}_q$ isometrically into a subset of $L^q(\Omega) \times L^{p_1}(\Omega_T) \times \dots \times L^{p_N}(\Omega_T)$ via the mapping $u \mapsto (u, \partial_1 u^{m_1}, \dots, \partial_N u^{m_N})$. 

We define $q_* := \max\big(\{m+1\} \cup \{m_j\,|\, j\in \{1,\dots, N\}\}\big)$ and focus on solutions belonging to the space
\begin{align*}
     V^{\bf p, \bf m} := V^{\bf p, \bf m}_{q_*},
\end{align*} 
since the additional $L^{m+1}$-integrability is natural for establishing the time-continuity of solutions satisfying our type of boundary condition on $\partial \Omega \times (0,T)$. For the details we refer the reader to the application of Lemma \ref{lem:time-cont} in Subsection \ref{sec:boundary_val_time_cont_initial_data}. Note however, that since we actually show the existence of bounded solutions, the precise choice of the integrability exponent is not of critical importance.

Although the definition of our function space only requires each function $u^{m_j}$ to have a weak partial derivative in the $j$th coordinate direction, in the parameter range \eqref{cond:m_j-closeness} it turns out that the function $u^m$ has a gradient and that $\partial_j u^m \in L^{p_j}_{\textnormal{loc}}(\Omega_T)$. For the proof of this fact we refer to Section \ref{sec:existence_grad}. The existence of the gradient of $u^m$ serves also as a motivation for our definition of solutions to the Cauchy-Dirichlet problem, and in fact we will prove the existence of solutions for which $\partial_j u^m$ is in $L^{p_j}(\Omega_T)$.

We use the following notion of weak solutions of the equation. 
\begin{defin}[Weak solutions]\label{def:weaksol}
A function $u \in V^{\bf p, \bf m}$ is a solution of \eqref{eq:diffusion} if 
\begin{align}\label{eq:weak_form}
&\iint_{\Omega_T} \sum^N_{j=1} a_j(x,t,u)|\partial_j u^{m_j}|^{p_j-2}\partial_j u^{m_j} \partial_j \varphi -  u\partial_t \varphi\d x\d t = \iint_{\Omega_T} f \varphi \d x \d t,
\end{align}
for all $\varphi \in C^\infty_{\textnormal{o}}(\Omega_T)$. 
\end{defin}
We use the following notion of sub- and super-solutions.
\begin{defin}[Sub (super) - weak solutions]\label{def:subsupersol}
   A function $u \in V^{\bf p, \bf m}$ is a sub (super)-solution of \eqref{eq:diffusion} if 
\begin{align}\label{eq:weak_formsupersub}
&\iint_{\Omega_T} \sum^N_{j=1} a_j(x,t,u)|\partial_j u^{m_j}|^{p_j-2}\partial_j u^{m_j} \partial_j \varphi -  u\partial_t \varphi\d x\d t \leq (\geq) \iint_{\Omega_T} f \varphi \d x \d t,
\end{align}
for all non-negative $\varphi \in C^\infty_{\textnormal{o}}(\Omega_T)$. 
\end{defin}
Note that a function is a solution iff it is both a sub- and super-solution. We are now ready to introduce the notion of a weak solution to the Cauchy-Dirichlet problem.

\begin{defin}\label{def:prob-CD}
  Let $f$, $g$ and $u_0$ satisfy \eqref{cond:f}, \eqref{cond:g} and \eqref{cond:u_0}, respectively. We say that $u \in V^{\bf p, \bf m}$ is a solution to the Cauchy-Dirichlet problem
 \begin{align}\label{prob:C-D}
   \left\{
\begin{array}{ll}
\partial_t u   - \sum^N_{j=1} \partial_j\big(a_j(x,t,u)|\partial_j u^{m_j}|^{p_j-2}\partial_j u^{m_j}\big) = f & \quad \text{in } \Omega_T, 
\\[5pt]
 u(x,0) = u_0(x)  & \quad \text{in } \Omega,
 \\
 u = g  & \quad \text{on } \partial \Omega \times (0,T),
\end{array}
\right.
 \end{align}
if the following conditions hold:
\begin{enumerate}
 \item $u$ is a solution to the PDE in the sense of Definition \ref{def:weaksol}.
 \item\label{u(t)-convg-to-u_0} $u\in C([0,T];L^{m+1}(\Omega))$ and $u(0) = u_0$. 
 \item $u^m \in g^m + L^{\bf p}(0,T;\overline W^{1,\bf p}_{\textnormal{o}}(\Omega))$.
\end{enumerate}
\end{defin}
Since $m\geq 1$, the boundedness of $g$ and the fact that $g$ is in $L^{\bf p}(0,T; W^{1, {\bf p}}(\Omega))$ together imply that also $g^m$ is in $L^{\bf p}(0,T; W^{1, {\bf p}}(\Omega))$. Thus, our notion of initial values is rather similar to the approach taken in \cite{Ve}.

\subsection{Existence}
For the Cauchy-Dirichlet problem we prove the following existence result. 
 \begin{theo}\label{thm:existence}
  Let $\Omega$ be open and bounded. Suppose that $a_i$ satisfy \eqref{cond:unif_ellipt} and \eqref{cond:lip_cont}. Let $u_0$ and $f$ and $g$ satisfy \eqref{cond:u_0}, \eqref{cond:f} and \eqref{cond:g} respectively. 
  Then there exists a solution $u \in L^\infty(\Omega_T)$ to the Cauchy-Dirichlet problem \eqref{prob:C-D} in the sense of Definition \ref{def:prob-CD}.
 \end{theo}
The strategy of the proof can be summarized as follows:
\begin{enumerate}[label=\arabic*.]
    \item Introduce an approximating sequence of anisotropic PDEs without double nonlinearity for which the existence of solutions to the Cauchy-Dirichlet problem is known. 
    \item Conclude that the resulting sequence $(u_k)_{k = 1}^\infty$ of solutions satisfy a comparison principle, in order to obtain lower bounds on the whole domain $\Omega_T$ by comparison with suitable constant functions.
    \item  Prove a uniform upper bound for the solutions $u_k$ by means of a De Giorgi type iteration.
    \item Conclude using the upper and lower bounds that $u_k$ satisfies the original equation for large $k$.
    \item Deduce strong and pointwise convergence of the sequence $(u_k)_{k = 1}^\infty$ via the comparison principle, and weak convergence of spatial partial derivatives via energy estimates.
    \item Use Minty's trick to conclude pointwise convergence of the partial derivatives, in order to show that the limit function $u$ solves the original equation.
\end{enumerate}

\subsection{Comparison principle}
We prove a comparison principle for solutions of \eqref{eq:diffusion}. The techniques employed in the proof draw inspiration from the approach in \cite{BoeStru}, where a comparison principle for the model case for doubly nonlinear equations was proved under similar assumptions. However, since we deal with a more general class of fully nonlinear anisotropic evolution equations great care needs to be taken when adapting the strategy. Both the dependence of the coefficients on the solution itself and the presence of distinct coefficients $p_i$ and $m_i$ require some attention. Again, working in the parameter range \eqref{cond:m_j-closeness} makes it natural to consider solutions $u$ for which $u^m$ has a gradient.
\begin{defin}\label{def:boundary-value-ineq}
For functions $v^m$, $u^m$ in $L^{\bf p}(0,T; W^{1, {\bf p}}(\Omega))$ we say that $u\leq v$ on $\partial \Omega \times (0,T)$ if $(u^m - v^m)_+ \in L^{\bf p}(0,T; \overline{W}^{1, {\bf p}}_{\textnormal{o}}(\Omega))$. The definition can be extended to the case where $u$ or $v$ is replaced by a constant in the obvious way.
\end{defin}
\noindent Note that for functions $v^m$, $u^m$ in $L^{\bf p}(0,T; W^{1, {\bf p}}(\Omega))$ we have that $u\leq v$ in the sense of Definition \ref{def:boundary-value-ineq} if and only if $u^m = v^m + w + \psi$ for some $w$ in $L^{\bf p}(0,T; W^{1, {\bf p}}(\Omega))$ with $w\leq 0$ and some $\psi$ in $L^{\bf p}(0,T; \overline{W}^{1, {\bf p}}_{\textnormal{o}}(\Omega))$. We can now state the comparison principle.

\begin{theo}\label{thm:comparison_principle}
Let $u$ be a weak subsolution with right-hand side $f_u$ and let $v$ be a weak supersolution with right-hand side $f_v$ in the sense of Definition \ref{def:subsupersol}. Assume furthermore that $u,v \in C([0,T]; L^1(\Omega))$  and that $u^m, v^m \in L^{\bf p}(0,T; W^{1, {\bf p}}(\Omega))$. Suppose that $v\geq \varepsilon$ on $\partial \Omega \times (0,T)$ in the sense of Definition \ref{def:boundary-value-ineq} for some $\varepsilon > 0$. If $v \geq u$ on $\partial \Omega \times (0,T)$ in the sense of Definition \ref{def:boundary-value-ineq}, we have
\begin{align}\label{est:comp_principle_general}
\int_{\Omega} (u - v)_+ (x,t_2) \d x &\leq 
\iint_{\Omega \times[t_1,t_2]} \chi_{\{v<u \}\cup\{u=v=0\}}  (f_u\chi_{\{u>0\}} - f_v) \d x \d t
\\
&\quad+ \int_{\Omega} (u - v)_+ (x,t_1) \d x \notag
\end{align}
for every $0\leq t_1<t_2\leq T$. In particular, if $0\leq f_u \leq f_v$ and $u(0) \leq v(0)$ then
\begin{align}\label{u<v}
    u \leq v, \quad \textnormal{a.e. in $\Omega_T$.}
\end{align}
\end{theo}
 \subsection{Uniqueness} The comparison principle allows us to conclude the uniqueness of solutions to the Cauchy-Dirichlet problem in the cases where the boundary values are bounded from below by a positive constant. 
 \begin{theo}\label{thm:uniqueness}
 Let $f$, $u_0$ and $g$ satisfy the assumptions \eqref{cond:f}, \eqref{cond:u_0}  and \eqref{cond:g} respectively. Suppose furthermore that $g\geq \varepsilon$ for some $\varepsilon>0$. Then there exists exactly one solution $u$ in $C([0,T];L^{m+1}(\Omega)) \cap L^\infty(\Omega_T)$ with $u^m \in L^{\bf p}(0,T;W^{1, {\bf p}}(\Omega))$ to the Cauchy-Dirichlet problem \eqref{prob:C-D} in the sense of Definition \ref{def:prob-CD}. 
 \end{theo}
\begin{proof}{}
Existence, continuity in time into $L^{m + 1}(\Omega)$ and boundedness follow from Theorem \ref{thm:existence}. Uniqueness follows from Theorem \ref{thm:comparison_principle} since, for two solutions $u, v$ with the same initial and boundary conditions and right-hand side, we obtain \eqref{u<v} where the roles of $u$ and $v$ are interchangeable, and thus $u=v$.
\end{proof}

\begin{rem}
Due to the assumption in Theorem \ref{thm:comparison_principle} of continuity only into $L^1(\Omega)$, one can in fact formulate a somewhat stronger uniqueness result. Namely, there can only be one function $u$ in $V^{\bf p, \bf m}\cap C([0,T];L^1(\Omega))$ satisfying the conditions (1) and (3) of Definition \ref{def:prob-CD} and such that $u(0) = u_0$ as an element of $L^1(\Omega)$.
\end{rem}

\section{Preliminaries}\label{sec:preliminaries}
In this section we introduce some notation and various lemmas that will be used in the subsequent arguments. 
\subsection{Notation}
Throughout this work we denote
\begin{align*}
    A_j(x, t, u, \xi) := a_j(x, t, u) \lvert \xi_j \rvert^{p_j - 2}\xi_j, 
    & \quad (x, t, u, \xi) \in \Omega_T \times \mathbb{R} \times \mathbb{R}^N, 
    & \quad j\in\{1, \ldots, N\},
\end{align*}
and work with the vector valued map $A=(A_1,\dots, A_N)$, which allows us to express Equation \eqref{eq:diffusion} in divergence form:
\begin{align*}
    \partial_t u - \nabla \cdot A(x,t,u,\partial_1 u^{m_1}, \ldots, \partial_N u^{m_N}) = f.
\end{align*}
The only case where $A$ takes a different meaning is Subsection \ref{sec:comparison_p_laplace_type} which is self-contained and focuses on a different type of equations. For brevity, the vector $(x, t, u, \partial_1 u^{m_1}, \ldots, \partial_N u^{m_N})$ will occasionally be denoted by $(x, t, u, (\partial_j u^{m_j})_{j = 1}^N)$.

For $0 \leq \tau_1 < \tau_2 \leq T$ and $\delta \in \big(0, \tfrac12 (\tau_2 - \tau_1)\big)$ we introduce the trapezoidal function
\begin{align}\label{func:trapezoid}
\zeta_{\tau_1,\tau_2}^\delta(t)=
\begin{cases}
0 & \text{if } t<\tau_1, \\
\delta^{-1}(t-\tau_1) & \text{if } t\in [\tau_1, \tau_1+\delta], \\
1 & \text{if } t \in (\tau_1+\delta, \tau_2-\delta), \\
1 - \delta^{-1}(t-\tau_2+\delta) & \text{if } t \in [\tau_2-\delta, \tau_2], \\
0 & \text{if } t \geq \tau_2.
\end{cases}
\end{align}
 Given $\delta>0$ we define the Lipschitz function $H_\delta :\R\to \R$, 
\begin{align}\label{def:H_delta}
H_\delta(s) = \left\{
\begin{array}{ll}
0 & \quad \text{if } s \leq 0,
\\[5pt]
 \frac{s}{\delta}  &\quad \text{if } 0< s < \delta,
 \\
 1  & \quad \text{if }s \geq \delta.
\end{array}
\right.
\end{align}
We also denote
\begin{align}\label{def:G_delta}
    G_\delta (s) := \int_0^s H_\delta (\sigma) \d \sigma = \begin{cases}
    0 & \text{if } s \leq 0, \\
    \frac{s^2}{2 \delta} & \text{if } 0 < s < \delta, \\
    s - \frac{\delta}{2} & \text{if } \delta \leq s.
\end{cases}
\end{align}

Throughout this work, $c$ will denote a generic positive constant which may change value from line to line.

\subsection{Algebraic quantities and estimates}

\noindent The following result was proved in \cite{CiaVeVe}.
\begin{lem}\label{lem:elementary_real}
 Let $\gamma > 1$. For all $a, b \in \R$ we have 
 \begin{align}\label{est:exponent_inside}
  |a-b|^\gamma \leq c\big||a|^{\gamma-1} a - |b|^{\gamma-1}b\big|
 \end{align}
for a constant $c=c(\gamma)$. 
\end{lem}
\noindent Given $m > 0$ and $u, v \geq 0,$ we define 
\begin{align}
    \b[u, v] & := \tfrac{1}{m + 1} (u^{m + 1} - v^{m + 1}) - v^m(u - v) \label{eq:b_definition}  \\
    & = \tfrac{m}{m + 1} (v^{m + 1} - u ^{m + 1}) - u(v^m - u^m). \notag
\end{align}
Observe that this quantity is nonnegative. This can be seen directly by considering the relation between $\b$ and the convex differentiable function $u \mapsto \tfrac{1}{m+1}u^{m + 1}$, or as a consequence of \eqref{b:basic-est} below.
The following estimate for $\b$ which has been proved in \cite{BoeDuKoSc} will be useful in the proof of Lemma \ref{lem:energy_estimates_approx_probs}.
\begin{lem}\label{lemma:b_properties}
    Let $m \geq 1$ and $u, v \geq 0.$ Then, there exists a constant $c$ depending only on $m$ such that the following estimate holds true
    \begin{align}\label{b:basic-est}
        \tfrac{1}{c} \lvert v^\frac{m + 1}{2} - u^\frac{m+1}{2}\rvert^2 \leq \b [u, v] \leq c \lvert v^\frac{m + 1}{2} - u^\frac{m + 1}{2}\rvert^2.
    \end{align}
\end{lem}

We will employ the following Lemma of fast geometric convergence in the De Giorgi iteration argument. A proof can be found in \cite{Giu}.

\begin{lem}\label{lem:fast_geometric_convergence}
    Let $C > 0$, $b > 1$, $\delta > 0$ and let $(Y_j)_{j=0}^\infty$ be a sequence of positive real numbers such that
    \begin{align*}
        Y_{j+1} \leq C b^j Y_j^{1 + \delta}, \quad j\in\mathbb{N}.
    \end{align*}
     If 
     \begin{align*}
         Y_0 \leq C^{-\frac{1}{\delta}} b^{-\frac{1}{\delta^2}},
     \end{align*}
    then $Y_j$ converges to zero as $j \to \infty.$
\end{lem}

\subsection{Time mollifications}
In this work we will use two different types of mollifications in time for functions defined on space-time cylinders. For $v\in L^1(\Omega_T)$ and $h\in (0,T)$, we recall that the \textit{Steklov average} of $v$ is defined as 
\begin{align}\label{def:Steklov_av}
    [v]_h (x,t) := \frac{1}{h} \int_{t}^{t+h} v(x, s) \d s, \quad (x,t) \in \Omega \times(0, T-h),
\end{align}
and its reversed analogue is defined as
\begin{align}\label{def:Steklov_av_rev}
    [v]_{\bar h} (x,t) := \frac{1}{h} \int_{t-h}^{t} v(x, s) \d s, \quad (x,t) \in \Omega \times(h, T).
\end{align}
We also recall the definition of the \textit{exponential time mollification}, utilized previously in \cite{KiLi}, \cite{BoeDuMa} and many other later papers, which for $h\in (0,T)$ and $v\in L^1(\Omega_T)$ is defined as
\begin{align}
\label{def:moll}
\Ex{v}(x,t) :=\frac{1}{h}\int^t_0 e^\frac{s-t}{h}v(x,s)\d s.
\end{align}
In this case, the reversed analogue takes the form
\begin{align}\label{def:mollrev}
\Exx{v}(x,t) :=\frac{1}{h}\int^T_t e^\frac{t-s}{h}v(x,s)\d s.
\end{align}
For details regarding the properties of the exponential mollification we refer to \cite[Lemma 2.2]{KiLi}, \cite[Lemma 2.2]{BoeDuMa} and \cite[Lemma 2.9]{St}. Some properties of the Steklov average were mentioned already in \cite{LaSoUr} and \cite{DiBe}.
The properties of the above mollifications that we will use have been collected for convenience into the following two lemmas.
\begin{lem}\label{lem:steklov_mol_properties}
    Let $v \in L^1(\Omega_T)$ and $p \in [1, \infty).$ Then, the Steklov average defined in \eqref{def:Steklov_av} satisfies the following properties:
    \begin{itemize}
        \item [(i)] If $v\in L^p(\Omega_T)$ then $[v]_h \in L^p(\Omega_T)$,
        $$ \norm{[v]_h}_{L^p(\Omega_T)}\leq \norm{v}_{L^p(\Omega_T)},$$
        and $[v]_h\to v$ in $L^p(\Omega_T)$.
        \item [(ii)] If $\partial_j v \in L^p(\Omega_T)$ in the weak sense for some $j$ then also $\partial_j [v]_h \in L^p(\Omega_T)$, 
        $$\partial_j [v]_h = [\partial_j v]_h,$$
        and $\partial_j [v]_h \to \partial_j v$ in $L^p(\Omega_T)$ as $h \to 0.$
        \item [(iii)] If $v \in C([0,T]; L^p(\Omega)),$ then for any $t \in (0, T)$, $[v]_h(\cdot, t) \to v(\cdot, t)$ in $L^p(\Omega).$
        \item [(iv)] If $v \in L^p(\Omega_T)$ then $\partial_t[v]_h \in L^p(\Omega \times (0, T-h))$, $\partial_t [v]_{\bar h} \in L^p(\Omega \times (h, T))$ and 
        \begin{align*}
            \partial_t [v]_h(x,t) = \frac{v(x, t+h) - v(x,t)}{h}, \quad \partial_t [v]_{\bar h} = \frac{v(x, t) - v(x,t-h)}{h}.
        \end{align*}
    \end{itemize}
\end{lem}
\begin{lem}\label{lem:expmolproperties} 
Suppose that $v \in L^1(\Omega_T)$, and let $p\in[1,\infty)$. Then the mollification $\Ex{v}$ defined in \eqref{def:moll} has the following properties:
\begin{enumerate}[(i)]
\item\label{expmol1}
If $v\in L^p(\Omega_T)$ then $\Ex{v} \in L^p(\Omega_T)$,
$$
\norm{\Ex{v}}_{L^p(\Omega_T)}\leq \norm{v}_{L^p(\Omega_T)},
$$
and $\Ex{v}\to v$ in $L^p(\Omega_T)$. A similar estimate also holds with $\Exx{v}$ on the left-hand side.
\item\label{expmol2}
In the above situation, $\Ex{v}$ has a weak time derivative $\partial_t \Ex{v}$ on $\Omega_T$ given by
\begin{align*}
\partial_t \Ex{v}=\tfrac{1}{h}(v - \Ex{v}),
\end{align*}
whereas for $\Exx{v}$ we have
\begin{align*}
\partial_t \Exx{v}=\tfrac{1}{h}(\Exx{v} - v).
\end{align*}
\item\label{expmol3}
If $v$ has a weak partial derivative in space then so does $\Ex{v}$ and $\Exx{v}$ and 
\begin{align*}
 \partial_j \Ex{v} = \Ex{\partial_j v}, \hspace{5mm} \partial_j \Exx{v} = \Exx{\partial_j v}.
\end{align*}
\item\label{expmol4} If $v\in L^p(0,T;L^{p}(\Omega))$ then $\Ex{v}, \Exx{v} \in C([0,T];L^{p}(\Omega))$.
\end{enumerate}
\end{lem}
\noindent The following lemma will be used in the proofs of the comparison principles of Theorem \ref{thm:comparison_principle} and Theorem \ref{thm:comparison_non-doubly-nonlinear}. The proof can be found in [Lemma 3.1, \cite{BoeStru}].
\begin{lem}\label{lem:almost-heaviside-est}
  Let $\delta>0,$ $H_\delta$ and $G_\delta$ be defined as in \eqref{def:H_delta} and \eqref{def:G_delta}, and $f\in C(0,T; L^1(\Omega))$. Then, for any $0<h<T$ the following inequality holds 
  \begin{equation*}
      \partial_t [G_{\delta}(f)]_{\bar h} \leq \partial_t [f]_{\bar h} H_{\delta} (f) \quad \mbox{a.e. in $\Omega_T$.}
  \end{equation*}
\end{lem}
\subsection{Properties of functions and Sobolev Embeddings}
The following lemma shows that solutions to a Dirichlet-type problem,  with the vector field $A$ replaced by a general vector valued map $\mathcal{A}$ of sufficiently high integrability, also possess a time-continuity property.
\begin{lem}\label{lem:time-cont}
 Suppose that $u, g \in L^\infty(\Omega_T)$ are non-negative functions with
 \begin{align*}
  u^m &\in g^m + L^{\bf p}(0,T; \overline W^{1, {\bf p}}_{\textnormal{o}}(\Omega)).
  \end{align*}
  Suppose also that
  \begin{align*}
  \mathcal{A} &: \Omega_T \to \R^N, \quad \mathcal{A}_j \in L^{p_j'}(\Omega_T), \quad f \in L^{\bar p'}(\Omega_T),\quad  \partial_t g \in L^2(\Omega_T),
 \end{align*}
and that 
\begin{align}\label{weak:general}
&\iint_{\Omega_T} \mathcal{A}\cdot \nabla \varphi - u\partial_t \varphi \d x\d t = \iint_{\Omega_T} f \varphi \d x \d t,
\end{align}
holds for all $\varphi \in C^\infty_{\textnormal{o}}(\Omega_T)$. Then $u$ has a representative in $C([0,T];L^{m+1}(\Omega))$. 
\end{lem}
\begin{proof}{}
Following the approach of \cite[Lemma 3.10]{Ve} we can verify that
\begin{align}\label{eq:amost-timecont}
    \iint_{\Omega_T} \b[u, (w+g^m)^\frac1m] \zeta'(t) \d x \d t &= \iint_{\Omega_T} \mathcal{A}\cdot \nabla (u^m - w - g^m)\zeta + f (w+ g^m - u^m)\zeta \d x \d t \notag
    \\
     &\quad + \iint_{\Omega_T}(u - (w + g^m)^\frac1m)\partial_t (w+g^m) \zeta \d x \d t,
\end{align}
with the choice $w=\Exx{u^m - g^m}$ and $\zeta \in C^\infty_{\textnormal{o}}((0,T); [0,\infty))$. Here we denote $s^\alpha = |s|^{\alpha - 1}s$ for any $\alpha > 0$ and $s \in \R$, so that the above expression makes sense regardless of the signs of the functions. The only nontrival part of the argument is the potentially weaker integrability assumption on $f$ in our case. However, since both $u$ and $g$ are bounded, this poses no problems.
Utilizing \eqref{eq:amost-timecont} as in the proof of \cite[Lemma 3.11]{Ve} we therefore end up with
\begin{align}\label{lim:basically_continuity}
    \lim_{j\to \infty} \sup_{\tau \in [0,T]\setminus N} \int_\Omega \big|u - \big([\![u^m - g^m]\!]_{\bar h_j} + g^m\big)^\frac{1}{m} \big|^{m+1} (x,\tau)\d x = 0,
\end{align}
for a suitable sequence $h_j\to 0$ and some set $N\subset[0,T]$ of measure zero. The integrability property of $\partial_t g$ and the fact that $g$ is bounded show that $g \in C([0,T];L^q(\Omega))$ for any $q \in [1,\infty)$. Similarly, Property (iv) of Lemma \ref{lem:expmolproperties} and the boundedness of $g^m$ and $u^m$ show that $[\![u^m - g^m]\!]_{\bar h_j} \in C([0,T];L^q(\Omega))$ for any $q \in [1,\infty)$. Together with the essentially uniform convergence in the $L^{m+1}(\Omega)$-norm established in \eqref{lim:basically_continuity}, we conclude that $u$ has a representative in $C([0,T];L^{m+1}(\Omega))$.
\end{proof}

The following embedding theorem is a version of \cite[Proposition 2.3]{DuMoVe} with $\alpha_i \equiv 1$, (appearing already in \cite{Tro}) adapted to the setting of this paper. Here we are able to consider a general bounded domain $\Omega$ as opposed to the rectangular sets considered in \cite{DuMoVe} by working with the smaller space $\overline{W}^{1, \mathbf{p}}_\textnormal{o}(\Omega)$ rather than $W^{1, \mathbf{p}}_\textnormal{o}(\Omega)$ which appears in \cite{DuMoVe}. In fact, our result follows from approximating any $u \in \overline{W}^{1, \mathbf{p}}_\textnormal{o}(\Omega)$ by smooth compactly supported functions which can trivially be extended to a rectangular domain. Applying \cite[Proposition 2.3]{DuMoVe} to the approximating sequence and passing to the limit we obtain \eqref{ineq:Sobolev_Troisi_product}.
\begin{lem}\label{lem:troisiii}
    Let $\Omega \subseteq \mathbb{R}^N$ be a bounded domain and assume that $\bar p < N$. There exists a constant $C(N, \mathbf{p}) > 0$ such that for all $u \in \overline{W}^{1, \mathbf{p}}_\textnormal{o}(\Omega)$,
    \begin{align}\label{ineq:Sobolev_Troisi_product}
        \lVert u \rVert_{L^{\bar p^*}(\Omega)} \leq C\prod_{j = 1}^N \lVert \partial_j u \rVert^\frac{1}{N}_{L^{p_j}(\Omega)}.
    \end{align}
\end{lem}

\begin{rem}\label{rem:on_Troisis_inequality}
    A careful analysis of the proof of \cite[Proposition 2.3]{DuMoVe} in the case $\alpha_i \equiv 1$ shows that the constant $C$ in \eqref{ineq:Sobolev_Troisi_product} can in fact be chosen as
\begin{align}\label{ineq:p_bar_constant}
    C = \Big[\prod_{j=1}^N (1 + \bar p^*/{p_j'})\Big]^\frac1N \leq 1 + \bar p^*.
\end{align}
Therefore, by combining H\"older's inequality, \eqref{ineq:Sobolev_Troisi_product} and Young's inequality we have in the case $\bar p < N$:
    \begin{align*}
        \lVert u \rVert_{L^{\bar p}(\Omega)}^{\bar p} \leq |\Omega|^{1-\frac{\bar p}{\bar p^*}}\Big( \int_\Omega \lvert u\rvert^{\bar p^*} \d x \Big)^\frac{\bar p}{\bar p^*} &\leq C(\bar p, N)|\Omega|^\frac{\bar p}{N} \prod_{j = 1}^N \Big[\int_\Omega \lvert \partial_j u \rvert^{p_j} \d x\Big]^\frac{\bar p}{N p_j}
        \\ 
        & \leq C(\bar p, N)|\Omega|^\frac{\bar p}{N} \sum_{j = 1}^N \int_\Omega \lvert \partial_j u \rvert^{p_j} \d x.
    \end{align*}
    Also in the range $\bar p \geq N$ we have a similar estimate. To see this, choose exponents $1 < q_j \leq p_j$ so that $\bar q = \tfrac{N\bar p}{N + \bar p}.$ Then $\bar q < \bar p$, $\bar q < N$ and $\bar q^* = \bar p.$ Hence, Lemma  \ref{lem:troisiii} and the observation made in \eqref{ineq:p_bar_constant} combined with H\"older's inequality and Young's inequality show that
    \begin{align*}
        \int_\Omega \lvert u \rvert^{\bar p} \d x 
        = \int_\Omega \abs{u}^{\bar q^*} \d x \leq (1+ \bar q^*)^{\bar q^*} \prod_{j = 1}^N \lVert \partial_j u \rVert_{L^{q_j}(\Omega)}^{{\bar q^*/N}} 
        &
        \leq (1+\bar p)^{\bar p}|\Omega|^\frac{\bar p}{N} \prod_{j =1}^N \lVert \partial_j u \rVert_{L^{p_j}(\Omega)}^{\bar p/N}
        \\
        &\leq C(\bar p) |\Omega|^\frac{\bar p}{N} \sum_{j = 1}^N \lVert \partial_j u \rVert_{L^{p_j}(\Omega)}^{p_j}. \notag
    \end{align*} 
    Thus, the functions $u \in \overline{W}^{1, \mathbf{p}}_\textnormal{o}(\Omega)$ satisfy an estimate of the form
\begin{align}\label{eq:Sobolev_Troisi_Inequality}
        \int_\Omega \lvert u \rvert^{\bar p} \d x \leq C(\bar p, N)|\Omega|^\frac{\bar p}{N} \sum_{j = 1}^N \int_\Omega \lvert \partial_j u \rvert^{p_j} \d x,
    \end{align}
    for any choice of exponents $p_j > 1$. 
\end{rem}

\subsection{Comparison principle for anisotropic parabolic $p$-Laplace equations}\label{sec:comparison_p_laplace_type}
In this subsection we prove a comparison principle for anisotropic equations \textit{without} double nonlinearity. The model case has been addressed already in \cite{Cia} (see also  \cite{CiaMoVe}). As we we did not find any source containing the proof in the case of general structure conditions, we present it here for completeness. We will use this comparison principle in the proof of the existence of solutions in Section \ref{sec:existence} in order to obtain a uniform lower bound for the solutions to \eqref{eq:approximative_problem}.

In this subsection we consider equations of the form
\begin{equation}\label{eq:basic_genera_anisotropic}
    \partial_t u - \nabla \cdot A(x, t, u, \nabla u) = f,
\end{equation}
and solutions to the Cauchy-Dirichlet problem 
\begin{center}
\begin{equation}\label{eq:basic_genera_anisotropic_problem}
    \begin{cases}
    \partial_t u - \nabla \cdot A(x, t, u, \nabla u) = f & \text{in} \;\; \Omega_T, \\
    u = g & \text{on} \;\; \partial\Omega \times (0,T), \\
    u(\cdot, 0) = u_0 & \text{in} \;\; \Omega \times \{0\},
    \end{cases}
\end{equation}
\end{center}
where $A = A(x,t,u,\xi)$ is a Carath\'eodory vector field satisfying the following structure conditions:
\begin{align}\label{cond:growth_A}
    \lvert A_j (x, t, u, \xi) \rvert  &\leq  \Lambda \Big( \sum_{k = 1}^N \lvert \xi_k \rvert^{p_k} + a(x, t) \Big)^{\frac{p_j - 1}{p_j}},
\\
\label{cond:coerc_A}
    A (x, t, u, \xi) \cdot \xi &\geq \Lambda^{-1} \sum_{j = 1}^N \lvert \xi_j \rvert^{p_j} - b(x, t),
\end{align}
where $a, b$ are non-negative functions in $L^1(\Omega_T)$ and $\Lambda > 0$. In order to prove a comparison principle we also need the following weak monotonicity condition:
\begin{align}\label{cond:monotonicity_A}
    ( A(x, t, u, \xi)-A(x, t, u, \eta) ) \cdot (\xi - \eta) \geq 0.
\end{align}
We also need to assume that each component $A_j$ is Lipschitz continuous in the $u$-variable in the sense that there exists a constant $L > 0$ such that for every $x, t, \xi$ and every $u, v,$ 
\begin{align}\label{cond:lip_cont_A}
    \lvert A_j(x, t, u, \xi) - A_j(x, t, v, \xi) \rvert \leq L \lvert u - v \rvert \Big( c(x, t) + \sum_{k = 1}^N 
 \lvert \xi_k \rvert^{p_k} \Big)^{\frac{p_j - 1}{p_j}},
\end{align}
where $c$ is an element of $L^1(\Omega_T)$.
In this subsection we consider source terms $f$ in $L^2(\Omega_T)$, boundary values $g \in L^\mathbf{p}( 0,T ; W^{1, \mathbf{p}}(\Omega))$ and initial values $u_0$ in $L^2(\Omega)$. 
For the sake of clarity we explicitly state the definition of weak solutions to \eqref{eq:basic_genera_anisotropic}.

\begin{defin}\label{def:weak_sol_coef_usual_anisotropic}
    A function $u \in L^\mathbf{p}( 0,T ; W^{1, \mathbf{p}}(\Omega)) \cap L^2(\Omega_T)$ is a weak solution to equation \eqref{eq:basic_genera_anisotropic} if and only if, for all $\varphi \in C^\infty_{\textnormal{o}}(\Omega_T)$,
    \begin{align*}\label{weak_solution_eq_A}
        \iint_{\Omega_T} A(x, t, u, \nabla u) \cdot \nabla\varphi \; \d x \d t - \iint_{\Omega_T} u\partial_t \varphi \; \d x \d t = \iint_{\Omega_T} f \varphi \; \; \d x \d t.
    \end{align*}
\end{defin}
As a consequence of the definition we have the following mollified weak formulation involving Steklov averages.
\begin{lem}
    If $u$ is a weak solution to equation \eqref{eq:basic_genera_anisotropic}, then for every $t_1, t_2 \in (0, T - h)$ such that $t_1 < t_2$ and any $\varphi \in L^{\bf p}(0,T; \overline W^{1, {\bf p}}_{\textnormal{o}}(\Omega)) \cap L^2(\Omega_T),$
    \begin{equation}\label{eq:weak_mollified_usual_anisotropic_Steklov}
            \int_{t_1}^{t_2}\int_\Omega \partial_t [u]_{h}  \varphi \; \d x \d t + [ A(\cdot, \cdot, u, \nabla u)]_{h} \cdot \nabla\varphi \; \d x \d t = \int_{t_1}^{t_2} \int_\Omega [f]_{h} \varphi \; \d x \d t.
        \end{equation}
    For every $t_1, t_2 \in (h, T)$ such that $t_1 < t_2$ and any $\varphi \in L^{\bf p}(0,T; \overline W^{1, {\bf p}}_{\textnormal{o}}(\Omega)) \cap L^2(\Omega_T),$
        \begin{equation}\label{eq:weak_mollified_usual_anisotropic}
            \int_{t_1}^{t_2}\int_\Omega \partial_t [u]_{\bar{h}} \varphi + [ A(\cdot, \cdot, u, \nabla u)]_{\bar{h}} \cdot \nabla\varphi \; \d x \d t = \int_{t_1}^{t_2} \int_\Omega [f]_{\bar{h}} \varphi \; \d x \d t.
        \end{equation}
\end{lem}

We consider the following standard definition of solution to problem \eqref{eq:basic_genera_anisotropic_problem}.
\begin{defin}\label{def:weak_sol_coef_usual_anisotropic_problem}
    Let $u_0 \in L^2(\Omega)$ and $f \in L^2(\Omega_T)$. A function $u \in L^\mathbf{p}( 0,T ; W^{1, \mathbf{p}}(\Omega)) \cap C( [ 0, T ]; L^2(\Omega))$ is a solution to the Cauchy-Dirichlet problem \eqref{eq:basic_genera_anisotropic_problem} if and only if 
    \begin{itemize}
        \item [(1)] $u$ is a solution to equation \eqref{eq:basic_genera_anisotropic} in the sense of Definition \ref{def:weak_sol_coef_usual_anisotropic},
        \item [(2)] $u(t) \to u_0$ in $L^2(\Omega)$ as $t \to 0,$
        \item [(3)] $u \in g + L^{\bf p}(0,T; \overline W^{1, {\bf p}}_{\textnormal{o}}(\Omega)).$
    \end{itemize}
\end{defin}
We are now in conditions to state and prove the comparison principle for problem \eqref{eq:basic_genera_anisotropic_problem}. This result does not follow directly from the comparison principle in \cite{Ve}, since here we allow the vector field $A$ to be time-dependent.
\begin{theo}\label{thm:comparison_non-doubly-nonlinear}
    Let $A$ be a Carathéodory vector field that satisfies \eqref{cond:growth_A},  \eqref{cond:coerc_A} and \eqref{cond:lip_cont_A}. Let u be a solution to problem \eqref{eq:basic_genera_anisotropic_problem} with initial data $u_0$ and right-hand side $f_u \in L^2(\Omega_T),$ and boundary data $g_u \in L^{\bf p}(0,T;  W^{1, {\bf p}}(\Omega)).$ Let v be a solution to problem \eqref{eq:basic_genera_anisotropic_problem} with initial data $v_0 \leq u_0$ and right-hand side $f_v \leq f_u,$ and boundary data $g_v \in L^{\bf p}(0,T;  W^{1, {\bf p}}(\Omega))$ such that $g_v \leq g_u$ in $\Omega_T.$ Then $v \leq u$ in $\Omega_T.$
\end{theo}
\begin{proof}{}
    We recall the function $H_\delta$ introduced earlier in \eqref{def:H_delta} and use $H_\delta( v - u )$ as test function in the weak formulations \eqref{eq:weak_mollified_usual_anisotropic} satisfied by $u$ and $v$ to obtain, for every $t_1, t_2 \in (h, T),$
    \begin{align}\label{eq:first_two_comparison_principle_1}
        & \int_{t_1}^{t_2} \int_\Omega  H_\delta( v - u ) \partial_t [ v- u ]_{\bar{h}} \d x \d s  \\
        & + \int_{t_1}^{t_2} \int_\Omega \left[ A(\cdot, \cdot, v, \nabla v) - A(\cdot, \cdot, u, \nabla u)\right]_{\bar{h}} (x, s) \cdot \nabla (v-u) \tfrac1\delta \chi_{\{0<v-u<\delta\}} \; \d x \d s \notag \\
        \notag & = \int_{t_1}^{t_2} \int_\Omega [f_v - f_u]_{\bar{h}} H_\delta( v - u ) \; \d x \d s. 
    \end{align}
    Given that $v - u$ can be written as  $w + (g_v - g_u),$ for $w \in L^{\bf p}(0,T; \overline W^{1, {\bf p}}_{\textnormal{o}}(\Omega))$ and $g_v - g_u \leq 0 $ in $\Omega_T,$ an approximation argument justifies the choice of test function. 
    Using Lemma \ref{lem:almost-heaviside-est} with $f=v-u$, we can estimate the first integrand on the left-hand side of \eqref{eq:first_two_comparison_principle_1} as follows:
    \begin{align}
        \partial_t [G_\delta ( v- u )]_{\bar{h}}(x, t)         & \leq  H_\delta ( v - u )(x, t) \partial_t[ v - u ]_{\bar{h}} (x, t).\notag
    \end{align}
    We can therefore estimate the first integral on the left-hand side of \eqref{eq:first_two_comparison_principle_1} as
\begin{align*}
    \int_{t_1}^{t_2} \int_\Omega  H_\delta( v - u ) \partial_t [ v- u ]_{\bar{h}} \d x \d s &\geq \int_{t_1}^{t_2} \int_\Omega  \partial_t [G_\delta ( v- u )]_{\bar{h}} \d x \d t 
    \\
    &= \int_\Omega [G_\delta ( v- u )]_{\bar{h}}(x, t_2)\d x - \int_\Omega [G_\delta ( v- u )]_{\bar{h}}(x, t_1)\d x
    \\
    \xrightarrow[h\to 0]{}  & \quad \int_\Omega G_\delta ( v- u )(x, t_2)\d x - \int_\Omega G_\delta ( v- u )(x, t_1)\d x.
\end{align*}
Passing to the limit $h\to 0$ in the other terms of \eqref{eq:first_two_comparison_principle_1} poses no problem and we end up with
    \begin{align}\label{aimsir_go_dona}
    \int_\Omega G_\delta ( v- u )(x, t_2)\d x - \int_\Omega G_\delta ( v- u )(x, t_1)\d x
    \\
    \notag    + \int_{t_1}^{t_2} \int_\Omega \big(A(x,t, v, \nabla v) - A(x,t, u, \nabla u)\big) \cdot \nabla (v-u) \tfrac1\delta \chi_{\{0<v-u<\delta\}} \d x \d t
    \\
    \notag    \leq \int_{t_1}^{t_2} \int_\Omega (f_v - f_u )H_\delta( v - u )  \d x \d t.
    \end{align}
The integrand on the second row of \eqref{aimsir_go_dona} can be estimated using the monotonicity condition \eqref{cond:monotonicity_A} and the Lipschitz continuity \eqref{cond:lip_cont_A} as
\begin{align*}
    &\big(A(x,t, v, \nabla v) - A(x,t, u, \nabla u)\big) \cdot \nabla (v-u) \tfrac1\delta \chi_{\{0<v-u<\delta\}} 
    \\
    &\quad = \big(A(x,t, v, \nabla v) - A(x,t, v, \nabla u)\big) \cdot \nabla (v-u) \tfrac1\delta \chi_{\{0<v-u<\delta\}}
    \\
    &\quad \quad + \big(A(x,t, v, \nabla u) - A(x,t, u, \nabla u) \big) \cdot \nabla (v-u) \tfrac1\delta \chi_{\{0<v-u<\delta\}}
    \\
    & \quad \geq \big(A(x,t, v, \nabla u) - A(x,t, u, \nabla u) \big) \cdot \nabla (v-u) \tfrac1\delta \chi_{\{0<v-u<\delta\}}
    \\
    & \quad \geq -L\sum^N_{j=1}\Big(c(x,t)+ \sum^N_{k=1}|\partial_k u|^{p_k}\Big)^\frac{p_j-1}{p_j}|\partial_j v - \partial_j u|\chi_{\{0<v-u<\delta\}}.
\end{align*}
Combining the last estimate with \eqref{aimsir_go_dona} and noting that the integral on the right-hand side of \eqref{aimsir_go_dona} is non-positive since $f_u\geq f_v$, we end up with
\begin{align}
    &\int_\Omega G_\delta ( v- u )(x, t_2)\d x - \int_\Omega G_\delta ( v- u )(x, t_1)\d x
    \\
    \notag  & \quad  \leq L \sum^N_{j=1}\int^{t_2}_{t_1} \int_\Omega \Big(c(x,t)+ \sum^N_{k=1}|\partial_k u|^{p_k}\Big)^\frac{p_j-1}{p_j}|\partial_j v - \partial_j u|\chi_{\{0<v-u<\delta\}}\d x \d t.
    \end{align}
Passing to the limit $\delta \to 0$ the integral on the right-hand side vanishes due to the Dominated Convergence Theorem, and we end up with 
    \begin{align}\label{almost-comp-principle}
    \int_\Omega  ( v - u )_+(x, t_2) \d x \leq \int_\Omega ( v - u )_+(x, t_1). 
    \end{align}
    Finally, passing to the limit $t_1\to 0$ in \eqref{almost-comp-principle} is possible due to the time-continuity of $v-u$ and we end up with
    \begin{align*}
    \int_\Omega  ( v - u )_+(x, t_2) \d x \leq \int_\Omega ( v - u )_+(x, 0) = 0, 
    \end{align*}
    where the last equality follows from the fact that $v_0 \leq u_0$.
    This concludes the proof that $u \geq v$ in $\Omega_T.$
\end{proof}

\section{Existence of the gradient of $u^m$}\label{sec:existence_grad}
In this section we prove that for solutions $u$ in $V^{\bf p, \bf m}$ to equation \eqref{eq:diffusion} in the parameter range \eqref{cond:m_j-closeness}, the function $u^m$ in fact has a weak gradient satisfying $\partial_j u^m \in L^{p_j}_{\textnormal{loc}}(\Omega_T)$ for all $j \in \{1,\dots,N\}$. The argument is subtle as we do not assume any a priori time continuity for our solutions. In order to prove the existence of the gradient of $u^m$ we need a weak formulation involving the exponential time mollification. However, this formulation requires some notion of time continuity, and the standard method for proving the time continuity into an $L^p$-space, used for various doubly nonlinear equations for example in \cite{SiVe,St,CiaVeVe}, would in our setting require that we already knew that $u^m$ has a gradient. In order to avoid a circular reasoning, we will instead use the following lemma which establishes the continuity of $u$ on the time interval into a dual space.
\begin{lem}\label{lem:cont_into_dual_space1}
Let $V$ denote the closure of $C^\infty_{\textnormal{o}}(\Omega)$ in the space 
\begin{align*}
    \{v\in L^{(m+1)/m}(\Omega)\,|\, \partial_j v \in L^{p_j}(\Omega)\}.
\end{align*}
Let $u \in V^{\bf p, \bf m}$ be a weak solution to \eqref{eq:diffusion} in the sense of Definition \ref{def:weaksol}. Then $u$ has a representative in $C([0,T], V')$. Especially, for all $\tau_1, \tau_2 \in [0,T]$, 
\begin{align}\label{spatrizio}
    \norm{u(\tau_1) - u(\tau_2)}_{V'} \leq C |\tau_1 - \tau_2|^\gamma,
\end{align}
where $\gamma = (\max\{\max{\{p_i\}}, (m+1)/m\})^{-1}$. Here $u(t)$ is regarded as an element of the dual of $V$ by the standard embedding 
\begin{align}\label{dual_element}
    \langle u(t), v\rangle := \int_\Omega u(t) v\d x.
\end{align}
\end{lem}
\begin{proof}{}
We will show the estimate \eqref{spatrizio} for $\tau_1, \tau_2 \in [0,3T/4]$. A similar argument can be used to treat the partially overlapping interval $[T/4,T]$ which completes the proof. It is sufficient to show the estimate \eqref{spatrizio} for $|\tau_1 - \tau_2 |\leq 1$. Namely, this proves the continuity which implies that the left-hand side of \eqref{spatrizio} remains bounded for all $\tau_1, \tau_2$. Thus, in the case $|\tau_1 - \tau_2| \geq 1$ the estimate will hold by choosing $C$ sufficiently large. We thus let $0\leq \tau_1 < \tau_2 \leq 3T/4$ and $|\tau_1 - \tau_2|\leq 1$. For $\delta \in (0, T/4)$ we can consider the Steklov averages $[u]_\delta$, which are defined on $\Omega\times[0,3T/4]$. As in \eqref{dual_element}, also $[u]_\delta(t)$ can be interpreted as an element of $V'$. 

Taking $\varphi\in C^\infty_{\textnormal{o}}(\Omega)$
and considering the weak formulation with the test function 
\newline $\varphi(x) \eta^\delta_{\tau_1, \tau_2}(t)$, where $\eta^\delta_{\tau_1, \tau_2} = \zeta^\delta_{\tau_1,\tau_2 + \delta}$, we end up with
\begin{align}
   \langle [u]_\delta(\tau_1) - [u]_\delta(\tau_2), \varphi \rangle &=  \frac1\delta \int^{\tau_1+\delta}_{\tau_1}\int_\Omega u \varphi \d x \d t - \frac1\delta\int^{\tau_2 + \delta}_{\tau_2}\int_\Omega u\varphi\d x \d t \\&\notag= \iint_{\Omega_T} \sum^N_{j=1}a_j(x,t,u) |\partial_j u^{m_j}|^{p_j-2}\partial_j u^{m_j}\partial_j\varphi {\eta}^\delta_{\tau_1,\tau_2}(t)\d x \d t
    \\
    &\notag  \quad + \iint_{\Omega_T}f(x,t)\eta^\delta_{\tau_1,\tau_2}(t)\varphi(x)\d x \d t.
\end{align}
Using H\"older's inequality, the fact that $0 \leq \eta^\delta_{\tau_1,\tau_2} \leq \chi_{[\tau_1, \tau_2+ \delta]} $ and recalling that $\varphi$ is independent of time, we have
\begin{align*}
    |\langle [u]_\delta &(\tau_1) - [u]_\delta(\tau_2), \varphi \rangle| 
    \leq c \sum^N_{j=1} \Big(\int^{\tau_2+\delta}_{\tau_1} \int_\Omega |\partial_j u^{m_j}|^{p_j}\d x \d t \Big)^\frac{p_j-1}{p_j}\norm{\partial_j \varphi}_{L^{p_j}(\Omega)} |\tau_2 +\delta - \tau_1|^\frac{1}{p_j}
    \\
    &\qquad + \Big(\int^{\tau_2+\delta}_{\tau_1}  \int_\Omega |f|^{m+1}\d x \d t\Big)^\frac{1}{m+1}\norm{\varphi}_{L^{(m+1)/m}(\Omega)}|\tau_2 +\delta - \tau_1|^\frac{m}{m+1}
    \\
    &\leq \Big(c \sum^N_{j=1} \norm{\partial_j u^{m_j}}^{p_j-1}_{L^{p_j}(\Omega_T)} \norm{\partial_j \varphi}_{L^{p_j}(\Omega)} + \norm{f}_{L^{m+1}(\Omega_T)}\norm{\varphi}_{L^{(m+1)/m}(\Omega)}\Big)(|\tau_1 - \tau_2|^\gamma + \delta^\gamma)
    \\
    &\leq \Big(c \sum^N_{j=1} \norm{\partial_j u^{m_j}}^{p_j-1}_{L^{p_j}(\Omega_T)} + \norm{f}_{L^{m+1}(\Omega_T)}\Big)\norm{\varphi}_V (|\tau_1 - \tau_2|^\gamma + \delta^\gamma).
\end{align*}
The expression in the brackets of the last line is a constant independent of $\delta$. Moreover, by density, we may replace $\varphi$ by any element of $V$. This confirms that
\begin{align*}
    \norm{[u]_\delta(\tau_1) - [u]_\delta(\tau_2)}_{V'} \leq C(|\tau_1 - \tau_2|^\gamma +\delta^\gamma),
\end{align*}
for a constant $C$ independent of $\delta$. By the convergence properties of the Steklov average, we can conclude that $[u]_\delta$ converges to $u$ also in $L^q(0,T;V')$, and hence that for some sequence $\delta_j \to 0$ we have that $[u]_{\delta_j}$ converges pointwise to some limit function $w:[0,T]\setminus N \to V'$ where $N$ is a set of measure zero. By the pointwise convergence, $w$ satisfies
\begin{align}\label{hoeldcont}
    \norm{w(\tau_1) - w(\tau_2)}_{V'} \leq C|\tau_1 - \tau_2|^\gamma, 
\end{align}
for all $\tau_1, \tau_2 \in [0,T]\setminus N$. Since $V'$ is complete, $w$ has a unique extension satisfying \eqref{hoeldcont} for all $\tau_1, \tau_2$ in $[0,T]$. By the pointwise a.e. convergence, $w$ is a representative of $u$ as a map into $V'$. 
\end{proof}

Using the time continuity we can obtain the following mollified weak formulation.
\begin{lem}\label{lem:mollified_weak_forM_*ased_on_cont_into_dual_space}
For all $\varphi \in C^\infty_{\textnormal{o}}(\Omega_T)$ we have 
\begin{align}\label{mollified_weak_forM_*ased_on_cont_into_dual_space}
\iint_{\Omega_T}\sum^N_{j=1}\Ex{a(x,\cdot,u)|\partial_j u^{m_j}|^{p_j-2}\partial_j u^{m_j}}\partial_j \varphi + \partial_t \Ex{u} \varphi \d x \d t 
&= \langle u(0), \Exx{\varphi}(\cdot, 0)\rangle 
\\ 
\notag &\quad + \iint_{\Omega_T} \Ex{f} \varphi \d x \d t,
\end{align}
where $u(0)$ denotes the value of $u$ at time zero, when regarded as a map into $V'$. The brackets again denote dual pairing of $V'$ and $V$.
\end{lem}
\begin{proof}{}
We use the weak formulation \eqref{eq:weak_form} with the test function $\Exx{\varphi} H_\delta(t)$ where $H_\delta(t)$ is defined as in \eqref{def:H_delta}.

We can write the parabolic term as
\begin{align*}
    \iint_{\Omega_T}u \partial_t \big(\Exx{\varphi}  H_\delta(t) \big) \d x \d t = \iint_{\Omega_T}u \partial_t \Exx{\varphi}  H_\delta(t) \d x \d t + \frac1\delta \int^\delta_0 \int_\Omega u \Exx{\varphi} \d x \d t.
\end{align*}
In the first term on the right-hand side passing to the limit $\delta\to 0$ is straight-forward. In the second term we can add and remove a term and use the duality pairing between $V$ and $V'$ to obtain
\begin{align*}
    \frac1\delta \int^\delta_0 \int_\Omega u \Exx{\varphi}(x,t) \d x \d t &= \frac1\delta \int^\delta_0 \int_\Omega u (\Exx{\varphi}(x,t) - \Exx{\varphi}(x,0)) \d x \d t \\ &\quad+ \frac1\delta\int^\delta_0 \langle u(t), \Exx{\varphi}(\cdot,0)\rangle \d t.
\end{align*}
The first integral on the right-hand side vanishes in the limit $\delta\to 0$ due to the Lipschitz continuity of $\Exx{\varphi}$ and the Dominated Convergence Theorem. For the second term on the right-hand side, the time continuity established in Lemma \ref{lem:cont_into_dual_space1} implies that
\begin{align*}
    \frac1\delta \int^\delta_0 \langle u(t), \Exx{\varphi}(\cdot,0)\rangle \d t \xrightarrow[\delta\to 0]{} \langle u(0), \Exx{\varphi}(\cdot,0)\rangle.
\end{align*}
Thus, passing to the limit $\delta\to 0$ in the weak formulation we have
\begin{align*}
    \iint_{\Omega_T} \sum^N_{j=1} a_j(x,t,u)|\partial_j u^{m_j}|^{p_j-2}\partial_j u^{m_j}\Exx{\partial_j \varphi} - u \partial_t \Exx{\varphi} \d x \d t &= \langle u(0),\Exx{\varphi}(\cdot,0)\rangle \\
    & \quad + \iint_{\Omega_T}f \Exx{\varphi} \d x \d t,
\end{align*}
where we have also used Property $(iii)$ of Lemma \ref{lem:expmolproperties} for the exponential time-mollification. Moving the mollification from $\varphi$ over to the elliptic term and the right-hand side is standard, and to treat the term involving the time derivative we use Property $(ii)$ of Lemma \ref{lem:expmolproperties}, ending up with \eqref{mollified_weak_forM_*ased_on_cont_into_dual_space}.
\end{proof}
We are now ready to prove the existence of the gradient of $u^m$.
\begin{prop}\label{prop:existence of the gradient}
Let $u \in V^{\bf p, \bf m}$ be a weak solution to \eqref{eq:diffusion} in the sense of Definition \ref{def:weaksol} and suppose that the exponents $m_j$ satisfy \eqref{cond:m_j-closeness}. Then $u^m$ has a weak gradient and $\partial_j u^m \in L^{p_j}_{\textnormal{loc}}(\Omega_T)$ for all $j \in \{1,\dots, N\}$. 
\end{prop}
\begin{proof}{}
We denote $u_\delta := \max\{\delta, u\}$ and use the mollified formulation \eqref{mollified_weak_forM_*ased_on_cont_into_dual_space} with the test function 
\begin{align*}
    \varphi = (u_\delta)^\varepsilon \psi,
\end{align*}
where $\psi \in C^\infty_{\textnormal{o}}(\Omega_T;[0,1])$ and $\varepsilon \in (0,1)$ is a parameter which will be taken sufficiently small. This choice of the test function is justified since we can write
\begin{align*}
    (u_\delta)^\varepsilon = \max\{\delta^{m_j}, u^{m_j}\}^\frac{\varepsilon}{m_j},
\end{align*}
and $s \mapsto \max\{\delta^{m_j}, s\}^\frac{\varepsilon}{m_j}$ is a Lipschitz piecewise $C^1$-function, which allows us to use the Chain Rule for Sobolev functions. By a similar argument, we can deduce that $(u_\delta)^m$ also has a gradient. Denoting $(\Ex{u})_\delta := \max\{\delta, \Ex{u}\}$, we can treat the parabolic term as
\begin{align*}
    \partial_t \Ex{u} \varphi = \partial_t \Ex{u}({(\Ex{u}})_\delta)^\varepsilon\psi + \partial_t \Ex{u} \big((u_\delta)^\varepsilon - (({\Ex{u}})_\delta)^\varepsilon\big)\psi \geq \partial_t \Ex{u}({(\Ex{u}})_\delta)^\varepsilon\psi,
\end{align*}
where we obtain the last estimate by using Property $(ii)$ of Lemma \ref{lem:expmolproperties} and the fact that $s\mapsto \max\{\delta,s\}^\varepsilon$ is increasing to see that 
\begin{align*}
    \partial_t \Ex{u} \big((u_\delta)^\varepsilon - ({\Ex{u}}_\delta)^\varepsilon\big) = \tfrac1h (u - \Ex{u}) \big((u_\delta)^\varepsilon - (({\Ex{u}})_\delta)^\varepsilon\big) \geq 0.
\end{align*}
Thus, defining
\begin{align*}
    G_{\varepsilon,\delta}:\R\to \R, \quad G_{\varepsilon,\delta}(u) = \int^u_0 \max\{\delta,s\}^\varepsilon \d s,
\end{align*}
we can use the chain rule for Sobolev functions to conclude that
\begin{align}\label{est:parabolicterm1}
    \iint_{\Omega_T} \partial_t \Ex{u} \varphi \d x \d t &\geq \iint_{\Omega_T} \partial_t \Ex{u}(({\Ex{u}})_\delta)^\varepsilon\psi \d x \d t 
    \\
   \notag &= \iint_{\Omega_T} \partial_t\big( G_{\varepsilon,\delta}(\Ex{u})\big)\psi \d x \d t
    \\
   \notag & = -\iint_{\Omega_T}  G_{\varepsilon,\delta}(\Ex{u})\partial_t\psi \d x \d t
    \\
   \notag & \xrightarrow[h\to 0]{} -\iint_{\Omega_T}  G_{\varepsilon,\delta}(u)\partial_t\psi \d x \d t
   \\
  \notag  &\geq -\iint_{\Omega_T} G_{\varepsilon,\delta}(u)|\partial_t \psi| \d x \d t.
\end{align}
In the elliptic term we have directly the limit
\begin{align*}
&\iint_{\Omega_T}\sum^N_{j=1}\Ex{a(x,\cdot,u)|\partial_j u^{m_j}|^{p_j-2}\partial_j u^{m_j}}\partial_j \varphi \d x \d t \\
&\quad \xrightarrow[h\to 0]{} \iint_{\Omega_T}\sum^N_{j=1} a(x,t,u)|\partial_j u^{m_j}|^{p_j-2}\partial_j u^{m_j}\partial_j \varphi \d x \d t.
\end{align*}
The terms appearing on the last line can be split as follows
\begin{align}\label{fudbocs}
    a_j(x,t,u) |\partial_j u^{m_j}|^{p_j-2}\partial_j u^{m_j}\partial_j \varphi &= a_j(x,t,u) |\partial_j u^{m_j}|^{p_j-2}\partial_j u^{m_j}\partial_j (u_\delta)^\varepsilon \psi 
    \\
    \notag &\quad + a_j(x,t,u) |\partial_j u^{m_j}|^{p_j-2}\partial_j u^{m_j}\partial_j \psi (u_\delta)^\varepsilon.
\end{align}
Noting that $\partial_j (u_\delta)^\varepsilon$ vanishes a.e. on the set $\{u\leq \delta\}$, we can treat the first term on the right-hand side as
\begin{align}\label{est:ellipticterms1}
    a_j(x,t,u) |\partial_j u^{m_j}|^{p_j-2}\partial_j u^{m_j}\partial_j (u_\delta)^\varepsilon \psi
   & = a_j(x,t,u) |\partial_j u_\delta^{m_j}|^{p_j-2}\partial_j u_\delta^{m_j}\partial_j (u_\delta)^\varepsilon \psi 
    \\
 \notag  & = c a_j(x,t,u) u_\delta^{(m_j-m)(p_j-1) - m + \varepsilon}|\partial_j u_\delta^m|^{p_j}\psi
    \\
 \notag  & \geq c u_\delta^{(m_j-m)(p_j-1) - m + \varepsilon}|\partial_j u_\delta^m|^{p_j}\psi. 
\end{align}
The second term on the right-hand side of \eqref{fudbocs} can be estimated using Young's inequality as
\begin{align}\label{ellipticterms2}
    a_j(x,t,u) |\partial_j u^{m_j}|^{p_j-2}\partial_j u^{m_j}\partial_j \psi (u_\delta)^\varepsilon &\geq -c |\partial_j u^{m_j}|^{p_j-1}|\partial_j \psi|(u_\delta)^\varepsilon 
    \\
\notag    &\geq -c|\partial_j u^{m_j}|^{p_j} - c|\partial_j \psi|^{p_j}(u_\delta)^{\varepsilon p_j}.
\end{align}
We also note that
\begin{align*}
    |\langle u(0), \Exx{\varphi}(\cdot, 0)\rangle| \leq c\Big(\norm{\Exx{\varphi}(\cdot, 0)}_{L^{(m+1)/m}(\Omega)} + \sum^N_{j=1}\norm{\Exx{\partial_j \varphi}(\cdot, 0)}_{L^{p_j}(\Omega)} \Big) \xrightarrow[h\to 0]{} 0,
\end{align*}
where the convergence of the norms can be deduced from the expression of $\varphi$ and the Dominated Convergence Theorem, utilizing the fact that $\varphi=0$ for small values of $t$. \\
\noindent For the source term in \eqref{mollified_weak_forM_*ased_on_cont_into_dual_space} we have
\begin{align*}
    \iint_{\Omega_T} \Ex{f} \varphi \d x \d t \xrightarrow[h\to 0]{} \iint_{\Omega_T} f \varphi \d x \d t.
\end{align*}
Thus, taking into account \eqref{est:parabolicterm1}, \eqref{est:ellipticterms1} and \eqref{ellipticterms2}, after passing to the limit in \eqref{mollified_weak_forM_*ased_on_cont_into_dual_space}, we get
\begin{align*}
    &\iint_{\Omega_T} u_\delta^{(m_j-m)(p_j-1) - m + \varepsilon}|\partial_j u_\delta^m|^{p_j}\psi \d x \d t 
    \\
    &\quad \leq c \iint_{\Omega_T} |\partial_j u^{m_j}|^{p_j} + |\partial_j \psi|^{p_j}(u_\delta)^{\varepsilon p_j} +   G_{\varepsilon,\delta}(u)|\partial_t \psi| + f (u_\delta)^\varepsilon \psi \d x \d t.
\end{align*}
Since we consider $\delta \in (0,1)$, we have the estimates
\begin{align*}
    0\leq G_{\varepsilon,\delta}(u) \leq u^{\varepsilon+1} + 1, \quad u_\delta \leq u + 1,
\end{align*}
which then imply
\begin{align}\label{diesse}
    &\iint_{\Omega_T} u_\delta^{(m_j-m)(p_j-1) - m + \varepsilon}|\partial_j u_\delta^m|^{p_j}\psi \d x \d t 
    \\
   \notag &\quad \leq c \iint_{\Omega_T} |\partial_j u^{m_j}|^{p_j} + |\partial_j \psi|^{p_j}(u^{\varepsilon p_j} + 1) +   (u^{\varepsilon+1} + 1)|\partial_t \psi| + f (u^\varepsilon + 1) \psi \d x \d t.
\end{align}
Since we consider the parameters range \eqref{cond:m_j-closeness}, we have that 
\begin{align*}
    (m_j-m)(p_j-1) - m < 0,
\end{align*}
and thus we can take an $\varepsilon$ so small that the exponents of $u_\delta$ appearing on the left-hand side of \eqref{diesse} are non-positive. Hence, we can estimate 
\begin{align*}
    u_\delta^{(m_j-m)(p_j-1) - m + \varepsilon} \geq 1, \quad \textnormal{for } u \leq 1, 
\end{align*}
which allows us to conclude that
\begin{align}\label{est_u_delta_leq1}
    &\iint_{\Omega_T} |\partial_j u_\delta^m|^{p_j}\chi_{\{u\leq 1\}} \psi \d x \d t 
    \\
   \notag &\quad \leq c \iint_{\Omega_T} |\partial_j u^{m_j}|^{p_j} + |\partial_j \psi|^{p_j}(u^{\varepsilon p_j} + 1) +   (u^{\varepsilon+1} + 1)|\partial_t \psi| + f (u^\varepsilon + 1) \psi \d x \d t.
\end{align}
On the other hand, on the set $\{u>1\}$ we also have $u_\delta > 1$, so we can estimate
\begin{align}\label{est_u_delta_geq1}
    |\partial_j u_\delta^m|^{p_j} 
    = c u_\delta^{(m - m_j)p_j}|\partial_j u_\delta^{m_j}|^{p_j} 
 \leq c u_\delta^{(m - m_j)p_j}|\partial_j u^{m_j}|^{p_j} \leq c |\partial_j u^{m_j}|^{p_j}.
\end{align}
Combining \eqref{est_u_delta_leq1} and \eqref{est_u_delta_geq1}, 
we have
\begin{align*}
    &\iint_{\Omega_T}|\partial_j u_\delta^m|^{p_j}\psi \d x \d t \\
   \notag &\quad \leq c \iint_{\Omega_T} |\partial_j u^{m_j}|^{p_j} + |\partial_j \psi|^{p_j}(u^{\varepsilon p_j} + 1) +   (u^{\varepsilon+1} + 1)|\partial_t \psi| + f (u^\varepsilon + 1) \psi \d x \d t.
\end{align*}
Note that the right-hand side remains bounded independently of $\delta$. Since for any open subset $V:= U \times (t_1,t_2) \Subset \Omega_T$ we can choose a $\psi$ such that $\psi=1$ on $V$, the last estimate shows that $(\partial_j u_\delta^m)^\infty_{k=1}$ is a bounded sequence in $L^{p_j}(V)$. Thus a subsequence converges weakly to some limit function $v_j \in L^{p_j}(V)$. Since $u_\delta^m$ also converges to $u^m$ in $L^1(\Omega_T)$, we see that $v_j = \partial_j u^m$ on $V$. By considering an exhaustion $\Omega_T= \cup^\infty_{j=1} V_j$ of open subsets $V_j$ compactly contained in $\Omega_T$, we can thus confirm that $u^m$ has a weak gradient and that $\partial_j u^m \in L^{p_j}_{\textnormal{loc}}(\Omega_T)$.
\end{proof}
As a simple consequence of the previous proposition, we obtain the continuity for solutions of sufficiently high integrability as functions on the time interval into $L^{m+1}_{\textnormal{loc}}(\Omega)$.
\begin{cor}
Let $u \in V^{\bf p, \bf m}$ be a weak solution to \eqref{eq:diffusion} in the sense of Definition \ref{def:weaksol} and suppose that the exponents $m_j$ satisfy \eqref{cond:m_j-closeness}. Suppose also $u \in L^{m\max\{p_j\}}(\Omega_T)$. Then $u$ has a representative in $C([0,T];L^{m+1}_\textnormal{loc}(\Omega))$.
\end{cor}
\begin{proof}{}
One can apply the reasoning of the proofs of Lemma 4.1 and Lemma 4.2 in \cite{CiaVeVe} with the choice $\alpha=\tfrac1m$ to the function $u^m$. The difference in the form of the elliptic term does not affect the argument, as only the integrability properties of the vector field are relevant. Terms arising due to the right-hand side $f$ can be treated since $f\in L^{\bar p'}(\Omega_T)$ and the inequality \eqref{eq:Sobolev_Troisi_Inequality} guarantees that the appropriate test function lies in $L^{\bar p}(\Omega_T)$ so that we have dual exponents in these terms.
\end{proof}
 
\section{Existence of Solutions}\label{sec:existence}
In order to prove the existence we introduce a family of Cauchy-Dirichlet problems which approximate the original problem, and for which existence is already known. We will prove that a sequence $(u_k)_{k = 1}^\infty$ of solutions to the approximating problems converges to a solution to the original problem. 
\subsection{Approximating solutions and structure conditions}
For $k\in \N$ we introduce the truncation $ T_k:\R \to \R$,
\begin{align*}
 T_k(s) := \min \left\{ k, \max \left\{s, \tfrac{1}{k} \right\} \right\},
\end{align*}
and define a modified vector field $\hat{A}^k = (\hat{A}^k_1, \dots, \hat{A}^k_N)$ as 
\begin{align}\label{def:A^k}
    \hat{A}_j^k (x, t, u, \xi) & :=  \,  a_j(x, t, u) \lvert m_j T_k(u) ^{m_j - 1} \xi_j \rvert ^{p_j - 2} m_j T_k(u) ^{m_j - 1} \xi_j 
    \\
  \notag  & \hphantom{:} =  \, a_j(x, t, u) m_j^{p_j-1} T_k(u)^{(m_j - 1)(p_j - 1)}\lvert \xi_j\rvert^{p_j - 2}\xi_j
    \\
   \notag &\hphantom{:} =:  \hat{a}^k_j(x,t,u)|\xi_j|^{p_j -2}\xi_j.
\end{align}
It follows directly from the upper and lower bounds imposed on $a_j$ in \eqref{cond:unif_ellipt} that the vector field $\hat{A}^k$ satisfies the following structure conditions
\begin{align}
\label{cond:structure1_A_k}    \lvert \hat{A}_j^k (x, t, u, \xi) \rvert  &\leq b_k \lvert \xi_j \rvert^{p_j - 1},
\\
\label{cond:structure2_A_k}    \hat{A}^k (x, t, u, \xi) \cdot \xi &\geq c_k \sum_{j = 1}^N \lvert \xi_j \rvert^{p_j}, 
\end{align}
for some $k$-dependent constants $b_k$ and $c_k$.
Using the Lipschitz continuity condition \eqref{cond:lip_cont} and the boundedness of the coefficients $a_j$ together with the Lipschitz continuity and boundedness of the functions $T_k$, we obtain the Lipschitz continuity of the coefficients $\hat{a}^k_j$ with respect to the $u$-variable:
\begin{align}\label{cond:lipschitz_A_k}
 |\hat{a}^k_j(x,t,u) - \hat{a}^k_j(x,t,v)| \leq c^k_j |u - v|.
\end{align}

We also have the strict monotonicity property:
\begin{align}
   \notag (\hat{A}^k(x,t,u,\xi) - \hat{A}^k(x,t,u,\eta))\cdot (\xi - \eta) &= \sum^N_{i=1} \hat{a}^k_j(x,t,u)(|\xi_j|^{p_j -2}\xi_j - |\eta_j|^{p_j-2}\eta_j)(\xi_j - \eta_j) 
    \\
   \label{monotonicity-A_k} &> 0, \textnormal{ if } \xi \neq \eta. 
\end{align}
Due to the properties \eqref{cond:structure1_A_k}, \eqref{cond:structure2_A_k} and \eqref{monotonicity-A_k}, we can conclude that there is a solution $u_k$ in $L^{\bf p}(0,T; W^{1,\bf p}(\Omega)) \cap C([0,T]; L^2(\Omega))$ to the problem
\begin{center}
\begin{equation}\label{eq:approximative_problem}
    \begin{cases}
    \partial_t u_k - \nabla \cdot \hat{A}^k(x, t, u_k, \nabla u_k) = f & \text{in} \;\; \Omega_T, \\
    u_k = g + \frac{1}{k} & \text{on} \;\; \partial\Omega \times (0,T), \\
    u_k(\cdot, 0) = u_0 + \frac{1}{k} & \text{in} \;\; \Omega \times \{0\}.
\end{cases}
\end{equation}
\end{center}
In the range $\bar p \leq 2$ this follows directly by applying the existence result in \cite{Ve} with $\alpha=1$, as $f$ is in $L^{\bar p'}(\Omega_T)$ and $\bar p' \geq 2 = (\alpha + 1)/\alpha$. In the range $\bar p > 2$ one can use the existence result in \cite{Ve} with the right-hand side $f_l := \min\{l, f\}$ to find a sequence of solutions which in the limit $l\to \infty$ converge to a solution to \eqref{eq:approximative_problem}. For the reader's convenience, the argument, which shares many features with the proof of our main existence result, is summarized in Appendix \ref{app:Existence_approx_problems}.

The vector field $\hat{A}^k$ satisfies all the conditions for the comparison principle in Theorem \ref{thm:comparison_non-doubly-nonlinear}. Since the right-hand side, the initial values and the boundary values are nonnegative by assumptions \eqref{cond:f}, \eqref{cond:u_0} and \eqref{cond:g}, and since the constant functions $\tfrac1k$ are solutions to \eqref{eq:approximative_problem} with $f=0$ we can apply Theorem \ref{thm:comparison_non-doubly-nonlinear} with $v=\frac1k$, $u=u_k$ to obtain the following lemma.
\begin{lem}\label{lem:lower_bound_uk}
    Let $u_k$ be a solution to problem \eqref{eq:approximative_problem}. Then $u_k \geq \frac{1}{k}$ in $\Omega_T.$
\end{lem}

\subsection{Global Boundedness for the Approximative Solutions}\label{sec:Global_Boundedness}
We now prove a global bound for the solutions to the approximative problems \eqref{eq:approximative_problem} that is independent of $k$. We will make extensive use of the quantity
\begin{align}\label{def:M_*}
    M_* := \max\{\norm{u_0}_{L^\infty(\Omega)}, \norm{g}_{L^\infty(\Omega_T)}\} + 1.
\end{align}
We begin by establishing an energy estimate.
\begin{lem}\label{lem:energy_estimates_approx_probs}
    Let $u_k$ be a weak solution to the problem \eqref{eq:approximative_problem} in the sense of Definition \ref{def:weak_sol_coef_usual_anisotropic}, and define $v_k := T_k \circ u_k = \min\{k, u_k\}$. Then $v_k$ satisfies
    \begin{align}\label{energy_est_approx_problms}
        \esssup_{\tau \in[0, T]} \int_\Omega (v_k^{\frac{m + 1}{2}} - M^{\frac{m + 1}{2}})^2_+ (x, \tau) \d x 
        & + \sum_{j = 1}^N \iint_{\Omega_T}  v_k ^{(m_j - m)(p_j - 1)} \lvert \partial_j (v_k^m - M^m)_+ \rvert^{p_j} \d x \d t \notag
        \\
        & \leq C \iint_{\Omega_T} \lvert f \rvert^{\bar p\prime} \chi_{[v_k > M]} \d x \d t,
    \end{align}
    for all $M\geq M_*$, with a constant $C$ which is independent of $k$ and $M$.
\end{lem}
\begin{proof}{}
    We recall the trapezoidal function defined in \eqref{func:trapezoid} and choose 
    \begin{align*}
       \varphi = (T_k([u_k]_h)^m - M^m)_+ \zeta_{0,\tau}^\delta, 
    \end{align*}
     as test function in the mollified weak formulation \eqref{eq:weak_mollified_usual_anisotropic_Steklov}. The test function vanishes on the lateral boundary for $M\geq M_*$ and lies in the correct Sobolev space by the Chain Rule for Sobolev functions since $s\mapsto (T_k(s)^m - M^m)_+$ is piecewise $C^1$ with bounded derivative. Our goal is to first pass to the limit $h \to 0$ and then take $\delta \to 0$.
    For the parabolic term we introduce the function
    \begin{align}
        F(s)  :=& \int_0^s ( T_k( \sigma )^m - M^m)_+ \d \sigma  \notag \\
        =& \begin{cases} 
            0 & \text{if } s < M, \\
            \b[s, M] & \text{if } s \in [M, k), \\
            \b[k, M] + (s - k)(k^m - M^m) & \text{if } s \geq k, 
        \end{cases} \label{prop:funct_F_def}
    \end{align}
    where, $\b$ is defined as in \eqref{eq:b_definition}. We observe that
    \begin{align*}
        \iint_{\Omega_{T - h}} \partial_t ( [u_k]_h ) \varphi \d x \d t & = \iint_{\Omega_{T-h}} \zeta_{0,\tau}^\delta \partial_t F( [u_k]_h ) \d x \d t \\
        & = - \iint_{\Omega_{T - h}} (\zeta_{0,\tau}^\delta)' F( [u_k]_h ) \d x \d t \\
        & \xrightarrow[h \to 0]{} - \iint_{\Omega} (\zeta_{0,\tau}^\delta)' F( u_k ) \d x \d t \\
        & = \frac{1}{\delta} \int_{\tau-\delta}^{\tau} \int_\Omega F( u_k ) \d x \d t - 
        \frac{1}{\delta} \int_0^{\delta} \int_\Omega F( u_k ) \d x \d t \\
        & \xrightarrow[\delta \to 0]{} \int_\Omega F(u_k) (x, \tau) \d x.
    \end{align*}
    In calculating the limit $\delta \to 0$ we utilize the fact that $F$ is Lipschitz and $u_k \in C([0, T]; L^2(\Omega))$. The second term on the penultimate row vanishes in the limit since $F(u_k)(\cdot, 0) = F(\tfrac{1}{k} + u_0(\cdot))$ and $\tfrac{1}{k} + \lVert u_0 \rVert_{L^\infty(\Omega)} \leq M.$ For the elliptic term, we note that 
    \begin{align*}
        \iint_{\Omega_{T - h}} [\hat{A}^k(x, t, u_k, \nabla u_k)]_h \cdot \nabla \varphi \d x \d t \xrightarrow[h \to 0]{} \iint_{\Omega_T} \zeta_{0,\tau}^\delta \hat{A}^k(x, t, u_k, \nabla u_k) \cdot \nabla (v_k^m - M^m)_+ \d x \d t.
    \end{align*}
     This follows since each component of $[\hat{A}^k]_h$ and the corresponding partial derivative of $\varphi$ converge as $h\to 0$ in $L^p$-spaces of matching H\"older exponents to the corresponding unmollified quantities. Passing to the limit $h\to 0$ on the right-hand side in the weak formulation can be justified in a similar manner.
     Thus, after passing to the limits $h\to 0$ and $\delta \to 0$ in all terms of the weak formulation we obtain for a.e. $\tau \in (0, T)$
    \begin{align}\label{gggr}
        \int_\Omega F( u_k )(x, \tau) \d x 
        + &\iint_{\Omega_\tau} \hat{A}^k(x, t, u_k, \nabla u_k) \cdot \nabla (v_k^m - M^m)_+ \d x \d t 
        \\
       \notag &= \iint_{\Omega_\tau} f(v_k^m - M^m)_+ \d x \d t.
    \end{align}
The Chain Rule for Sobolev functions shows that
   \begin{align*}
       \partial_j (v_k^m - M^m)_+ = \chi_{\{M < u_k < k\}} m u_k^{m - 1} \partial_j u_k = \chi_{\{M < u_k < k\}} m v_k^{m - 1} \partial_j u_k,
   \end{align*}
so that $\partial_j u_k = m^{-1} v_k^{1-m} \partial_j (v_k^m - M^m)_+$ on the set $\{M<u_k<k\}$. This observation and the definition of $v_k$ allows us to compute
\begin{align}\label{est:hat-A_k}
    &\hat{A}_j^k(x, t, u_k, \nabla u_k)\partial_j(v_k^m - M^m)_+ 
    \\
   \notag &\quad = a_j(x,t,u_k)\lvert m_j v_k^{m_j - 1} \partial_j u_k \rvert^{p_j - 2} m_j v_k^{m_j - 1}\partial_j u_k \partial_j (v_k^m - M^m)_+
    \\
   \notag &\quad \geq c v_k^{(m_j - m)(p_j - 1)} \lvert \partial_j (v_k^m - M^m)_+ \rvert^{p_j}.
\end{align}
    Since $v_k\leq k$ we can combine \eqref{prop:funct_F_def} and the left inequality in \eqref{b:basic-est} to conclude
    \begin{align}\label{est:b-F}
        c \big(v_k^{\frac{m + 1}{2}} - M^\frac{m + 1}{2}\big)_+^2 \leq \b[v_k, M] \chi_{\{v_k > M\}} = F(v_k) \leq F(u_k).
    \end{align}
    Using Young's inequality with $\varepsilon$ we have
    \begin{align}\label{eq:energy_est_v_k^beta}
       \notag \iint_{\Omega_\tau} f(v_k^m - M^m)_+ \d x \d &t  \leq  c_\varepsilon \iint_{\Omega_\tau} \abs{f}^{\bar p\prime} \chi_{\{v_k > M\}} \d x \d t 
         + \varepsilon \iint_{\Omega_\tau} (v_k^m - M^m)^{\bar p}_+ \d x \d t 
        \\
         &\leq  c_\varepsilon \iint_{\Omega_\tau} \abs{f}^{\bar p\prime} \chi_{\{v_k > M\}} \d x \d t 
        + c \varepsilon \sum_{j = 1}^N  \iint_{\Omega_\tau} \lvert \partial_j (v_k^m - M^m)_+\rvert^{p_j} \d x \d t,
    \end{align}
    where, in the last step we utilize the Sobolev-Troisi inequality \eqref{eq:Sobolev_Troisi_Inequality} slice-wise. Utilizing each estimate \eqref{est:hat-A_k}, \eqref{est:b-F} and \eqref{eq:energy_est_v_k^beta} to treat the corresponding term in \eqref{gggr} we obtain
\begin{align*}
    \int_\Omega \big(v_k^{\frac{m + 1}{2}} - M^\frac{m + 1}{2}\big)_+^2(x,\tau) \d x + \sum_{j = 1}^N \iint_{\Omega_\tau}  v_k ^{(m_j - m)(p_j - 1)} \lvert \partial_j (v_k^m - M^m)_+ \rvert^{p_j} \d x \d t 
    \\
    \leq c_\varepsilon \iint_{\Omega_\tau} \abs{f}^{\bar p\prime} \chi_{\{v_k > M\}} \d x \d t 
        + c \varepsilon \sum_{j = 1}^N  \iint_{\Omega_\tau} \lvert \partial_j (v_k^m - M^m)_+\rvert^{p_j} \d x \d t.
\end{align*}
Since $v_k \geq 1$ on the set where $(v_k^m - M^m)_+$ is nonzero, we can absorb the terms in the sum on the right-hand side into the corresponding terms on the left-hand side by choosing $\varepsilon$ sufficiently small. Taking the essential supremum over $\tau$ we recover \eqref{energy_est_approx_problms}.
\end{proof}

\begin{cor}\label{coro:uniform_Lp_bound_approx_seq}
    The sequence $(v_k^m)$ is bounded in $L^{\bar p}(\Omega_T).$
\end{cor}
\begin{proof}{}
    Note that
    \begin{align}\label{itsamee}
        \lVert v_k^m \rVert_{L^{\bar p}(\Omega_T)} & \leq \Big( \iint_{\Omega_T} (v_k^m - M_*^m)_+^{\bar p} \d x \d t + \iint_{\Omega_T} (v_k^m - M_*^m)_{-}^{\bar p} \d x \d t \Big)^\frac{1}{\bar p} + M_*^m \lvert \Omega_T \rvert^\frac{1}{\bar p} 
        \\
       \notag & \leq c \Big[ \iint_{\Omega_T} (v_k^m - M_*^m)_+^{\bar p} \d x \d t \Big]^\frac{1}{\bar p} + cM_*^m \lvert \Omega_T \rvert^\frac{1}{\bar p} 
        \\
        \notag &\leq c \Big[\sum_{j = 1}^N  \iint_{\Omega_\tau} \lvert \partial_j (v_k^m - M_*^m)_+\rvert^{p_j} \d x \d t \Big]^\frac{1}{\bar p} + c M_*^m \lvert \Omega_T \rvert^\frac{1}{\bar p},
    \end{align}
    where in the last step we use \eqref{eq:Sobolev_Troisi_Inequality}. Since $v_k \geq 1$ on the set where $(v_k^m - M_*^m)_+$ is nonzero, we can use the energy estimate \eqref{energy_est_approx_problms} to bound the sum on the last row of \eqref{itsamee}, to obtain
    \begin{align}\label{def:K_0}
        \lVert v_k^m \rVert_{L^{\bar p}(\Omega_T)}  
        & \leq c\Big( \iint_{\Omega_T} \abs{f}^{\bar p^\prime} \d x \d t \Big)^\frac{1}{\bar p} + cM_*^m \lvert \Omega_T \rvert^\frac{1}{\bar p} =: K_0,
    \end{align}
    where $K_0$ depends on $\Omega, N, \mathbf{p}, \mathbf{m}, T, f, g$ and $u_0$ but is independent of $k$.
\end{proof}


\begin{lem}\label{lemma:De_Giorgi_Approx_probs}
    Let $u_k$ be a weak solution to the problem \eqref{eq:approximative_problem} in the sense of Definition \ref{def:weak_sol_coef_usual_anisotropic}. There exists a constant $L$ depending upon $\Omega, N, \mathbf{p}, m, T, f, g$ and $u_0$ and independent  of $k$  such that, for sufficiently large $k,$ $u_k \leq L$ a.e. in $\Omega_T.$
\end{lem}
\begin{proof}{}
    We first prove a uniform upper bound for the functions $(v_k)$. If $f\equiv 0$, Lemma \ref{lem:energy_estimates_approx_probs} implies that $v_k \leq M_*$ for all $k$. We may thus assume that $f\neq 0$.  We choose numbers $1< q_i \leq p_i$ so that $1 < \bar q < N$. If $\bar p < N$ we can take $ q_i:=  p_i$ for all $i$. Otherwise, this can be achieved by setting $q_i := p_i$ for $i \geq 2$ and taking $q_1$ as close to 1 as necessary. Let $M\geq M_*$ and define for every $j\in \mathbb{N}_0$:
    \begin{align*}
        M_j := M(2 - 2^{-j})^{\frac{2}{m + 1}}, 
        & \hspace{0.5em} Y_j := \iint_{\Omega_T} (v_k^\frac{m + 1}{2} - M_j^\frac{m + 1}{2})_+^\frac{2 m \bar q}{m + 1} \d x \d t
        & \textnormal{and} \hspace{0.5em} E_j := \Omega_T \cap \{v_k > M_j\}.
    \end{align*}
In the following calculations we denote $\mu:= \tfrac{m+1}{m}$ for convenience. Since $v_k$ is bounded from above and below by positive constants, two applications of the Chain Rule for Sobolev functions allow us to obtain the following estimate:
\begin{align}\label{De_Giorgi_derivative_estimate}
        \lvert \partial_i (v_k^{\frac{m + 1}{2}} - M_{j+1}^\frac{m + 1}{2})_+^\frac{2}{\mu} \rvert 
        & \leq \tfrac{2 m}{m + 1} (v_k^{\frac{m + 1}{2}} - M_{j+1}^\frac{m + 1}{2})_+^{\frac{2m}{m + 1} - 1} \lvert \partial_i (v_k^{\frac{m + 1}{2}} - M_{j+1}^\frac{m + 1}{2})_+ \rvert 
        \\
      \notag   & \leq \tfrac{2 m}{m + 1} v_k^\frac{m - 1}{2} \chi_{\{v_k > M_{j+1}\}}|\partial_i v_k^{\frac{m + 1}{2}}|
        \\
       \notag  & = \lvert \partial_i (v_k^m - M_{j+1}^m)_+ \rvert. 
    \end{align}
To estimate $Y_{j+1}$ we apply H\"older's inequality first over the whole domain, and then only in space. We estimate one of the resulting factors by its essential supremum over time and apply Lemma \ref{lem:troisiii} for the exponents $q_i$ in the other factor. The spatial derivatives can be further estimated upwards using  \eqref{De_Giorgi_derivative_estimate} and another application of H\"older's inequality, after which we end up with terms that can be bounded from above using the
energy estimate \eqref{energy_est_approx_problms} (since $v_k \geq 1$ on $E_{j+1}$). After this, another application of H\"older's inequality introduces the parameter $\sigma$ tied to the integrability condition \eqref{cond:f}. All in all the calculation takes the form: 
    \begin{align}\label{De_Giorgi_first_estimates}
        Y_{j + 1} 
        &\leq \lvert E_{j+1} \rvert ^\frac{\mu}{N + \mu} \Big( \iint_{\Omega_T} (v_k^\frac{m + 1}{2} - M_{j+1}^\frac{m + 1}{2})_+^{\frac{2\bar q}{N} + \frac{2\bar q}{\mu}} \d x \d t \Big)^\frac{N}{N + \mu}
        \\
        &  \leq \lvert E_{j+1} \rvert^\frac{\mu}{N + \mu} \Big(\int_0^T \Big[  \int_\Omega (v_k^\frac{m + 1}{2} - M_{j+1}^\frac{m + 1}{2})_+^2 \d x \Big]^\frac{\bar q}{N} \Big[\int_\Omega (v_k^\frac{m + 1}{2} - M_{j+1}^\frac{m + 1}{2})_+^\frac{2 \bar q^*}{\mu} \d x \Big]^\frac{\bar q}{\bar q^*} 
          \d t \Big)^\frac{N}{N + \mu} \nonumber
        \\
        & \leq \lvert E_{j+1} \rvert ^\frac{\mu}{N + \mu} \Big( \esssup\limits_{\tau \in (0, T)} \int_\Omega (v_k^\frac{m + 1}{2} - M_{j+1}^\frac{m + 1}{2})_+^2 (x, \tau) \d x\Big)^\frac{\bar q}{N + \mu} \nonumber
        \\
        &\quad  \times\Big(\int_0^T \Big[\int_{\Omega} (v_k^\frac{m + 1}{2} - M_{j+1}^\frac{m + 1}{2})_+^\frac{2 \bar q^*}{\mu}  \d x \Big]^\frac{\bar q}{\bar q^*} \d t \Big)^\frac{N}{N + \mu} \nonumber
        \\
        & \leq c \lvert E_{j+1} \rvert ^\frac{\mu}{N + \mu} \Big( \esssup\limits_{\tau \in (0, T)} \int_\Omega (v_k^\frac{m + 1}{2} - M_{j+1}^\frac{m + 1}{2})_+^2 (x, \tau) \d x\Big)^\frac{\bar q}{N + \mu} \nonumber
        \\
        & \qquad \times
        \Big( \int_0^T \prod_{i = 1}^N \Big[ \int_\Omega  \lvert \partial_i (v_k^\frac{m+1}{2} - M_{j+1}^\frac{m+1}{2})_+^\frac{2}{\mu}\rvert^{q_i} \d x \Big]^\frac{\bar q}{N q_i} \d t\Big)^\frac{N}{N + \mu} \nonumber 
        \\
        & \leq c \lvert E_{j+1} \rvert ^\frac{\mu}{N + \mu} \Big( \esssup\limits_{\tau \in (0, T)} \int_\Omega (v_k^\frac{m + 1}{2} - M_{j+1}^\frac{m + 1}{2})_+^2 (x, \tau) \d x\Big)^\frac{\bar q}{N + \mu} \nonumber
        \\
        & \qquad\times
        \Big( \int_0^T \prod_{i = 1}^N \Big[ \int_\Omega  \lvert \partial_i (v_k^m - M_{j+1}^m)_+\rvert^{q_i} \d x \Big]^\frac{\bar q}{N q_i} \d t\Big)^\frac{N}{N + \mu} \nonumber 
        \\
        & \leq c\lvert E_{j+1} \rvert ^\frac{\mu}{N + \mu} \Big( \esssup\limits_{\tau \in (0, T)} \int_\Omega (v_k^\frac{m + 1}{2} - M_{j+1}^\frac{m + 1}{2})_+^2 (x, \tau) \d x\Big)^\frac{\bar q}{N + \mu} \nonumber
        \\
        & \qquad\times
        \prod_{i = 1}^N \Big( \iint_{\Omega_T} \lvert \partial_i (v_k^m - M_{j+1}^m)_+\rvert^{q_i} \d x \d t \Big)^\frac{\bar q}{q_i(N + \mu)}  \nonumber 
        \\
        & \leq c\lvert E_{j+1} \rvert ^\frac{\mu}{N + \mu} \Big( \esssup\limits_{\tau \in (0, T)} \int_\Omega (v_k^\frac{m + 1}{2} - M_{j+1}^\frac{m + 1}{2})_+^2 (x, \tau) \d x\Big)^\frac{\bar q}{N + \mu} \nonumber
        \\
        & \qquad\times
        \prod_{i = 1}^N \Big( \lvert E_{j+1}\rvert^{1 - \frac{q_i}{p_i}} \Big( \iint_{\Omega_T} \lvert \partial_i (v_k^m - M_{j+1}^m)_+\rvert^{p_i} \d x \d t \Big)^\frac{q_i}{p_i} \Big)^\frac{\bar q}{q_i(N + \mu)}  \nonumber 
        \\
        & \leq c\lvert E_{j+1} \rvert ^{\frac{\mu}{N + \mu} + (\frac{1}{\bar q} - \frac{1}{\bar p})\frac{N\bar q}{N + \mu}} \Big( \iint_{E_{j+1}} \abs{f}^{\bar p^\prime} \d x \d t \Big)^{\frac{\bar q}{N + \mu}(1 + \frac{N}{\bar p})} 
        \nonumber 
        \\
       \notag  & \leq c \lvert E_{j+1} \rvert ^{\frac{\mu}{N + \mu} + (\frac{1}{\bar q} - \frac{1}{\bar p})\frac{N\bar q}{N + \mu} + (1 - \frac{1}{\sigma})\frac{N \bar q}{N + \mu}(\frac{1}{N} + \frac{1}{\bar p})}  \Big( \iint_{\Omega_T} \abs{f}^{\sigma \bar p^\prime} \d x \d t \Big)^{\frac{\bar q}{\sigma(N + \mu)}(1 + \frac{N}{\bar p})}
       \\
      \notag  & = c|E_{j+1}|^{1 + \delta} \Big( \iint_{\Omega_T} \abs{f}^{\sigma \bar p^\prime} \d x \d t \Big)^Q,
\end{align}
where we denote
\begin{align*}
    \delta :=  \tfrac{N \bar q}{N + \mu}(\tfrac{1}{N} - \tfrac{1}{\sigma}(\tfrac{1}{N} + \tfrac{1}{\bar p})), \quad Q := \tfrac{\bar q}{\sigma(N + \mu)}(1 + \tfrac{N}{\bar p}).
\end{align*}
Note that \eqref{cond:lower_bound_sigma} guarantees that $\delta > 0$. For the set $\lvert E_{j + 1} \rvert$ we have the following bound:
    \begin{align}\label{De_Giorgi_Estimate_A_{j+1}}
        \lvert E_{j + 1} \rvert & = \lvert E_{j + 1} \rvert M^{-m \bar q} 2^{(j + 1)\frac{2m \bar q}{m + 1}} (M_{j + 1}^{\frac{m + 1}{2}} - M_j^{\frac{m + 1}{2}})^\frac{2m \bar q}{m + 1}  \\
        & \leq M^{-m \bar q} 2^{(j + 1)\frac{2m \bar q}{m + 1}} \iint_{E_{ j + 1 } } (v_k^{\frac{m + 1}{2}} - M_j^{\frac{m + 1}{2}})_+^\frac{2m \bar q}{m + 1} \d x \d t \notag \\
        \notag & \leq M^{-m \bar q} 2^{(j + 1)\frac{2m \bar q}{m + 1}} Y_j. 
    \end{align}
    Combining \eqref{De_Giorgi_first_estimates} and \eqref{De_Giorgi_Estimate_A_{j+1}} we have
    \begin{align}\label{est:fastgeom}
       Y_{j+1} & \leq K M^{-m\bar q(1+\delta)} b^j Y_j^{1+\delta},
    \end{align}
where 
\begin{align*}
K:= c \Big( \iint_{\Omega_T} \abs{f}^{\sigma \bar p^\prime} \d x \d t \Big)^Q, \quad b = 2^\frac{2m\bar q(1+\delta)}{m+1} > 1.
\end{align*}
Thus, by Lemma \ref{lem:fast_geometric_convergence}, $(Y_j)_{j=0}^\infty$ converges to zero provided that 
\begin{align}\label{cond_Y_0_convg}
    Y_0 \leq K^{-\frac{1}{\delta}} M^{m\bar q \frac{(1+\delta)}{\delta}} b^{-\frac{1}{\delta^2}}.
\end{align}
The definition of $Y_0$ and \eqref{def:K_0} show that
\begin{align*}
    Y_0 \leq c \norm{v_k^m}_{L^{\bar p}(\Omega_T)}^{\bar q} \leq c K_0^{\bar q},
\end{align*}
so we conclude that \eqref{cond_Y_0_convg} is satisfied if 
\begin{align*}
    c K_0^{\bar q} \leq K^{-\frac{1}{\delta}} M^{m\bar q \frac{(1+\delta)}{\delta}} b^{-\frac{1}{\delta^2}},
\end{align*}
which can be rephrased as
\begin{align}\label{bound_lower_M_new}
    M \geq \gamma (K_0^{\bar q\delta} K)^\frac{1}{m \bar q(1+\delta)},
\end{align}
for some $\gamma$ depending on the data and $\Omega_T$ but independent of $k$. This bound, together with the fact that $M\geq M_*$ which was also required in our argument, are both satisfied if we set
\begin{align}
    M := \max\{M_*, \gamma (K_0^{\bar q\delta} K)^\frac{1}{m \bar q(1+\delta)}  \},
\end{align}
and the expression on the right-hand side is independent of $k$. With this choice of $M$, 
\begin{align*}
     \iint_{\Omega_T} (v_k^\frac{m + 1}{2} - 2M^\frac{m + 1}{2})_+^\frac{2 m \bar q}{m + 1} \d x \d t \leq Y_j \xrightarrow[j\to 0]{} 0,
 \end{align*}
which shows that $v_k \leq 2^\frac{2}{m+1}M =: L$ a.e. on $\Omega_T$. This is our uniform upper bound for the sequence $(v_k)$. Considering the definition of $v_k$ as a truncation of $u_k$ this implies that, for $k\geq L$, we also have that $u_k \leq L$.
\end{proof}

\begin{rem}\label{rem:A^k-rewritten}
Considering that $\tfrac1k \leq u_k \leq L$ we see that for sufficiently large $k$, the truncation $T_k$ appearing in the definition \eqref{def:A^k} of $\hat{A}^k$ can be removed, which shows that
\begin{align*}
    \hat{A}^k_j(x,t,u_k,\nabla u_k) &= a_j(x, t, u_k) \lvert m_j u_k^{m_j - 1} \partial_j u_k \rvert ^{p_j - 2} m_j u_k^{m_j - 1} \partial_j u_k 
    \\
    &= a_j(x,t,u_k)|\partial_j u_k^{m_j}|^{p_j-2}\partial_j u_k^{m_j},
\end{align*}
where in the last step we have used the Chain Rule. Thus, $u_k$ satisfies the original equation for sufficiently large $k$. That is, $u_k$ satisfies 
\begin{align}\label{eq:weak_form_u_k}
&\iint_{\Omega_T} \sum^N_{j=1} a_j(x,t,u_k)|\partial_j u_k^{m_j}|^{p_j-2}\partial_j u_k^{m_j} \partial_j \varphi -  u_k\partial_t \varphi\d x\d t = \iint_{\Omega_T} f \varphi \d x \d t.
\end{align}
\end{rem}

\subsection{Weak convergence, compactness and pointwise convergence}
\begin{lem}\label{lem:grad_bdd_in_Lp}
   For all $j \in \{1,\dots, N\}$ the sequence $(\partial_j u_k^m)^\infty_{k=1}$ is bounded in $L^{p_j}(\Omega_T)$. In particular, the sequences $(\partial_j u_k^{m_j})^\infty_{k=1}$ are also bounded in $L^{p_j}(\Omega_T)$.
\end{lem}
\begin{proof}{}
We use the mollified weak formulation \eqref{mollified_weak_forM_*ased_on_cont_into_dual_space} that involves the exponential mollification in time \eqref{def:moll} with the test function $\varphi = (u_k^{\varepsilon} - (g + \tfrac{1}{k})^{\varepsilon}) \zeta_{\tau_1,\tau_2}^\delta$  where $\zeta_{\tau_1,\tau_2}^\delta$ is as in \eqref{func:trapezoid}, $0<\tau_1 <\tau_2< T$ and $\varepsilon \in (0,1]$ is a parameter to be chosen later. This is an admissible test function since $u_k$ belongs to the correct anisotropic parabolic Sobolev space, and since $u_k$ attains the values of $g + \tfrac{1}{k}$ over the lateral boundary. Observe as well that the time continuity of $u_k$ allows us to interpret $u_k(0)$ as an element of $L^2(\Omega)$ in this case. We will perform estimates for each term in \eqref{mollified_weak_forM_*ased_on_cont_into_dual_space} and pass to the limit $h\to 0$.
For the first term on the right-hand side of \eqref{mollified_weak_forM_*ased_on_cont_into_dual_space} we have
\begin{align}\label{eq:vanishing_initial_time_term}
   \langle u_k(0), \Exx{\varphi}(\cdot, 0) \rangle = \int_\Omega u_k(x, 0) \Exx{\varphi}(x, 0) \d x \xrightarrow[h \to 0]{} 0,
\end{align}
which can be seen by the Dominated Convergence Theorem, using also the fact that $\varphi$ vanishes on $[0,\tau_1]$. For the source term on the right-hand side of \eqref{mollified_weak_forM_*ased_on_cont_into_dual_space} we have
\begin{align}\label{convg:source}
    \iint_{\Omega_T} \Ex{f}\varphi \d x \d t &\xrightarrow[h \to 0]{} \iint_{\Omega_T} f (u_k^{\varepsilon} - (g + \tfrac{1}{k})^{\varepsilon}) \zeta_{\tau_1,\tau_2}^\delta \d x \d t .
\end{align}
We treat the parabolic term in \eqref{mollified_weak_forM_*ased_on_cont_into_dual_space} as follows:
    \begin{align}\label{est:parabterm}
        \iint_{\Omega_T} \partial_t \Ex{u_{k}} \varphi \d x \d t  & = \iint_{\Omega_T}\partial_t \Ex{u_k} (\Ex{u_k}^{\varepsilon} - (g + \tfrac{1}{k})^{\varepsilon}) \zeta_{\tau_1,\tau_2}^\delta \d x \d t
        \\
       \notag & \quad + \iint_{\Omega_T} \partial_t \Ex{u_k} (u_k^\varepsilon - \Ex{u_k}^\varepsilon) \zeta_{\tau_1,\tau_2}^\delta \d x \d t
        \\
       \notag & = \iint_{\Omega_T}\partial_t \Ex{u_k} (\Ex{u_k}^{\varepsilon} - (g + \tfrac{1}{k})^{\varepsilon}) \zeta_{\tau_1,\tau_2}^\delta \d x \d t
        \\
       \notag & \quad + \iint_{\Omega_T} \tfrac{1}{h}(u_k - \Ex{u_k}) (u_k^\varepsilon - \Ex{u_k}^\varepsilon) \zeta_{\tau_1,\tau_2}^\delta \d x \d t
        \\
       \notag & =: I_a + I_b \geq I_a,
    \end{align}
    where the last estimate follows from the fact that $I_b\geq 0$ since $t \mapsto t^\varepsilon$ is increasing. We split $I_a$ into two integrals and use the Chain Rule and integration by parts as follows:  
    \begin{align}\label{convg:I_a}
        I_a &= \iint_{\Omega_T} \Ex{u_k}^{\varepsilon}\partial_t \Ex{u_k}\zeta_{\tau_1,\tau_2}^\delta \d x \d t  - \iint_{\Omega_T} \partial_t \Ex{u_k} (g + \tfrac{1}{k})^{\varepsilon} \zeta_{\tau_1,\tau_2}^\delta \d x \d t
        \\
      \notag  &= -\iint_{\Omega_T} \tfrac{1}{\varepsilon+1}\Ex{u_k}^{\varepsilon+1} (\zeta_{\tau_1,\tau_2}^\delta)'(t) \d x \d t + \iint_{\Omega_T} \Ex{u_k} \partial_t ((g + \tfrac{1}{k})^{\varepsilon} \zeta_{\tau_1,\tau_2}^\delta) \d x \d t
        \\
      \notag  &= \frac1\delta \int^{\tau_2}_{\tau_2-\delta} \int_\Omega \tfrac{1}{\varepsilon+1}\Ex{u_k}^{\varepsilon+1}\d x \d t 
        - \frac1\delta \int^{\tau_1 + \delta}_{\tau_1} \int_\Omega\tfrac{1}{\varepsilon+1}\Ex{u_k}^{\varepsilon+1}\d x \d t
        \\
      \notag  & \quad + \frac1\delta \int^{\tau_1 + \delta}_{\tau_1} \int_\Omega \Ex{u_k} (g + \tfrac{1}{k})^\varepsilon \d x \d t 
        - \frac1\delta \int^{\tau_2}_{\tau_2-\delta} \int_\Omega \Ex{u_k} (g + \tfrac{1}{k})^\varepsilon \d x \d t
        \\
      \notag  & \quad + \iint_{\Omega_T} \Ex{u_k} \partial_t ( g + \tfrac{1}{k})^\varepsilon \zeta_{\tau_1,\tau_2}^\delta \d x \d t 
        \\
       \notag &\xrightarrow[h\to 0]{} \frac1\delta \int^{\tau_2}_{\tau_2-\delta} \int_\Omega \tfrac{1}{\varepsilon+1}u_k^{\varepsilon+1}\d x \d t 
        - \frac1\delta \int^{\tau_1 + \delta}_{\tau_1} \int_\Omega \tfrac{1}{\varepsilon+1}u_k^{\varepsilon+1}\d x \d t
        \\
      \notag  & \qquad + \frac1\delta \int^{\tau_1 + \delta}_{\tau_1} \int_\Omega u_k (g + \tfrac{1}{k})^\varepsilon \d x \d t 
        - \frac1\delta \int^{\tau_2}_{\tau_2-\delta} \int_\Omega u_k (g + \tfrac{1}{k})^\varepsilon \d x \d t
        \\
       \notag & \qquad + \iint_{\Omega_T} u_k \partial_t ( g + \tfrac{1}{k})^\varepsilon \zeta_{\tau_1,\tau_2}^\delta \d x \d t.
    \end{align}   
    For the elliptic part we have, recalling Remark \ref{rem:A^k-rewritten},
    \begin{align}\label{convg:elliptic}
        &\iint_{\Omega_T} \Ex{A(x,t,u_k,\partial_1 u_k^{m_1},\dots, \partial_N u_k^{m_N})} \cdot \nabla \varphi \d x \d t 
        \\
       \notag &\xrightarrow[h \to 0]{} \iint_{\Omega_T} A(x,t,u_k,\partial_1 u_k^{m_1},\dots, \partial_N u_k^{m_N}) \cdot \nabla \big(u_k^{\varepsilon} - (g + \tfrac{1}{k})^{\varepsilon}\big) \zeta_{\tau_1,\tau_2}^\delta \d x \d t.
    \end{align}
Using the estimate \eqref{est:parabterm} in \eqref{mollified_weak_forM_*ased_on_cont_into_dual_space} and taking into account the limits \eqref{eq:vanishing_initial_time_term}, \eqref{convg:source}, \eqref{convg:I_a} and \eqref{convg:elliptic}  we end up with
\begin{align*}
    &\frac1\delta \int^{\tau_2}_{\tau_2-\delta} \int_\Omega \tfrac{1}{\varepsilon+1}u_k^{\varepsilon+1}\d x \d t 
         + \frac1\delta \int^{\tau_1 + \delta}_{\tau_1} \int_\Omega u_k (g + \tfrac{1}{k})^\varepsilon \d x \d t 
       + \iint_{\Omega_T} u_k \partial_t ( g + \tfrac{1}{k})^\varepsilon \zeta_{\tau_1,\tau_2}^\delta \d x \d t 
         \\
         &\quad + \iint_{\Omega_T} A(x,t,u_k,\partial_1 u_k^{m_1},\dots, \partial_N u_k^{m_N}) \cdot \nabla \big(u_k^{\varepsilon} - (g + \tfrac{1}{k})^{\varepsilon}\big) \zeta_{\tau_1,\tau_2}^\delta \d x \d t 
         \\
         &\leq \!\frac1\delta \!\hspace{-0.5mm}\int^{\tau_1 + \delta}_{\tau_1}\hspace{-2mm} \int_\Omega \tfrac{1}{\varepsilon+1}u_k^{\varepsilon+1}\d x \d t + \hspace{-0.5mm} \frac1\delta \int^{\tau_2}_{\tau_2-\delta} \int_\Omega u_k (g + \tfrac{1}{k})^\varepsilon \d x \d t +\iint_{\Omega_T} f (u_k^{\varepsilon} - (g + \tfrac{1}{k})^{\varepsilon}) \zeta_{\tau_1,\tau_2}^\delta \d x \d t. 
\end{align*}
We omit the first two terms on the left-hand side since they are non-negative, and note that all three terms on the right-hand side can be bounded by a constant $C_1$ independent of $k$ and $\delta$ since $g$ is bounded, $f$ is integrable and since the functions $u_k$ satisfy a uniform upper bound. Therefore, when passing to the limit $\delta \to 0$ we obtain:
\begin{align}\label{est:intermediate}
    \int^{\tau_2}_{\tau_1}\int_\Omega A(x,t,u_k,\partial_1 u_k^{m_1}, \dots, \partial_N u_k^{m_N}) \cdot \nabla u_k^{\varepsilon} \d x \d t 
   \leq C_1  -\int^{\tau_2}_{\tau_1}\int_\Omega u_k \partial_t ( g + \tfrac{1}{k})^\varepsilon \d x \d t &
   \\
  \notag + \int^{\tau_2}_{\tau_1}\int_\Omega A(x,t,u_k,\partial_1 u_k^{m_1}, \dots, \partial_N u_k^{m_N}) \cdot \nabla (g + \tfrac{1}{k})^{\varepsilon} \d x \d t.&
\end{align}
If $g\equiv 0$, the integrals on the right-hand side of \eqref{est:intermediate} vanish, so that the right-hand side is bounded by a constant. Otherwise, we are in the case $g\geq \varepsilon_0$ due to \eqref{cond:g}. For the first integral on the right-hand side of \eqref{est:intermediate} we then obtain the upper bound
\begin{align}\label{est:timeder-g}
    \Big|\int^{\tau_2}_{\tau_1}\int_\Omega u_k \partial_t ( g + \tfrac{1}{k})^\varepsilon \d x \d t\Big| \leq c\varepsilon_0^{\varepsilon-1}\iint_{\Omega_T}|\partial_t g| \d x \d t =: C_2,
\end{align}
due to the uniform upper bound for $(u_k)_{k = 1}^\infty$. Note that $C_2$ is finite due to the integrability of $\partial_t g$. The positive lower and upper bound for the function $u_k$ allows us to calculate using the Chain Rule:
\begin{align}\label{chainrule-consequences}
    \partial_j u_k^{\varepsilon} = \tfrac{\varepsilon}{m}u_k^{\varepsilon-m} \partial_j u_k^{m}, \qquad \partial_j u_k^{m_j} = \tfrac{m_j}{m} u_k^{m_j - m}\partial_j u_k^m.
\end{align}
Using these identities and the lower bound for the coefficients $a_j$ we can write
\begin{align}\label{elliptic-term-with-eps2}
\notag A(x,t,u_k,\partial_1 u_k^{m_1}, \dots, \partial_N u_k^{m_N}) \cdot \nabla u_k^{\varepsilon} &=\sum_{j = 1}^N a_j(x,t,u_k)|\partial_j u_k^{m_j}|^{p_j-2}\partial_j u_k^{m_j}  \partial_j u_k^{\varepsilon}
\\
 &\geq  c \varepsilon\sum_{j = 1}^N  u_k^{(m_j-m)(p_j-1) - m + \varepsilon}|\partial_j u_k^{m}|^{p_j}.
\end{align}
Condition \eqref{cond:m_j-closeness} guarantees that 
\begin{align*}
   (m_j-m)(p_j-1) - m < 0,
\end{align*}
so fixing a sufficiently small $\varepsilon \in (0,1]$ we can ensure that the exponents of $u_k$ appearing on the last line of \eqref{elliptic-term-with-eps2} are all non-positive. Therefore, using the uniform upper bound for the functions $u_k$, we obtain from \eqref{elliptic-term-with-eps2} the estimate
\begin{align}\label{est:elliptic-new}
    A(x,t,u_k,\partial_1 u_k^{m_1},\dots, \partial_N u_k^{m_N}) \cdot \nabla u_k^{\varepsilon} \geq \gamma \sum_{j=1}^N |\partial_j u_k^{m}|^{p_j}
\end{align}
for some $\gamma > 0$. For the integrand in the second integral on the right-hand side of \eqref{est:intermediate} we utilize the second identity in \eqref{chainrule-consequences}, the uniform upper bound for $(u_k)_{k = 1}^\infty$ and the uniform lower bound for $g$ appearing in \eqref{cond:g}, and Young's inequality to make the estimate 
\begin{align}\label{est:spaceder-g}
   \notag A(x,t,u_k,\partial_1 u_k^{m_1},\dots, \partial_N u_k^{m_N}) \cdot \nabla (g + \tfrac{1}{k})^{\varepsilon} &= \sum^N_{j=1} a_j(x,t,u_k)|\partial_j u_k^{m_j}|^{p_j-2}\partial_j u_k^{m_j} \varepsilon (g + \tfrac{1}{k})^{\varepsilon-1} \partial_j g
    \\
    &\leq c \varepsilon \varepsilon_0^{\varepsilon -1} \sum^N_{j=1} u_k^{(m_j-m)(p_j-1)}|\partial_j u_k^m|^{p_j-1}|\partial_j g|
    \\
    \notag & \leq c\sum^N_{j=1} |\partial_j u_k^m|^{p_j-1}|\partial_j g|
    \\
    \notag & \leq \frac{\gamma}{2}\sum^N_{j=1} |\partial_j u_k^m|^{p_j} + C_\gamma \sum^N_{j=1}|\partial_j g|^{p_j}.
\end{align}
Using the estimates \eqref{est:timeder-g}, \eqref{est:elliptic-new} and \eqref{est:spaceder-g} for the corresponding terms in \eqref{est:intermediate}, and taking $\tau_1 \to 0$, $\tau_2\to T$ we end up with 
\begin{align*}
    \frac{\gamma}{2} \iint_{\Omega_T} \sum_{j=1}^N |\partial_j u_k^{m}|^{p_j}\d x \d t \leq C_1 + C_2 + C_\gamma \iint_{\Omega_T} \sum_{j=1}^N |\partial_j g|^{p_j} \d x \d t,
\end{align*}
from which the first claim of the lemma follows since the right-hand is a finite constant independent of $k$ due to the integrability properties of the spatial derivatives of $g$ appearing in \eqref{cond:g}. The second claim follows from the first claim combined with the Chain Rule and the uniform upper bound for the functions $(u_k)_{k = 1}^\infty$.
\end{proof}

\begin{lem}\label{lem:monotone_conv_comp_princ}
The sequence $(u_k(x,t))^\infty_{k=1}$ is decreasing and convergent for a.e. $(x,t) \in \Omega_T$. The sequence $(u_k)^\infty_{k=1}$ converges in $L^q(\Omega_T)$ for all $1\leq q<\infty$.
\end{lem}
\begin{proof}{} 
We will use the comparison principle from Theorem \ref{thm:comparison_non-doubly-nonlinear}. Suppose $l\geq k$. Then both $u_k$ and $u_l$ satisfy the PDE 
\begin{align*}
    \partial_t u - \nabla \cdot A(x,t,u,\nabla u) = f,
\end{align*}
with
\begin{align*}
    A_j (x, t, u, \xi)  &:=  \,  a_j(x, t, u) \lvert m_j T(u) ^{m_j - 1} \xi_j \rvert ^{p_j - 2} m_j T(u) ^{m_j - 1} \xi_j, 
    \\
    T(u) &:= \min \{ L, \max\{u, \tfrac1l\}\},
\end{align*}
where $L$ is the uniform upper bound for the functions $u_k$ obtained in Lemma \ref{lemma:De_Giorgi_Approx_probs}. Since $u_k(0) \geq u_l(0)$ and since $u_k = 1/k + g \geq 1/l + g = u_l$ on $\partial \Omega \times (0,T)$, we have, by the aforementioned comparison principle, that $u_k \geq u_l$ a.e. in $\Omega_T$. For every pair $(k,l)$ with $l\geq k$ we can thus exclude a set $N_{k,l}$ of measure zero outside of which $u_k \geq u_l$ and thus the convergence is monotone outside $\cup N_{k,l}$ which is a set of measure zero. The uniform upper bound and the pointwise convergence a.e. also imply convergence in every $L^q(\Omega_T)$ for $1\leq q < \infty$ via the Dominated Convergence Theorem.
\end{proof}

\begin{rem}\label{rem:weak-convg}
We have seen that $(u_k)_{k = 1}^\infty$ converges pointwise and in $L^q(\Omega_T)$ to a bounded limit function $u$. Due to Lemma \ref{lem:grad_bdd_in_Lp} we may pass to a subsequence so that $(\partial_j u_k^m)^\infty_{k=1}$ converges weakly in $L^{p_j}(\Omega_T)$, and we may conclude that the limit is $\partial_j u^m$. Similarly, due to Lemma \ref{lem:grad_bdd_in_Lp}, we may assume also that $(\partial_j u_k^{m_j})^\infty_{k=1}$ converges weakly in $L^{p_j}(\Omega_T)$ to $\partial_j u^{m_j}$. Finally, the structure conditions show that the functions 
\begin{align*}
    \mathcal{A}_j^k := A_j(\cdot,\cdot,u_k,\partial_1 u_k^{m_1},\dots, \partial_N u_k^{m_N})
\end{align*}
are bounded in $L^{p_j'}(\Omega_T)$ which means that we may assume that they converge weakly in $L^{p_j'}(\Omega_T)$ to some limit $\mathcal{A}_j \in L^{p_j'}(\Omega_T)$.  To sum up we have the following convergence results.
\begin{align}
    \partial_j u_k^m &\xrightarrow[k\to \infty]{} \partial_j u^m \textnormal{ weakly in } L^{p_j}(\Omega_T),
    \\
    \partial_j u_k^{m_j} &\xrightarrow[k\to \infty]{} \partial_j u^{m_j} \textnormal{ weakly in } L^{p_j}(\Omega_T),
    \\
    A_j(\cdot,\cdot, u_k, \partial_1 u_k^{m_1},\dots, \partial_N u_k^{m_N}) &\xrightarrow[k\to \infty]{} \mathcal{A}_j \textnormal{ weakly in } L^{p_j'}(\Omega_T). \label{convg:A_j-weak}
\end{align}
\end{rem}
Due to \eqref{convg:A_j-weak} and Lemma \ref{lem:monotone_conv_comp_princ} we can pass to the limit in \eqref{eq:weak_form_u_k} and conclude that $u$ satisfies the condition
\begin{align}\label{eq:weak_form_for_u_with_mathcalA}
    \iint_{\Omega_T}  \mathcal{A}\cdot \nabla \varphi - u \partial_t \varphi \d x \d t = \iint_{\Omega_T} f \varphi \d x \d t,
\end{align}
for all $\varphi\in C^\infty_{\textnormal{o}}(\Omega_T)$ where $\mathcal{A} := (\mathcal{A}_1, \ldots, \mathcal{A}_N),$ following the notation introduced in Remark \ref{rem:weak-convg}.

\subsection{Boundary values, time continuity and initial data}\label{sec:boundary_val_time_cont_initial_data} 
We proceed to prove that $u$ satisfies the boundary condition in the sense of Definition \ref{def:prob-CD}. Due to the boundary condition appearing in \eqref{eq:approximative_problem}, and the lower bound for $u_k$ established in Lemma \ref{lem:lower_bound_uk}, we can write
\begin{align*}
    \tfrac1k \leq u_k = \tfrac1k + g + w_k,
\end{align*}
where $w_k \in L^{\bf p}(0,T;\overline W^{1,\bf p}_{\textnormal{o}}(\Omega))$. From this estimate we also see that $w_k + g \geq 0$. Take a sequence $(\varphi_j)\subset C^\infty(\Omega\times [0,T])$ of functions with compact support in space such that $\varphi_j \to w_k$ in $L^{\bf p}(0,T;W^{1,\bf p}(\Omega))$. Since $w_k$ is bounded we may also assume that the sequence $(\varphi_j)$ is bounded in the $L^\infty$-norm. Note that we also have
\begin{align*}
    0 \leq (g + \varphi_j)_+\xrightarrow[j\to \infty]{} (g+w_k)_+ = g+w_k \quad \mbox{in $L^{\bf p}(0,T;W^{1,\bf p}(\Omega))$.}
\end{align*}
Thus, 
\begin{align}\label{massaro}
    \tfrac1k \leq \tfrac1k + (g + \varphi_j)_+\xrightarrow[j\to \infty]{} \tfrac1k + g + w_k = u_k \quad \mbox{in $L^{\bf p}(0,T;W^{1,\bf p}(\Omega))$.}
\end{align}
 Due to the lower bound in \eqref{massaro} the quantity on the right-hand side of the inequality can be raised to the power $m$. The fact that $m\geq 1$ and that $(\varphi_j)$ is uniformly bounded then shows that 
\begin{align}\label{strange-convergence}
    \big(\tfrac1k + (g + \varphi_j)_+\big)^m \xrightarrow[j\to \infty]{} u_k^m \quad \mbox{in $L^{\bf p}(0,T;W^{1,\bf p}(\Omega))$.}
\end{align}
Since the functions $\varphi_j$ are compactly supported in space we see that 
\begin{align}\label{asdf}
    \big(\tfrac1k + (g + \varphi_j)_+\big)^m \in (\tfrac1k + g)^m + L^{\bf p}(0,T;\overline W^{1,\bf p}_{\textnormal{o}}(\Omega)).
\end{align}
Since the set appearing in \eqref{asdf} is closed we have due to the limit \eqref{strange-convergence} that 
\begin{align}\label{space_of_u_kbeta}
    u_k^m \in ( \tfrac1k + g)^m + L^{\bf p}(0,T;\overline W^{1,\bf p}_{\textnormal{o}}(\Omega)).    
\end{align}
Thus we have
\begin{align*}
L^{\bf p}(0,T;\overline W^{1,\bf p}_{\textnormal{o}}(\Omega)) \ni    u_k^m - ( \tfrac1k + g)^m = [u_k^m - g^m] + [g^m - ( \tfrac1k + g)^m].
\end{align*}
The expression in the first square brackets on the right-hand side converges weakly to $u^m - g^m$ in $L^{\bf p}(0,T;W^{1,\bf p}(\Omega))$, see also Remark \ref{rem:weak-convg}. The expression in the second square brackets converges to zero strongly in the limit $k\to \infty$, so that the sum of the two expressions converges weakly to $u^m - g^m$. Since the weak and strong closure of convex sets coincide, we conclude that
\begin{align*}
    u^m - g^m \in L^{\bf p}(0,T;\overline W^{1,\bf p}_{\textnormal{o}}(\Omega)),
\end{align*}
which means that $u$ satisfies the boundary conditions in the sense of Definition \ref{def:prob-CD}.
Thus, we can apply Lemma \ref{lem:time-cont} and conclude that $u$ is in $C([0,T]; L^{m+1}(\Omega))$. 

We move on to address the initial data. Let $\zeta \in C^\infty([0,T];\R)$ be a smooth function which takes the value $1$ near $0$ and vanishing near $T$. Using \eqref{eq:weak_form_for_u_with_mathcalA} with a test function of the form $\varphi = H_\delta(t)\zeta(t) \psi(x)$ where $\psi \in C^\infty_{\textnormal{o}}(\Omega)$ and the function $H_\delta$ was defined in \eqref{def:H_delta}, we have
\begin{align*}
    \frac{1}{\delta} \int^\delta_0 \int_\Omega u \psi \d x \d t =& \iint_{\Omega_T}\mathcal{A}\cdot \nabla \psi(x) H_\delta(t)\zeta(t) - f H_\delta(t)\zeta(t) \psi(x) \d x \d t\\& - \iint_{\Omega_T} u H_\delta(t)\zeta'(t) \psi(x) \d x \d t.
\end{align*}
Passing to the limit $\delta \to 0$ we end up with
\begin{align}\label{initial_val-mystery1}
    \int_\Omega u(x,0) \psi(x) \d x = & \iint_{\Omega_T}\mathcal{A}\cdot \nabla \psi(x) \zeta(t) - f \zeta(t) \psi(x) \d x \d t
    \\
    \notag & - \iint_{\Omega_T} u \zeta'(t) \psi(x) \d x \d t.
\end{align}
Using the same test function in the equation satisfied by $u_k$ we have
\begin{align*}
    \int_\Omega (u_0 + \tfrac1k) \psi(x) \d x = & \iint_{\Omega_T}A(x,t,u_k, \partial_1 u_k^{m_1},\dots, \partial_N u_k^{m_N})\cdot \nabla \psi(x) \zeta(t) - f \zeta(t) \psi(x) \d x \d t\\& - \iint_{\Omega_T} u \zeta'(t) \psi(x) \d x \d t.
\end{align*}
Passing to the limit $k\to \infty$ we obtain
\begin{align}\label{initial_val-mystery2}
    \int_\Omega u_0 \psi(x) \d x = & \iint_{\Omega_T}\mathcal{A} \cdot \nabla \psi(x) \zeta(t) - f \zeta(t) \psi(x) \d x \d t\notag \\& - \iint_{\Omega_T} u \zeta'(t) \psi(x) \d x \d t.
\end{align}
Noting that the right-hand sides of \eqref{initial_val-mystery1} and \eqref{initial_val-mystery2} are identical, and taking into account that $\psi \in C^\infty_{\textnormal{o}}(\Omega)$ is arbitrary, we conclude that $u(\cdot, 0) = u_0$.


\begin{lem}\label{lem:complicated-initial-val}
The functions $u_k$ satisfy the condition 
\begin{align*}
 \tfrac{1}{m+1}&\int_\Omega |u_k(x,T)|^{m+1}\d x - \tfrac{1}{m+1}\int_\Omega |u_k(x,0)|^{m+1}\d x 
 \\
 &= \iint_{\Omega_T} f(x,t)(u_k^m - (g(x,t)+\tfrac1k)^m) - \mathcal{A}^k(x,t)\cdot \nabla (u_k^m - (g(x,t)+\tfrac1k)^m)\d x \d t  
 \\
  &\quad -m \iint_{\Omega_T} (g(x,t) + \tfrac1k)^{m-1}\partial_t g(x,t) u_k(x,t)\d x \d t 
    \\&
   \quad  + \int_\Omega (g(x,T) + \tfrac1k)^m u_k(x,T)\d x \d t - \int_\Omega (g(x,0) + \tfrac1k)^m u_k(x,0)\d x \d t, 
\end{align*}
where $\mathcal{A}^k(x, t) = A(x,t,u_k,\partial_1 u_k^{m_1}, \dots, \partial_N u_k^{m_N})$.
\end{lem}
\begin{proof}{}
We extend $u_k$ to the time interval $[-T, 2T]$ by reflection, i.e. 
\begin{equation}
    \tilde{u}_k(t)= \begin{cases} u_k(-t)& \text{if } -T \leq t \leq 0,\\
    u_k(t) & \text{if } 0\leq t \leq T,\\
    u_k(2T-t) & \text{if } T \leq t \leq 2T.        
    \end{cases}
\end{equation} 
Similarly, consider extensions $\tilde{\mathcal{A}}^k$ and $\tilde f$ of the functions $\mathcal{A}^k$ and $f$ to the set $\Omega \times (-T,2T)$. In the following, we will denote the extended maps simply by $u_k$, $\mathcal{A}^k$ and $f$. Using the weak formulation \eqref{eq:weak_form_u_k} with a test function of the form $v\zeta^\delta_{t_1,t_2}$ where $v\in C^\infty_{\textnormal{o}}(\Omega)$ and $t_1 < t_2$ we have after and passing to the limit $\delta \to 0$:
\begin{equation}\label{eq:with_v}
  \int_{\Omega} (u_k(x,t_2)-u_k(x, t_1)) v(x) \d x = \int_{t_1}^{t_2}\int_{\Omega} (f(x,s)v(x) - \mathcal{A}^k(x,s) \cdot \nabla v(x)) \, \d x \d s.
\end{equation}
By approximation, the equation remains valid for $v$ in $L^2(\Omega)\cap \overline{W}^{1,\bf p}_{\textnormal{o}}(\Omega)$.
Let $t_1=t-h$ and $t_2=t$. We can, due to \eqref{space_of_u_kbeta}, choose $v= \big(u_k^m - (g+\tfrac1k)^m\big)(\cdot, t)$  and integrate over the time interval $[0,T]$ to obtain
\begin{align}\label{porky_burger} 
\iint_{\Omega_T} &(u_k(x,t)-u_k(x,t-h))(u_k^m(x,t) - (g(x,t)+\tfrac1k)^m)\, \d x \d t 
 \\
&= \int^T_0 \int^t_{t-h}\int_\Omega \big(f(x,s)(u_k^m(x,t) - (g(x,t)+\tfrac1k)^m)\nonumber \\& \quad-\mathcal{A}^k(x,s)\cdot (\nabla u_k^m(x,t) - \nabla (g(x,t)+\tfrac1k)^m)\big)\d x \d s \d t. \nonumber
\end{align}
By the convexity of the map $s\mapsto \tfrac{1}{m+1}s^{m+1}$, we have the pointwise estimate
\begin{equation*}
    \tfrac{1}{m+1}u_k(x,t)^{m+1} - \tfrac{1}{m+1}u_k(x,t-h)^{m+1}  \leq u_k^m(x,t)(u_k(x,t) - u_k(x,t-h)).
\end{equation*}
Using this estimate in \eqref{porky_burger} we have 
\begin{align}\label{monreale}
\tfrac{1}{m+1}&\int^T_{T-h} \int_\Omega u_k(x,t)^{m+1}\d x \d t - \tfrac{1}{m+1}\int^0_{-h} \int_\Omega u_k(x,t)^{m+1}\d x \d t
\\
\notag    &= \iint_{\Omega_T} \left(\tfrac{1}{m+1}u_k(x,t)^{m+1}(x,t) - \tfrac{1}{m+1}u_k(x,t-h)^{m+1}\right) \d x \d t 
    \\
\notag    &\leq \int^T_0 \int^t_{t-h}\int_\Omega \big(f(x,s)(u_k^m(x,t) - (g(x,t)+\tfrac1k)^m) \nonumber 
\\
\notag & \quad - \mathcal{A}^k(x,s)\cdot \nabla (u_k^m(x,t)-(g(x,t)+\tfrac1k)^m)\big)\d x \d s \d t \nonumber
    \\
\notag    & \quad + \iint_{\Omega_T}(g(x,t) + \tfrac1k)^m (u_k(x,t) - u_k(x,t-h)) \d x \d t.
\end{align}
We can now split the integral on the last row of \eqref{monreale} and perform a change of variables as in the second resulting integral as follows:
\begin{align*}
    &\iint_{\Omega_T}(g(x,t) + \tfrac1k)^m (u_k(x,t) - u_k(x,t-h)) \d x \d t 
    \\ &
    = \iint_{\Omega_T} (g(x,t) + \tfrac1k)^m u_k(x,t)\d x \d t - \int^{T-h}_{-h}\int_\Omega (g(x,t+h) + \tfrac1k)^m u_k(x,t)\d x \d t
    \\ &
    =\int_0^{T-h} \int_\Omega \left((g(x,t) + \tfrac1k)^m -(g(x,t+h) + \tfrac1k)^m\right) u_k(x,t)\d x \d t 
    \\ & \quad + \int^T_{T-h} \int_\Omega (g(x,t) + \tfrac1k)^m u_k(x,t)\d x \d t - \int^0_{-h}\int_\Omega (g(x,t+h) + \tfrac1k)^m u_k(x,t)\d x \d t.
\end{align*}
Dividing the last equation by $h$ we obtain
\begin{align*}
    & \frac1h \iint_{\Omega_T}(g(x,t) + \tfrac1k)^m (u_k(x,t) - u_k(x,t-h)) \d x \d t  
    \\ &
    = -\int^{T-h}_0 \int_\Omega \frac{\big((g(x,t+h) + \tfrac1k)^m - (g(x,t) + \tfrac1k)^m\big)}{h}u_k(x,t)\d x \d t
    \\ &
  \quad  + \frac1h \int^T_{T-h} \int_\Omega (g(x,t) + \tfrac1k)^m u_k(x,t)\d x \d t - \frac1h \int^0_{-h}\int_\Omega (g(x,t+h) + \tfrac1k)^m u_k(x,t)\d x \d t
    \\
    & \xrightarrow[h\to 0]{} -m \iint_{\Omega_T} (g(x,t) + \tfrac1k)^{m-1}\partial_t g(x,t) u_k(x,t)\d x \d t 
    \\&
    \quad + \int_\Omega (g(x,T) + \tfrac1k)^m u_k(x,T)\d x \d t - \int_\Omega (g(x,0) + \tfrac1k)^m u_k(x,0)\d x \d t.
\end{align*}
Thus, by dividing \eqref{monreale} by $h$ and passing to the limit  $h\to 0$ we end up with
\begin{align*}
    \tfrac{1}{m+1}&\int_\Omega u_k(x,T)^{m+1}\d x  - \tfrac{1}{m+1} \int_\Omega u_k(x,0)^{m+1}\d x 
    \\
    &\leq \int^T_0 \int_\Omega \big(f(x,t)(u_k^m(x,t) - (g(x,t)+\tfrac1k)^m)\\
    &\quad- \mathcal{A}^k(x,t) \cdot\nabla (u_k^m(x,t) - (g(x,t)+\tfrac1k)^m) \big)\d x \d t
    \\
    &\quad -m \iint_{\Omega_T} (g(x,t) + \tfrac1k)^{m-1}\partial_t g(x,t) u_k(x,t)\d x \d t 
    \\&
   \quad  + \int_\Omega (g(x,T) + \tfrac1k)^m u_k(x,T)\d x \d t - \int_\Omega (g(x,0) + \tfrac1k)^m u_k(x,0)\d x \d t.
\end{align*}
An estimate in the reverse direction can be obtained by analysing the quantity 
\begin{align*}
    \iint_{\Omega_T} &(u_k(x,t+h)-u_k(x,t))(u_k^m(x,t)- (g(x,t) + \tfrac1k)^m)\, \d x \d t
\end{align*}
in an analogous manner.
\end{proof}
We now prove an analogue of Lemma \ref{lem:complicated-initial-val} for $u$.
\begin{lem}\label{lem:complicated-initial-val_for_u}
    The function $u$ satisfies the condition
    \begin{align*}
 \tfrac{1}{m+1}&\int_\Omega |u(x,T)|^{m+1}\d x - \tfrac{1}{m+1}\int_\Omega |u_0|^{m+1}\d x 
 \\
 &= \iint_{\Omega_T} f (u^m - g^m)  - \mathcal{A} \cdot \nabla (u^m - g^m) \d x \d t -m \iint_{\Omega_T} g^{m-1}\partial_t g u \d x \d t 
    \\&
   \quad  + \int_\Omega g(x,T)^m u(x,T)\d x \d t - \int_\Omega g(x,0)^m u(x,0)\d x \d t.
\end{align*}
\end{lem}
\begin{proof}{}
The proof is similar to that of Lemma \ref{lem:complicated-initial-val}. We use \eqref{eq:weak_form_for_u_with_mathcalA} to obtain the equation
\begin{equation*}
  \int_{\Omega} (u(x,t_2)-u(x, t_1)) v(x) \d x = \int_{t_1}^{t_2}\int_{\Omega} (f(x,s)v(x) - \mathcal{A}(x,s) \cdot \nabla v(x)) \, \d x \d s,
\end{equation*}
which replaces \eqref{eq:with_v}. Here we can take $v=u^m(x,t) - g^m(x,t)$ and proceed as before. We thus replace $u_k$ by $u$, $g+\tfrac1k$ by $g$ and $\mathcal{A}^k$ by $\mathcal{A}$. In the argument we also need the previously established time continuity of $u$ as a map into $L^{m+1}(\Omega)$, see the text after \eqref{eq:weak_form_for_u_with_mathcalA}. Finally we also use the fact that $u(\cdot, 0) = u_0$.
\end{proof}

\vspace{3mm}
\subsection{Pointwise convergence of the gradient}\label{sec:Pointwise_convergence_of_the_gradient}
Using the the second identity of \eqref{chainrule-consequences} we can write 
\begin{align}\label{minty1}
  \notag  \big( &A(x,t,u_k, (\partial_j u_k^{m_j})^N_{j=1}) - A(x,t,u_k, (\partial_j u^{m_j})^N_{j=1})\big)\cdot (\nabla u_k^m - \nabla u^m)
 \\
  \notag &= \sum^N_{j=1} a_j(x,t,u_k)\tfrac{m}{m_j}u_k^{m - m_j}\big[|\partial_j u_k^{m_j}|^{p_j-2}\partial_j u_k^{m_j} - |\partial_j u^{m_j}|^{p_j-2}\partial_j u^{m_j}\big](\partial_j u_k^{m_j} - \partial_j u^{m_j})
   \\
\notag &\quad  + \sum^N_{j=1} a_j(x,t,u_k) \tfrac{m}{m_j}(u_k^{m -m_j} - u^{m - m_j})\big[|\partial_j u_k^{m_j}|^{p_j-2}\partial_j u_k^{m_j} - |\partial_j u^{m_j}|^{p_j-2}\partial_j u^{m_j}\big]\partial_j u^{m_j}
\\
 &=: h_k + i_0,
\end{align}
with the understanding that $i_0$ vanishes on the set where $u=0$, since $\partial_j u^{m_j}$ vanishes a.e. on this set. Thus the potentially negative exponent $m-m_j$ appearing on $u$ in this expression does not pose any problems. On the other hand, note that we can write
\begin{align}\label{minty2}
  \big( A&(x,t,u_k, (\partial_j u_k^{m_j})^N_{j=1}) - A(x,t,u_k, (\partial_j u^{m_j})^N_{j=1})\big)\cdot (\nabla u_k^m - \nabla u^m)\notag 
  \\ 
 & =
  A(x,t,u_k, (\partial_j u_k^{m_j})^N_{j=1}) \cdot \nabla (u_k^m - g^m) \notag 
 - A(x,t,u_k, (\partial_j u_k^{m_j})^N_{j=1})\cdot \nabla (u^m - g^m) \notag 
  \\
  & \quad - A(x,t,u_k, (\partial_j u^{m_j})^N_{j=1})\cdot  (\nabla u_k^m - \nabla u^m)\notag 
  \\
  &=: i_1-i_2 - i_3.
\end{align}
Combining \eqref{minty1} and \eqref{minty2} we can write
\begin{align}
    H_k& := \iint_{\Omega_T} h_k \d x \d t = \iint_{\Omega_T} i_1 \d x \d t - \iint_{\Omega_T}i_2\d x \d t - \iint_{\Omega_T}i_3 \d x \d t - \iint_{\Omega_T}i_0 \d x \d t
    \\
  \notag  &=: I_1 - I_2 - I_3 - I_0.
\end{align}
Note that $h_k \geq 0$. We have that
\begin{align}\label{i_0}
    i_0 
    =
    \sum^N_{j=1} 
    a_j(x,t,u_k) \tfrac{m}{m_j}
    &\Big[\Big(\frac{u}{u_k}\Big)^{m_j-m} - 1\Big] 
    \\
    & \times \big[|\partial_j u_k^{m_j}|^{p_j-2}\partial_j u_k^{m_j} - |\partial_j u^{m_j}|^{p_j-2}\partial_j u^{m_j}\big] u^{m-m_j}\partial_j u^{m_j}. \notag
\end{align}
By the weak convergence of $\partial_j u_k^m$ and the strong convergence of $u_k$, we can deduce that the Chain Rule is valid for the last factor:
\begin{align*}
    u^{m -m_j}\partial u^{m_j} = c \partial_j u^m \in L^{p_j}(\Omega_T).
\end{align*}
Due to the monotone pointwise convergence of $u_k$ established in Lemma \ref{lem:monotone_conv_comp_princ}  we have
\begin{align*}
    \frac{u}{u_k} \leq 1.
\end{align*}
Thus, H\"older's inequality and the Dominated Convergence Theorem allows us to calculate
\begin{align}\label{lim:I0}
    |I_0|  \leq  c\sum^N_{j=1} 
    \Big(
    \Big[ \iint_{\Omega_T} \Big|\Big(\frac{u}{u_k}\Big)^{m_j-m} - 1\Big|^{p_j}&|\partial_j u^m|^{p_j} \d x \d t\Big]^\frac{1}{p_j}
    \\
    &
    \times \Big[ \iint_{\Omega_T}\hspace{-0.7mm} |\partial_j u_k^{m_j}|^{p_j}+ |\partial_j u^{m_j}|^{p_j} \d x \d t\Big]^{\frac{p_j-1}{p_j}} \Big) \xrightarrow[k\to \infty]{} 0, \notag
\end{align}
where we also use Lemma \ref{lem:grad_bdd_in_Lp} to guarantee that the second integral has an upper bound independent of $k$. We further divide $I_3$ to see that
\begin{align}
I_3 &= \iint_{\Omega_T} \big(A(x,t,u_k, (\partial_j u^{m_j})^N_{j=1})-A(x,t,u, (\partial_j u^{m_j})^N_{j=1})\big) 
 (\nabla u_k^m - \nabla u^m) \d x \d t
\\
\notag &\quad + \iint_{\Omega_T} A(x,t,u, (\partial_j u^{m_j})^N_{j=1}) \cdot  (\nabla u_k^m - \nabla u^m) \d x \d t
\\
\notag &:= I_3^a + I_3^b.
\end{align}
To treat $I_3^a$ we use H\"older's inequality:
\begin{align}\label{est:I^a_3}
    |I^a_3| &\leq c\sum^N_{j=1}\Big[\iint_{\Omega_T}|a_j(x,t,u_k) - a_j(x,t,u)|^{p_j'}|\partial_j u^{m_j}|^{p_j}\d x \d t\Big]^\frac{p_j-1}{p_j}
    \\
\notag    & \hspace{1,3cm} \times \Big[\iint_{\Omega_T} |\partial_j u_k^m|^{p_j} + |\partial_j u^m|^{p_j}\d x \d t\Big]^\frac{1}{p_j}.
\end{align}
 The first integral on the right-hand side vanishes in the limit $k\to \infty$ due to the pointwise convergence of $u_k$ and the $u$-continuity and boundedness of $a_j$ and the second integral can again be bounded using Lemma \ref{lem:grad_bdd_in_Lp}. Thus we have confirmed that $I^a_3$ vanishes in the limit $k\to \infty$. By the weak convergence established in Remark \ref{rem:weak-convg} we also have that $I^b_3$ vanishes in the limit $k\to \infty$. Thus 
\begin{align}\label{lim:I3}
    \lim_{k\to 0}I_3 = 0.
\end{align} 
As for $I_2$, we have that
\begin{equation}\label{lim:I2}
   I_2=  \iint_{\Omega_T} i_2 \, \d x \d t \xrightarrow[k\to 0]{} \iint_{\Omega_T} \mathcal{A} \cdot \nabla (u^m - g^m) \, \d x \d t
\end{equation}
by weak convergence, see Remark \ref{rem:weak-convg}. In order to estimate $I_1$ we first divide this term as
\begin{align*}
    I_1 &= \iint_{\Omega_T} A(x,t,u_k, (\partial_j u_k^{m_j})^N_{j=1}) \cdot \nabla (u_k^m - (g+\tfrac1k)^m) \d x \d t
    \\
    &\quad + \iint_{\Omega_T} A(x,t,u_k,(\partial_j u_k^{m_j})^N_{j=1}) \cdot \nabla ((g+\tfrac1k)^m - g^m) \d x \d t
    \\
    &=: I_1^a + I_1^b.
\end{align*}
Note that $I_1^b$ vanishes in the limit $k\to \infty$ since, by H\"older's inequality,
\begin{align*}
    |I_1^b| \leq \sum^N_{j=1} \Big(\iint_{\Omega_T}|\mathcal{A}_k^j(x,t)|^{p_j'}\d x \d t\Big)^\frac{1}{p_j'}\Big(\iint_{\Omega_T}m^{p_j}|(g+\tfrac1k)^{m-1} - g^{m-1}|^{p_j}|\partial_j g|^{p_j} \d x \d t\Big)^\frac{1}{p_j},
\end{align*}
and the integrals involving $\mathcal{A}_k^j$ remain bounded as $k\to \infty$, whereas the other integrals vanish in the limit due to space of $g$ and the Dominated Convergence Theorem. In order to analyse $I^a_1$ we use Lemma \ref{lem:complicated-initial-val} to see that
\begin{align}\label{I_a-expr}
  \notag  I_1^a   = \iint_{\Omega_T} f(x,t)&(u_k^m -  (g(x,t)+\tfrac1k)^m)\d x \d t  
 -m \iint_{\Omega_T} (g(x,t) + \tfrac1k)^{m-1}\partial_t g(x,t) u_k(x,t)\d x \d t 
     \\
     & \quad  + \int_\Omega (g(x,T) + \tfrac1k)^m u_k(x,T)\d x \d t - \int_\Omega (g(x,0) + \tfrac1k)^m u_k(x,0)\d x \d t
  \\ \notag
  &\quad +\tfrac{1}{m+1}\int_\Omega |u_k(x,0)|^{m+1}\d x - \tfrac{1}{m+1}\int_\Omega |u_k(x,T)|^{m+1}\d x.
\end{align}
From the energy estimates obtained in Lemma \ref{lem:energy_estimates_approx_probs}, we see that the sequence $(u_k(\cdot,T))^\infty_{k=1}$ is bounded in $L^{m+1}(\Omega)$, therefore a subsequence converges weakly to some $w \in L^{m+1}(\Omega)$. Utilizing the equation for $u_k$ we can write, for any $\varphi \in C^\infty_{\textnormal{o}}(\Omega)$,
\begin{align*}
    \int_\Omega u_k(x,T) \varphi(x)\d x =& \int_\Omega (u_0 + \tfrac{1}{k})\varphi \d x - \iint_{\Omega_T} A(x,t,u_k, \partial_1 u_k^{m_1},\dots , \partial_N u_k^{m_N}) \cdot \nabla \varphi \d x \d t 
    \\
    &+ \iint_{\Omega_T} f(x,t)\varphi(x) \d x \d t.
\end{align*}
Passing to the limit in the previous equation we end up with
\begin{align*}
    \int_\Omega w \varphi \d x \d t = \int_\Omega u_0 \varphi \d x - \iint_{\Omega_T}\mathcal{A}\cdot \nabla \varphi + f\varphi \d x \d t = \int_\Omega u(x,T)\varphi \d x,
\end{align*}
due to the equation \eqref{eq:weak_form_for_u_with_mathcalA} satisfied by $u$ and the fact that $u(0) = u_0$. Since $\varphi$ is arbitrary, we have confirmed that $u(T) = w$. Then, from the weak convergence we have 
\begin{align*}
    \liminf_{k\to \infty} \int_\Omega u_k^{m +1}(x,T) \d x \geq \int_\Omega u(x,T)^{m +1} \d x.
\end{align*}
The convergence of $I_1^b$, the weak convergence of $u_k(\cdot, T)$ and the strong convergence of $u_k$ combined with \eqref{I_a-expr} show that
\begin{align}\label{limsup:I1}
    \limsup_{k\to \infty} I_1 &\leq - \tfrac{1}{m+1}\int_\Omega |u(x,T)|^{m+1}\d x + \tfrac{1}{m+1}\int_\Omega |u_0|^{m+1}\d x + \iint_{\Omega_T} f (u^m - g^m) \d x \d t
    \notag \\
  & \quad +\int_\Omega g(x,T)^m u(x,T)\d x \d t - \int_\Omega g(x,0)^m u(x,0)\d x \d t -m \iint_{\Omega_T} g^{m-1}\partial_t g u \d x \d t 
    \notag  \\
    &= \iint_{\Omega_T}\mathcal{A}\cdot \nabla (u^m - g^m) \d x \d t,
\end{align}
where in the last step we use Lemma \ref{lem:complicated-initial-val_for_u}.
Combining the estimates \eqref{lim:I0}, \eqref{limsup:I1}, \eqref{lim:I2}, \eqref{lim:I3} for $I_0$, $I_1$, $I_2$ and $I_3$ respectively, we have
\begin{align*}
    \limsup_{k\to\infty} H_k &\leq \limsup_{k\to \infty} I_1 - \lim_{k\to \infty} I_2 - \lim_{k\to \infty} I_3 - \lim_{k\to \infty} I_0 = 0.
\end{align*}
 By the non-negativity of $h_k$ this confirms that $h_k \to 0$ in $L^1(\Omega_T)$. Hence, a subsequence of $(h_k)$ convergences pointwise a.e.~to zero. The lower bound for $a_j$, the non-positivity of the exponents $m - m_j$ in \eqref{minty1} and the uniform upper bound for $(u_k)_{k = 1}^\infty$ show that for some $c>0$,
\begin{align*}
      c \sum^N_{j=1}\big[|\partial_j u_k^{m_j}|^{p_j-2}\partial_j u_k^{m_j} - |\partial_j u^{m_j}|^{p_j-2}\partial_j u^{m_j}\big](\partial_j u_k^{m_j} - \partial_j u^{m_j}) \leq h_k \xrightarrow[k\to \infty]{} 0,
\end{align*}
and hence $\partial_j u_k^{m_j}$ converges pointwise a.e.~to $\partial_j u^{m_j}$. This combined with the a.e. pointwise convergence of $u_k$ to $u,$ and the fact that $A$ is a Carathéodory vector field, imply that
\begin{align}\label{eq:pw_convergence_of_A_j}
    \mathcal{A}^k(x, t) =  A(x, t, u_k(x, t), (\partial_j u_k^{m_j}(x, t))_{j = 1}^N) \xrightarrow[k \to \infty]{} A(x, t, u(x, t), (\partial_j u^{m_j}(x, t))_{j = 1}^N), 
\end{align}
a.e.~in $\Omega_T.$ By definition $\mathcal{A}_j$ is the weak limit in $L^{p_j^\prime} (\Omega_T)$ of the coordinate functions $\mathcal{A}_j^k$ (see Remark \ref{rem:weak-convg}), so Mazur's Lemma guarantees that a sequence $(\beta^j_l)_{l=1}^\infty$ of convex combinations of $(\mathcal{A}_j^k)^\infty_{k=1}$ converges to $\mathcal{A}_j$ strongly in $L^{p_j^\prime}(\Omega_T)$. Passing to another subsequence we may assume that $(\beta^j_l)^\infty_{l=1}$ converges pointwise a.e.~to $\mathcal{A}_j$. However, by \eqref{eq:pw_convergence_of_A_j} the sequence $(\beta^j_l)^\infty_{j=1}$ also converges a.e.~to $A_j(\cdot, \cdot, u, (\partial_j u^{m_i})_{i = 1}^N)$, and thus $\mathcal{A}$ coincides with $A(\cdot, \cdot, u, (\partial_i u^{m_i})_{i = 1}^N)$. In light of \eqref{eq:weak_form_for_u_with_mathcalA} this confirms that $u$ satisfies the original equation.
\qed

\section{Comparison principle}\label{sec:comparison}
This section is devoted to the study of the  comparison principle for the doubly nonlinear equation \eqref{eq:diffusion}. 
We note that the sub- and super-solutions defined in Definition \ref{def:subsupersol} satisfy estimates involving the Steklov average.
\begin{lem}\label{lem:steklov_est_for_subsuper}
Let $u$ be a sub $($super$)$-solution. Then the following estimate is true:
\begin{align*}
    \int^{t_2}_{t_1}\int_{\Omega} \phi \partial_t [u]_{\bar h} + \big[A(x,\cdot,u,\partial_1 u^{m_1},\dots, \partial_N u^{m_N})\big]_{\bar h} \cdot \nabla \phi \,\d x \d t \leq (\geq)\int^{t_2}_{t_1}\int_{\Omega} [f]_{\bar h} \phi \d x \d t
\end{align*}
for all $\phi \in L^{\bf p}(0,T; \overline{W}^{1,{\bf p}}_{\textnormal{o}}(\Omega)) \cap L^\infty(\Omega_T)$ and all $0\leq t_1 < t_2 \leq T$.
\end{lem}

 \noindent The proof of the following result utilizes methods from \cite[Lemma A.1]{BoeDuLi}  and \cite[Lemma 2.4]{BoeStru}. 

\begin{lem}\label{lem:max_of_subsol_is_subsol}
Let $u$ be a sub-solution of \eqref{eq:diffusion} in the sense of Definition \ref{def:subsupersol}. Suppose also $u^m \in L^{\bf p}(0,T; W^{1, {\bf p}}(\Omega))$. Then for any $k>0$, the function $v:=\max\{u,k\}$ satisfies  
\begin{align}\label{chioggia}
\iint_{\Omega_T} \sum^N_{j=1} a_j(x,t,v)|\partial_j v^{m_j}|^{p_j-2}\partial_j v^{m_j} \partial_j \eta -  v\partial_t \eta \d x\d t \leq \iint_{\Omega_T} f\chi_{\{u>k\}} \eta \d x \d t,
\end{align}
for all non-negative $\eta \in C^\infty_{\textnormal{o}}(\Omega_T)$. That is, $v$ is a sub-solution to the original equation with the modified right-hand side $f\chi_{\{u>k\}}$.
\end{lem}
\begin{proof}{}
   We use the estimate \eqref{eq:weak_formsupersub} for sub-solutions with the test function
   \begin{align*}
   \varphi = \eta \varphi_h, \quad \varphi_h = \frac{(\Exx{u^m} - k^m)_+}{(\Exx{u^m} - k^m)_+ + \mu},
   \end{align*}
where $\Exx{\cdot}$ denotes the reversed exponential time mollification as defined in \eqref{def:mollrev} and $\mu>0$. In the parabolic term we can estimate
\begin{align}\label{Loacker}
    u \partial_t \varphi &=\Exx{u^m}^\frac1m\partial_t \varphi + \big(u - \Exx{u^m}^\frac1m\big) \partial_t \varphi \notag \\
    &= \Exx{u^m}^\frac1m\partial_t \varphi +  \big(u - \Exx{u^m}^\frac1m\big) \varphi_h \partial_t \eta + \big(u - \Exx{u^m}^\frac1m\big)\partial_t \varphi_h \eta 
    \notag\\
    &\leq \Exx{u^m}^\frac1m\partial_t \varphi +  \big(u - \Exx{u^m}^\frac1m\big) \varphi_h \partial_t \eta,
\end{align}
where in the last step we use Lemma \ref{lem:expmolproperties} \ref{expmol2} to calculate
\begin{align*}
    \partial_t \varphi_h = \frac{\mu\partial_t (\Exx{u^m} - k^m)_+}{\big( (\Exx{u^m} - k^m)_+ + \mu \big)^2} = \frac{\mu \chi_{\{\Exx{u^m} > k \}}\partial_t \Exx{u^m}}{\big( (\Exx{u^m} - k^m)_+ + \mu \big)^2} = \frac{\mu \chi_{\{\Exx{u^m} > k \}} \tfrac1h (\Exx{u^m} - u^m)}{\big( (\Exx{u^m} - k^m)_+ + \mu \big)^2},
\end{align*}
to see that the term omitted in the estimate \eqref{Loacker} is non-positive. Thus, we end up with   
\begin{align}\label{chelsea}
&\iint_{\Omega_T} \sum^N_{j=1} a_j(x,t,u)|\partial_j u^{m_j}|^{p_j-2}\partial_j u^{m_j} \partial_j \varphi -   \Exx{u^m}^\frac1m\partial_t \varphi -  \big(u - \Exx{u^m}^\frac1m\big) \varphi_h \partial_t \eta   \d x\d t 
\\
\notag &\quad \leq  \iint_{\Omega_T} f \varphi \d x \d t.
\end{align}
We will perform some estimates and eventually we will pass to the limit as $h\to 0$. In the second term on the left-hand a limit procedure and two integrations by parts show that
\begin{align*}
    \iint_{\Omega_T} \Exx{u^m}^\frac1m\partial_t \varphi \d x \d t & = \lim_{\varepsilon\to 0} \iint_{\Omega_T} (\Exx{u^m}+\varepsilon)^\frac1m\partial_t \varphi \d x \d t
    \\
    &= - \lim_{\varepsilon\to 0} \iint_{\Omega_T} \tfrac1m (\Exx{u^m} + \varepsilon)^{\frac1m-1} \partial_t \Exx{u^m}\varphi \d x \d t 
    \\
    &= - \lim_{\varepsilon\to 0}\iint_{\Omega_T} \tfrac1m (\Exx{u^m}+\varepsilon)^{\frac1m-1} \frac{(\Exx{u^m} - k^m)_+}{(\Exx{u^m} - k^m)_+ + \mu} \partial_t \Exx{u^m}   \eta \d x \d t
    \\
    &= - \lim_{\varepsilon\to 0} \iint_{\Omega_T} \partial_t F_\varepsilon(\Exx{u^m}, k, \mu)  \eta \d x \d t
    = \iint_{\Omega_T} F_0(\Exx{u^m}, k, \mu)  \partial_t \eta \d x \d t,
\end{align*}
where
\begin{align*}
    F_\varepsilon(\theta,k,\mu) := k+ \int_{k^{m}}^\theta \tfrac{1}{m} (s+\varepsilon)^{\frac{1}{m}-1} \frac{(s-k^{m})_+}{(s-k^{m})_+ + \mu} \d s, \quad \varepsilon \geq 0.
\end{align*}
For the last term on the left-hand side of \eqref{chelsea} we note that
\begin{align}
   \lim_{h\to 0} \iint_{\Omega_T} \big(u - \Exx{u^m}^\frac1m\big) \varphi_h \partial_t \eta   \d x\d t = 0.
\end{align}
Noting that
\begin{align*}
    \partial_j \varphi = \partial_j \eta \varphi_h + \eta \partial_j \varphi_h= \partial_j \eta \varphi_h + \eta \frac{\mu \chi_{\{\Exx{u^m} > k \}}\partial_j \Exx{u^m}}{\big( (\Exx{u^m} - k^m)_+ + \mu \big)^2},
\end{align*}
we can calculate
\begin{align}\label{nonono}
   &\iint_{\Omega_T} \sum^N_{j=1} a_j(x,t,u)|\partial_j u^{m_j}|^{p_j-2}\partial_j u^{m_j} \partial_j \varphi \d x \d t\notag 
   \\
   & = \iint_{\Omega_T} \sum^N_{j=1}  a_j(x,t,u)|\partial_j u^{m_j}|^{p_j-2}\partial_j u^{m_j} \partial_j \eta \varphi_h \d x \d t \notag 
   \\ 
   &\quad + \iint_{\Omega_T} \sum^N_{j=1} a_j(x,t,u)|\partial_j u^{m_j}|^{p_j-2}\partial_j u^{m_j}   \eta \frac{\mu \chi_{\{\Exx{u^m} > k \}}\partial_j \Exx{u^m}}{\big( (\Exx{u^m} - k^m)_+ + \mu \big)^2}\d x \d t 
   \notag 
   \\
   &\xrightarrow[h\to 0]{} \iint_{\Omega_T} \sum^N_{j=1}  a_j(x,t,u)|\partial_j u^{m_j}|^{p_j-2}\partial_j u^{m_j} \partial_j \eta \frac{(u^m - k^m)_+}{(u^m - k^m)_+ + \mu}  \d x \d t
   \notag
   \\
   &\quad\quad + \iint_{\Omega_T} \sum^N_{j=1} a_j(x,t,u)|\partial_j u^{m_j}|^{p_j-2}\partial_j u^{m_j}  \eta \frac{\mu \chi_{\{ u^m > k \}}\partial_j u^m}{\big( (u^m - k^m)_+ + \mu \big)^2}\d x \d t.
\end{align}
The last term is non-negative since, by the Chain Rule $\partial_j u^{m_j}\partial_j u^m \geq 0$.
Taking into account these limits and estimates we end up with
\begin{align*}
    &\iint_{\Omega_T} \sum^N_{j=1}  a_j(x,t,u)|\partial_j u^{m_j}|^{p_j-2}\partial_j u^{m_j} \partial_j \eta \frac{(u^m - k^m)_+}{(u^m - k^m)_+ + \mu} - F_0(u^m,k,\mu) \partial_t \eta \d x \d t  
    \\
    & \quad \leq \iint_{\Omega_T} f \eta \frac{(u^m - k^m)_+}{(u^m - k^m)_+ + \mu}  \d x \d t. 
\end{align*}
Finally, passing to the limit $\mu \to 0$ we end up with \eqref{chioggia}.
\end{proof}

The following result is a preliminary version of the comparison principle where we assume that the subsolution is bounded from below by a positive constant on the whole domain $\Omega_T$.
\begin{prop}\label{prop:comp_principle_special_case}
    Suppose that $u$ is a weak subsolution with right-hand side $f_u$ and that $v$ is a weak supersolution with right-hand side $f_v$ in the sense of Definition \ref{def:subsupersol}. Assume furthermore that $u,v \in C([0,T]; L^1(\Omega))$ and that $u^m$ and  $v^m$ belong to $L^{\bf p}(0,T; W^{1, {\bf p}}(\Omega))$.
Assume that 
\begin{equation}\label{cases}
    u \geq \varepsilon \quad \mbox{a.e. in $\Omega_T$}
\end{equation}
for some $\varepsilon>0$. If $v\geq u$ on $\partial \Omega \times (0,T)$ in the sense of Definition \ref{def:boundary-value-ineq}, we have
\begin{equation}\label{est:comp_principle_special_case}
\int_{\Omega} (u-v)_+ (x,t_2) \d x \leq \int_{\Omega} (u-v)_+ (x,t_1) \d x + \iint_{\Omega \times[t_1,t_2]\cap \{v<u\}}(f_u - f_v) \d x \d t,
\end{equation}
for every $0\leq t_1<t_2\leq T$.
\end{prop}
   
\begin{proof}{}
Let 
\begin{equation}\label{phi}
\phi = H_\delta(u-v), \quad \mbox{where $0<\delta \leq \min\{1, \tfrac{\varepsilon}{2}\}.$}
\end{equation}
We prove that $\phi$ is an admissible test function in the estimates of Lemma \ref{lem:steklov_est_for_subsuper}. Due to the lower bound for $u$ we can see by the Chain Rule and the function space of $u$ that $u$ itself has a gradient. We cannot draw the same conclusion for $v$, however $\hat v:= \max\{v,\tfrac{\varepsilon}{2}\}$ has a gradient due to the Chain Rule. Noting that in the range of the parameters $\delta$ and $\varepsilon$ we have 
\begin{align}\label{partial_jphi}
    \phi = H_\delta(u - v) &= H_\delta(u - \hat v),
   \nonumber \\
    \partial_j \phi &= \delta^{-1}\chi_{\{0< u-\hat v < \delta\}} \partial_j(u-\hat v),
\end{align}
we see that $\partial_j \phi$ is in $L^{p_j}(\Omega_T)$ as required. To see that we can also approximate $\phi$ by smooth functions with compact support in space, we recall that $u^m$ has the representation
\begin{align*}
    u^m = v^m + w + \psi,
\end{align*}
where $w$ in $L^{\bf p}(0,T; W^{1, {\bf p}}(\Omega))$, $w\leq 0$, and $\psi$ in $L^{\bf p}(0,T; \overline{W}^{1, {\bf p}}_{\textnormal{o}}(\Omega))$, see the comment after Definition \ref{def:boundary-value-ineq}. Thus we can pick a sequence $(\psi_k) \subset C^\infty(\Omega\times(0,T))$ with $\supp \psi_k \subset K_k \times [0,T]$ for some $K_k \Subset \Omega$ converging to $\psi$ in $L^{\bf p}(0,T; W^{1, {\bf p}}(\Omega))$. We define
\begin{align*}
    u_k &:= \max\{ (\tfrac{\varepsilon}{2})^m, v^m + w + \psi_k\}^\frac1m,
    \\
    \phi_k &:= H_\delta(u_k - \hat v).
\end{align*} 
Note that $\phi_k$ vanishes outside of $K_k\times[0,T]$. Thus, we can approximate $\phi_k$ by smooth functions that are compactly supported in space by convolution with standard mollifiers. To conclude that $\phi$ is a valid test function it is thus sufficient to show that $\phi$ can be approximated by $\phi_k$ in the norm of $L^{\bf p}(0,T; W^{1, {\bf p}}(\Omega))$. Due to the positive lower bound for $u_k$, we see that $u_k$ has a gradient and that $\partial_j u_k$ converges to $\partial_j u$ in $L^{p_j}(\Omega_T)$. By passing to a subsequence we may also assume that $u_k$ converges pointwise a.e. to $u$, and by the boundedness of $H_\delta$ we can then conclude the convergence of $\phi_k$ to $\phi$ in every $L^q(\Omega_T)$, $q< \infty$. Finally, we can write the partial derivatives as
\begin{align*}
    \partial_j \phi_k = \delta^{-1}\chi_{\{0< u_k - \hat v < \delta\}} \partial_j(u_k - \hat v).
\end{align*}
Comparing this with the expression for $\partial_j \phi$ in \eqref{partial_jphi} we can deduce the convergence in $L^{p_j}(\Omega_T)$ using the aforementioned convergences of $u_k$ and $\partial_j u_k$. 

Applying Lemma \ref{lem:steklov_est_for_subsuper} to $u$ and $v$ with the test function $\phi$ and subtracting the two estimates we have
\begin{align}\label{uv_ests_combined}
    &\int^{t_2}_{t_1}\int_{\Omega} \phi \partial_t [u-v]_{\bar h} + \big[A(x,\cdot,u,(\partial_j u^{m_j})^N_{j=1}) - A(x,\cdot,v,(\partial_j v^{m_j})^N_{j=1})\big]_{\bar h} \cdot \nabla \phi \,\d x \d t 
    \\
   \notag & \leq \int^{t_2}_{t_1}\int_{\Omega} [f_u - f_v]_{\bar h} \phi \d x \d t.
\end{align}
  By Lemma \ref{lem:almost-heaviside-est} with $f=u-v$, we have that 
  \begin{align*}
   \int^{t_2}_{t_1}\int_{\Omega} \phi \partial_t [u-v]_{\bar h}\d x \d t &= \int^{t_2}_{t_1}\int_{\Omega} \partial_t [u-v]_{\bar h} H_{\delta}(u-v) \d x \d t 
   \geq \int^{t_2}_{t_1}\int_{\Omega} \partial_t [G_{\delta}(u-v)]_{\bar h} \d x \d t
   \\
   &= \int_{\Omega\times\{t_2\}}[G_{\delta}(u-v)]_{\bar h} \d x  - \int_{\Omega\times\{t_1\}} [G_{\delta}(u-v)]_{\bar h} \d x,
  \end{align*}
  and combining this estimate with \eqref{uv_ests_combined} we end up with
  \begin{align*}
     \int_{\Omega\times\{t_2\}}[G_{\delta}(u-v)]_{\bar h} \d x &\leq     \int^{t_2}_{t_1}\int_{\Omega} \big[A(x,\cdot,v,(\partial_j v^{m_j})^N_{j=1})-A(x,\cdot,u,(\partial_j u^{m_j})^N_{j=1}) \big]_{\bar h} \cdot \nabla \phi \,\d x \d t \nonumber \\
     &\quad + \int^{t_2}_{t_1}\int_{\Omega} [f_u - f_v]_{\bar h} H_{\delta}(u-v) \d x \d t +\int_{\Omega\times\{t_1\}} [G_{\delta}(u-v)]_{\bar h} \d x .\nonumber 
  \end{align*}
Passing to the limit $h\to 0$ in the integrals involving $G_\delta$ is justified due to the time continuity of $u$ and $v$. In the other terms we can pass to the limit $h\to 0$ due to H\"older's inequality. Thus, we have
\begin{align}\label{est:scirocco}
     \int_{\Omega\times\{t_2\}} G_{\delta}(u-v) \d x &\leq     \int^{t_2}_{t_1}\int_{\Omega} \big( A(x,t,v,(\partial_j v^{m_j})^N_{j=1})-A(x,t,u,(\partial_j u^{m_j})^N_{j=1})\big) \cdot \nabla \phi \,\d x \d t \nonumber 
     \\
     &\quad + \int^{t_2}_{t_1}\int_{\Omega} (f_u - f_v) H_{\delta}(u-v) \d x \d t +\int_{\Omega\times\{t_1\}} G_{\delta}(u-v) \d x .
  \end{align}
We split the elliptic term as follows
\begin{align}\label{mangiatorella}
&\int^{t_2}_{t_1}\int_{\Omega} \left(A(x,t,v,(\partial_j v^{m_j})^N_{j=1})-A(x,t,u,(\partial_j u^{m_j})^N_{j=1})\right) \cdot \nabla \phi \,\d x \d t 
\\
&=   \int^{t_2}_{t_1}\int_{\Omega} \left(A(x,t,v,(\partial_j v^{m_j})^N_{j=1})- A(x,t,u,(\partial_j v^{m_j})^N_{j=1})\right) \cdot \nabla \phi \,\d x \d t 
\nonumber\\
\nonumber&\quad\quad +  \int^{t_2}_{t_1}\int_{\Omega} \left(A(x,t,u,(\partial_j v^{m_j})^N_{j=1})- A(x,t,u,(\partial_j u^{m_j})^N_{j=1})\right) \cdot \nabla \phi \,\d x \d t.
\end{align}
Using the precise expression for $\partial_j \phi$ established in \eqref{partial_jphi}, the expression for $A$ and the Lipschitz continuity \eqref{cond:lip_cont} of the coefficients $a_j$, we can estimate the first integral on the right hand side of \eqref{mangiatorella} as follows
\begin{align}\label{est:accardi}
  &\Big| \int^{t_2}_{t_1}\int_{\Omega} \left(A(x,t,v,(\partial_j v^{m_j})^N_{j=1})-A(x,\cdot,u,(\partial_j v^{m_j})^N_{j=1})\right) \cdot \nabla \phi  \,\d x \d t \Big|
  \\
  & \leq \sum_{j=1}^N \int^{t_2}_{t_1}\hspace{-1mm}\int_{\Omega}  |a_i(x,t,v) - a_i(x,t,u)||\partial_j v^{m_j}|^{p_j-1} 
  \chi_{\{0< u-\hat{v}<\delta\}}\tfrac1\delta |\partial_j u-\partial_j \hat{v}|\, \d x\d t \nonumber \\
  &\leq c \sum_{j=1}^N \int^{t_2}_{t_1}\int_{\Omega} |u-v||\partial_j v^{m_j}|^{p_j-1} 
  \chi_{\{0< u-\hat{v}<\delta\}}\tfrac1\delta |\partial_j u-\partial_j \hat{v}|\, \d x\d t
  \nonumber
  \\
  \nonumber & \leq c \sum_{j=1}^N \int^{t_2}_{t_1}\int_{\Omega}|\partial_j v^{m_j}|^{p_j-1} 
  \chi_{\{0< u-\hat{v}<\delta\}} |\partial_j u-\partial_j \hat{v}|\, \d x\d t, 
\end{align}
where we use the fact that $v=\hat v$ on the set where $u-\hat v < \delta$. The expression on the last line of \eqref{est:accardi} disappears in the limit $\delta \to 0$ due to the Dominated Convergence Theorem. 

\noindent The second integral on the right-hand side of \eqref{mangiatorella} can be written as
\begin{align}\label{fontalba}
    &\int^{t_2}_{t_1}\int_{\Omega} \left(A(x,t,u,(\partial_j v^{m_j})^N_{j=1})- A(x,t,u,(\partial_j u^{m_j})^N_{j=1})\right) \cdot \nabla \phi \,\d x \d t \\
    \nonumber &  =\sum_{j=1}^N \int^{t_2}_{t_1}\int_{\Omega} a_j(x,t,u)(|\partial_j v^{m_j}|^{p_j-2}\partial_j v^{m_j} - |\partial_j u^{m_j}|^{p_j-2}\partial_j u^{m_j})H_\delta'(u- \hat v) (\partial_j u - \partial_j \hat v) \d x \d t.
    \end{align}
By the Chain Rule, we have
\begin{align*}
    \partial_j u - \partial_j \hat v &= \tfrac{1}{m_j}(u^{1-m_j}\partial_j u^{m_j} - \hat{v}^{1-m_j}\partial_j \hat{v}^{m_j})
    \\
    &= \tfrac1{m_j} u^{1-m_j}(\partial_j u^{m_j} - \partial_j \hat{v}^{m_j}) + \tfrac1{m_j}(u^{1-m_j} - \hat{v}^{1-m_j}) \partial_j \hat{v}^{m_j}.
\end{align*}
Using this calculation we can split the right-hand side of  \eqref{fontalba} as follows
\begin{align}\label{panna}
    &\sum_{j=1}^N \int^{t_2}_{t_1}\int_{\Omega} a_j(x,t,u)(|\partial_j v^{m_j}|^{p_j-2}\partial_j v^{m_j} - |\partial_j u^{m_j}|^{p_j-2}\partial_j u^{m_j})
    \\ 
   \nonumber &\hspace{2cm}\times H_\delta'(u- \hat v)  \tfrac1{m_j} u^{1-m_j}(\partial_j u^{m_j} - \partial_j \hat{v}^{m_j})\d x \d t
    \\
    \nonumber & +\sum_{j=1}^N \int^{t_2}_{t_1}\int_{\Omega} a_j(x,t,u)(|\partial_j v^{m_j}|^{p_j-2}\partial_j v^{m_j} - |\partial_j u^{m_j}|^{p_j-2}\partial_j u^{m_j})
    \\ 
    \nonumber &\hspace{2cm}\times H_\delta'(u- \hat v) \tfrac1{m_j}(u^{1-m_j} - \hat{v}^{1-m_j}) \partial_j \hat{v}^{m_j} \d x \d t.
\end{align}
The first sum is nonpositive due to monotonicity and the fact that $v = \hat v$ on the set where $H_\delta'(u-\hat v)$ is nonzero. We show that the second sum converges to zero in the limit $\delta \to 0$. To see this, note that the map $s \mapsto s^{1-m_j}$ is Lipschitz on the interval $[\varepsilon/2, \infty)$ since $m_j \geq 1$. Since $u$ and $\hat v$ take values in this interval, we can estimate
\begin{align*}
    H_\delta'(u- \hat v) |u^{1-m_j} - \hat{v}^{1-m_j}| \leq c  H_\delta'(u- \hat v)|u-\hat v| \leq c.
\end{align*}
Thus, we can use the Dominated Convergence Theorem to see that the second sum in \eqref{panna} vanishes in the limit $\delta \to 0$. We have verified that the term involving $A$ on the right-hand side of \eqref{est:scirocco} can be written as a sum of a non-positive term and other terms that vanish in the limit $\delta \to 0$. Therefore, passing to the limit $\delta \to 0$, we end up with \eqref{est:comp_principle_special_case}.
\end{proof}
Using Proposition \ref{prop:comp_principle_special_case}, we can now prove the comparison principle of Theorem \ref{thm:comparison_principle}. Note that this is a more general result compared to Proposition \ref{prop:comp_principle_special_case} as we only require the supersolution to be  bounded from below on the lateral boundary by a positive constant.

\vspace{2mm}
\begin{proof}[Proof of Theorem \ref{thm:comparison_principle}]
Let $\kappa \in (0,\varepsilon)$ and set $u_\kappa = \max\{\kappa, u\}$. Due to Lemma \ref{lem:max_of_subsol_is_subsol} we can apply Proposition \ref{prop:comp_principle_special_case} to $u_\kappa$ and $v$ to conclude that 
\begin{equation*}
\int_{\Omega} (u_\kappa - v)_+ (x,t_2) \d x \leq \int_{\Omega} (u_\kappa - v)_+ (x,t_1) \d x + \iint_{\Omega \times[t_1,t_2]\cap \{v<u_\kappa \}}(f_u\chi_{\{u>\kappa\}} - f_v) \d x \d t.
\end{equation*}
Since $\chi_{\{ v < u_\kappa \}}$  converges to $\chi_{\{v<u \}\cup\{u=v=0\}}$ pointwise as $\kappa \to 0$, we use the Dominated Convergence Theorem to obtain \eqref{est:comp_principle_general}. 
Finally, if $u(0)\leq v(0)$ and $0\leq f_u \leq f_v$, we see that the last integral on the right-hand side of \eqref{est:comp_principle_general} vanishes, and the integrand in the first integral is non-positive. This leads to the estimate
\begin{align*}
    \int_{\Omega\times \{t\}} (u-v)_+ \d x \leq 0,
\end{align*}
for all $t \in (0,T)$, which confirms \eqref{u<v}.
\end{proof}

\appendix 
\section{Completeness}\label{app:completeness}
\begin{lem}\label{lem:Vmp-completeness}
The metric space $V^{\bf p,\bf m}_q$ is complete. 
\end{lem}
\begin{proof}{}
 Let $(u_k)_{k = 1}^\infty$ be a Cauchy sequence in $V^{\bf p,\bf m}_q$. Then $(u_k)_{k = 1}^\infty$ is a Cauchy sequence in $L^q(\Omega)$ and $(\partial_j u_k^{m_j})_{k=1}^\infty$ is a Cauchy sequence in $L^{p_j}(\Omega)$. By the completeness of $L^p$-spaces, there is $u\in L^q(\Omega)$ and $v_j \in L^{p_j}(\Omega)$ for every $j \in \{1,\dots N\}$ such that 
 \begin{align*}
  u_k \rightarrow u \textnormal{ in } L^q(\Omega), \quad \partial_j u_k^{m_j} \rightarrow v_j \textnormal{ in } L^{p_j}(\Omega) \textnormal{ for all } j \in \{1,\dots, N\}.
 \end{align*}
It is now sufficient to show that $v_j = \partial_j u^{m_j}$. If $m_j \geq 1$ then 
\begin{align*}
 |u_k^{m_j} - u^{m_j}| \leq m_j(|u_k|^{m_j - 1} + |u|^{m_j-1})|u_k - u|,
\end{align*}
and we can use H\"older's inequality and the fact that $q\geq m_j$ to conclude that $u_k^{m_j} \to u^{m_j}$ in $L^1(\Omega_T)$. If $m_j < 1$ we can instead use the basic estimate \eqref{est:exponent_inside} with $\gamma= 1/m_j$ to see that
\begin{align*}
 |u_k^{m_j} - u^{m_j}| \leq c |u_k - u|^{m_j}.
\end{align*}
The fact that $q \geq m_j$ allows us again to use H\"older's inequality and conclude that $u_k^{m_j} \to u^{m_j}$ in $L^1(\Omega_T)$. Thus, for all $\varphi \in C^\infty_{\textnormal{o}}(\Omega)$,
\begin{align*}
 \iint_{\Omega_T} u^{m_j} \partial_j \varphi \d x \d t = \lim_{k\to\infty} \iint_{\Omega_T} u_k^{m_j} \partial_j \varphi \d x \d t = -  \lim_{k\to\infty} \iint_{\Omega_T} \partial_j u_k^{m_j}  \varphi \d x \d t = \iint_{\Omega_T} v_j \varphi \d x,
\end{align*}
which confirms that $v_j = \partial_j u^{m_j}$.
\end{proof}

\section{Existence of solutions to the approximating problem}\label{app:Existence_approx_problems} 
This appendix is devoted to the proof of Theorem \ref{thm:app_existence_general_f}, which confirms that \eqref{eq:approximative_problem} has a solution also in the range $\bar p > 2$. The theorem is formulated and proved in a slightly more general setting than required, and the existence of solutions to \eqref{eq:approximative_problem} follows by choosing $\hat g = g + \tfrac1k$, $\hat u_0 = u_0 + \tfrac1k$.
\begin{theo}\label{thm:app_existence_general_f}
 Suppose that 
 \begin{align*}
 \hat A_j(x,t,u,\xi) = a_j(x,t,u)|\xi_j|^{p_j - 2}\xi_j,
\end{align*}
with $0<c_1 \leq a_j \leq c_2$ and $|a_j(x,t,u) - a_j(x,t,v)| \leq c_3|u-v|$ for all $x,t,u,v$ and 
 \begin{align*}
\hat g &\in L^{\bf p}(0,T; W^{1, {\bf p}}(\Omega)) \cap L^2(\Omega_T;[0,\infty)), \quad \partial_t \hat g \in L^2(\Omega_T), 
\\
f &\in L^{\bar p'}(\Omega_T;[0,\infty)), \quad \hat u_0 \in L^2(\Omega;[0,\infty)).
\end{align*}
Then there exists a solution $u \in L^{\bf p}(0,T; W^{1, {\bf p}}(\Omega)) \cap C([0,T]; L^2(\Omega))$ to the problem 
\begin{equation}\label{problem_with_general_f}
    \begin{cases}
    \partial_t u - \nabla \cdot \hat{A}(x, t, u, \nabla u) = f & \textnormal{in} \;\; \Omega_T, \\
    u = \hat g & \textnormal{on} \;\; \partial\Omega \times (0,T), \\
    u(\cdot, 0) = \hat u_0 & \textnormal{in} \;\; \Omega \times \{0\}.
\end{cases}
\end{equation}
\end{theo}
As in the main part of the article, the boundary condition by definition means that $u - \hat g$ belongs to $L^{\bf p}(0,T;\overline W^{1,\bf p}_{\textnormal{o}}(\Omega))$. If $\bar p' \geq 2$, the result follows directly from \cite[Theorem 2.5]{Ve} in the case $\alpha=1$. In the range $\bar p' < 2$ we define $f_l := \min\{l,f\}$ for $l\in \N$. By \cite[Theorem 2.5]{Ve} there is a solution $w_l$ in $L^{\bf p}(0,T; W^{1,\bf p}(\Omega)) \cap C([0,T]; L^2(\Omega))$ to the problem 
\begin{equation}\label{eq:really_approximative_problem}
    \begin{cases}
    \partial_t w_l - \nabla \cdot \hat{A}(x, t, w_l, \nabla w_l) = f_l & \text{in} \;\; \Omega_T, \\
    w_l = \hat g & \text{on} \;\; \partial\Omega \times (0,T), \\
    w_l(\cdot, 0) = \hat u_0 & \text{in} \;\; \Omega \times \{0\}.
\end{cases}
\end{equation}
Applying the comparison principle of Theorem \ref{thm:comparison_non-doubly-nonlinear} to $w_l$ and the zero-function we see that $w_l$ is non-negative. The strategy of the proof is to show that the functions $w_l$ converge to a solution to \eqref{problem_with_general_f}. For this reason we need the following uniform estimates.
\begin{lem}\label{lem:w_l-first-energy_est}
 The functions $w_l$ satisfy
 \begin{align}\label{est:w_l-first-energy}
  \iint_{\Omega_T} \sum^N_{j=1} |\partial_j w_l|^{p_j} \d x \d t + \sup_{\tau \in [0,T]}\int_{\Omega} w_l^2(x,\tau)\d x &\leq C_1(f,\hat g,\hat u_0,\Omega,T),
  \\
  \iint_{\Omega_T} |w_l - \hat g|^{\bar p} \d x \d t &\leq C_2(f,\hat g, \hat u_0,\Omega,T), \label{w_l--L-bar-p-bound}
 \end{align}
for constants $C_1$ and $C_2$ that are independent of $l$.
\end{lem}
\begin{proof}{}
Testing the weak formulation of the equation in \eqref{eq:really_approximative_problem} mollified with the reversed Steklov-average with the test function $\varphi = (w_l - \hat g)\zeta,$ where $\zeta \in C^\infty_{\textnormal{o}}((0,T);[0,\infty)),$ we end up with
\begin{align}\label{ta_se_an_te_anois}
 \iint_{\Omega_T} \hat{A}(x, t, w_l, \nabla w_l) \cdot \nabla w_l \zeta - \frac{w_l^2}{2} \zeta' \d x \d t 
 \leq \iint_{\Omega_T} \hat{A}(x, t, w_l, \nabla w_l) \cdot \nabla \hat g \zeta \d x \d t 
 \\
 \notag - \iint_{\Omega_T}w_l (\partial_t \hat g \zeta + \hat g \zeta') \d x\d t + \iint_{\Omega_T} f_l (w_l - \hat g)\zeta \d x \d t,
\end{align}
where for the treatment of the parabolic term we have used the estimate $\partial_t [\tfrac{w_l^2}{2}]_{\bar h} \leq \partial_t [w_l]_{\bar h} w_l,$ which itself follows from $(iv)$ in Lemma \ref{lem:steklov_mol_properties} and the fact that $s \mapsto \tfrac{s^2}{2}$ is convex.
We treat the term involving $f_l$ using Young's inequality and \eqref{eq:Sobolev_Troisi_Inequality} as follows:
\begin{align}\label{btut}
 \iint_{\Omega_T} f_l (w_l - \hat g)\zeta \d x \d t &\leq \iint_{\Omega_T} c_\varepsilon |f|^{\bar p'}\zeta + \varepsilon |w_l - \hat g|^{\bar p}\zeta \d x \d t 
 \\
 \notag &\leq c_\varepsilon \iint_{\Omega_T}|f|^{\bar p'}\zeta \d x \d t + c\varepsilon \iint_{\Omega_T} \sum^N_{j=1} |\partial_j (w_l - \hat g)|^{p_j} \zeta \d x \d t,
\end{align}
where we also used the fact that $0\leq f_l \leq f$. We find a lower bound for the first term on the left-hand side of \eqref{ta_se_an_te_anois} side using the structure condition \eqref{cond:structure2_A_k} and an upper bound for first term on the right-hand side using \eqref{cond:structure1_A_k} and Young's inequality. Combining these estimates and \eqref{btut} we obtain
\begin{align*}
\iint_{\Omega_T} c\sum^N_{j=1} |\partial_j w_l|^{p_j} \zeta -\frac{w_l^2}{2} \zeta' \d x \d t &\leq c\iint_{\Omega_T} \sum^N_{j=1} |\partial_j \hat g|^{p_j}\zeta + |f|^{\bar p'}\zeta \d x \d t 
\\
&\quad - \iint_{\Omega_T} w_l\big(\partial_t \hat g \zeta + \hat g \zeta' \big) \d x \d t.
\end{align*}
Taking $\zeta = \zeta^\delta_{\tau_1, \tau}$ and passing to the limits $\delta \to 0$ and $\tau_1 \to 0$ we have 
\begin{align*}
 c \iint_{\Omega_\tau} \sum^N_{j=1} |\partial_j w_l|^{p_j} \d x \d t & + \frac12 \int_{\Omega} w_l^2(x,\tau)\d x \leq c\iint_{\Omega_\tau} \sum^N_{j=1} |\partial_j \hat g|^{p_j} + |f|^{\bar p'}\d x \d t 
 \\
 & + \frac12 \int_{\Omega} {\hat u_0}^2 \d x - \iint_{\Omega_\tau} w_l \partial_t \hat g \d x \d t + \int_{\Omega}w_l \hat g(x,\tau) \d x
 \\
 &\leq  c\iint_{\Omega_\tau} \sum^N_{j=1} |\partial_j \hat g|^{p_j} + |f|^{\bar p'} + |\partial_t \hat g|^2 \d x \d t  + \tfrac12\norm{\hat u_0}_{L^2(\Omega)}^2 
 \\
 &\quad + c \int_\Omega \hat g^2(x,\tau) \d x + \frac14\int_{\Omega} w_l^2(x,\tau)\d x 
 + \frac{1}{8T} \iint_{\Omega_\tau} w_l^2 \d x \d t,
\end{align*}
where in the last step we used Young's inequality on the terms containing products of $w_l$ and $\hat g$ or $\partial_t \hat g$. Taking the supremum over $\tau$ and noting that the last term can be bounded by 
\begin{align*}
 \frac{1}{8T} \iint_{\Omega_\tau} w_l^2 \d x \d t \leq \frac{1}{8T} \int^\tau_0 \sup_{s\in [0,\tau]} \int_\Omega w_l^2 (x,s)\d x \d t \leq \frac18 \sup_{s\in [0,T]} \int_\Omega w_l^2 (x,s)\d x,
\end{align*}
we finally end up with 
\begin{align*}
 \iint_{\Omega_T} \sum^N_{j=1} |\partial_j w_l|^{p_j} \d x \d t + \sup_{\tau \in [0,T]}\int_{\Omega} w_l^2(x,\tau)\d x &\leq c\iint_{\Omega_\tau} \sum^N_{j=1} |\partial_j \hat g|^{p_j} + |f|^{\bar p'} + |\partial_t \hat g|^2 \d x \d t 
 \\
 & \quad + c\norm{\hat u_0}_{L^2(\Omega)}^2  + c \sup_{\tau \in [0,T]}\int_{\Omega} {\hat g}^2(x,\tau)\d x,
\end{align*}
which is an estimate of the form of \eqref{est:w_l-first-energy}. To prove \eqref{w_l--L-bar-p-bound} note that, by \eqref{eq:Sobolev_Troisi_Inequality}, we have
\begin{align*}
 \iint_{\Omega_T} |w_l - \hat g|^{\bar p} \d x \d t 
 \leq c\iint_{\Omega_T} \sum^N_{j=1} |\partial_j(w_l - \hat g)|^{p_j} \d x \d t \leq c\iint_{\Omega_T} \sum^N_{j=1} |\partial_j w_l|^{p_j} + |\partial_j \hat g|^{p_j}  \d x \d t,
\end{align*}
and then combine the last estimate with \eqref{est:w_l-first-energy}.
\end{proof}
From \eqref{est:w_l-first-energy}, \eqref{w_l--L-bar-p-bound} and the structure condition \eqref{cond:structure1_A_k} we can conclude the following:
\begin{enumerate}
 \item The sequence $(w_l)^\infty_{l=1}$ is bounded in $L^2(\Omega_T)$  and thus a subsequence converges weakly in this space to a limit function which we denote by $u$.
 
 \item\label{item:weakL-pbar-convg} The sequence $(w_l - \hat g)^\infty_{l=1}$ is bounded in $L^{\bar p}(\Omega_T)$ and thus a subsequence converges weakly in this space to a limit which we may identify as $u - \hat g$. Thus, $u - \hat g \in L^{\bar p}(\Omega_T)$.
 
 \item Each sequence $(\partial_j w_l)^\infty_{l=1}$ is bounded $L^{p_j}(\Omega_T)$ so a subsequence converges weakly in $L^{p_j}(\Omega_T)$ to a function which we may identify as $\partial_j u$.
 
 \item\label{app_vfield_weak_convg} By the structure condition each sequence $(\hat{A}_j(\cdot, \cdot, w_l, \nabla w_l))^\infty_{l=1}$ is bounded in $L^{p_j'}(\Omega_T)$ so a subsequence converges weakly to some $\hat{\mathcal{A}}_j \in L^{p_j'}(\Omega_T)$. 
\end{enumerate}
Moreover, the comparison principle of Theorem \ref{thm:comparison_non-doubly-nonlinear} and the fact that $f_l \leq f_{l+1}$ imply that the sequence $(w_l)^\infty_{l=1}$ is pointwise increasing almost everywhere. Combining this fact with Mazur's Lemma we can also conclude that $(w_l)^\infty_{l=1}$ converges pointwise a.e. to $u$. 

Since the weak and strong closure coincide on vector subspaces we can conclude that 
\begin{align}\label{eq:app_bdry_condition_satisfied}
 u - \hat g \in L^{\bf p}(0,T;\overline W^{1,\bf p}_{\textnormal{o}}(\Omega)),
\end{align}
i.e.~$u$ satisfies the correct boundary condition. 
By passing to the limit in the equation satisfied by $w_l$ we see by the above convergences that $u$ satisfies the equation 
\begin{align}\label{eq:u_strange}
 \iint_{\Omega_T} \hat{\mathcal{A}} \cdot \nabla \varphi  - u \partial_t \varphi \d x \d t = \iint_{\Omega_T} f\varphi  \d x \d t,
\end{align}
for all $\varphi \in C^\infty_{\textnormal{o}}(\Omega_T)$. 

\begin{lem}\label{lem:app_time_cont}
 The function $u$ has a representative in $C([0,T];L^2(\Omega))$ and $u(0) = \hat u_0$.
\end{lem}
\begin{proof}{}
Using the weak formulation as in the proof of \cite[Lemma 3.10]{Ve} in the special case $\alpha=1$ we end up with
\begin{align*}
 \iint_{\Omega_T} \tfrac12 (u - \hat g - w)^2\zeta' \d x \d t &= \iint_{\Omega_T} \hat{\mathcal{A}}\cdot \nabla (u - \hat g - w) \zeta \d x \d t  + \iint_{\Omega_T}f ( w + \hat g - u) \zeta \d x \d t
 \\
 & \quad + \iint_{\Omega_T}(u - \hat g - w)\partial_t (w + \hat g) \zeta \d x \d t
\end{align*}
where the choice $w= \Exx{u - \hat g}$ is admissible. The only difference is the integrability of the source function $f$, but since $u -\hat g$ lies in $L^{\bar p}(\Omega_T)$ we still have dual exponents in the terms involving $f$. Thus, also the proof of the time continuity of \cite[Lemma 3.11]{Ve} works in the current setting: we are able to approximate $u$ essentially uniformly with the functions $\hat g + \Exx{u - \hat g}$ which belong to $C([0,T];L^2(\Omega))$ due to the properties of $\hat g$ and the exponential time mollification. In order to prove that $u(0) = \hat u_0$ we test the equations for $w_l$ and $u$ with $\varphi = H_\delta(t)\zeta(t) \psi(x)$ where $\psi \in C^\infty_{\textnormal{o}}(\Omega)$ and pass to the limit $l\to \infty$  as in Section \ref{sec:boundary_val_time_cont_initial_data}.
\end{proof}
It remains to show that $u$ satisfies the equation in \eqref{problem_with_general_f}. 
\begin{lem}
 The function $u$ satisfies the identity 
 \begin{align}\label{Raviart_app_u}
 \iint_{\Omega_T} \hat{\mathcal{A}} \cdot \nabla (u - \hat g) \d x \d t =  \tfrac12\norm{\hat u_0}_{L^2(\Omega)}^2  - \tfrac12 \norm{u(T)}_{L^2(\Omega)}^2
 + \iint_{\Omega_T} f(u-\hat g)\d x \d t 
 \\
 \notag - \int_\Omega \hat u_0 \hat g(0)\d x + \int_\Omega u \hat g(x,T)\d x - \iint_{\Omega_T} u \partial_t \hat g \d x \d t.
\end{align}
whereas $w_l$ satisfies
 \begin{align}\label{Raviart_app_w_l}
 &\iint_{\Omega_T} \hat{A}(x, t, w_l, \nabla w_l) \cdot \nabla (w_l - \hat g) \d x \d t =  \tfrac12\norm{\hat u_0}_{L^2(\Omega)}^2  - \tfrac12 \norm{w_l(T)}_{L^2(\Omega)}^2
 \\
  \notag &+ \iint_{\Omega_T} f_l(w_l - \hat g)\d x \d t 
  - \int_\Omega \hat u_0 \hat{g}(0)\d x + \int_\Omega w_l \hat{g}(x,T)\d x - \iint_{\Omega_T} w_l \partial_t \hat g \d x \d t.
\end{align}
\end{lem}
\begin{proof}{}
 Both identities are essentially versions of \cite[Lemma 3.13]{Ve} in the case $\alpha=1$ except for the fact that the integrability assumption on $f$ is different. In the proof of \eqref{Raviart_app_u}, any terms involving $f$ can be treated using the fact that $u - \hat g$ is in $L^{\bar p}(\Omega_T)$ due to \eqref{item:weakL-pbar-convg}, and the fact that $f\in L^{\bar p'}(\Omega_T)$ by assumption, so that also in the current setting we have dual exponents in these terms. 
\end{proof}

\begin{lem}\label{lem:app_ptwise_convg_of_grad}
 For a suitable subsequence still labeled $(w_l)^\infty_{l=1}$ we have that $(\partial_j w_l)^\infty_{l=1}$ converges pointwise a.e. to $\partial_j u$ for every $j \in \{1, \dots , N\}$.
\end{lem}
\begin{proof}{}
We define
\begin{align*}
 &h_l:= \big(\hat A(x,t,w_l, \nabla w_l) - \hat A(x,t,w_l, \nabla u)\big)\cdot (\nabla w_l - \nabla u)
 \\
 &\quad = \sum^N_{j=1}a_j(x,t,w_l)(|\partial_j w_l|^{p_j -2}\partial_j w_l - |\partial_j u|^{p_j-2}\partial_j u)(\partial_j w_l - \partial_j u)
 \\
 &\quad \geq c \sum^N_{j=1} (|\partial_j w_l|^{p_j -2}\partial_j w_l - |\partial_j u|^{p_j-2}\partial_j u)(\partial_j w_l - \partial_j u).
\end{align*}
On the other hand
\begin{align*}
 \big(\hat A(x,t,w_l, \nabla w_l) - \hat A(x,t,w_l, \nabla u)\big)\cdot (\nabla w_l - \nabla u) 
  = \hat A(x,t,w_l, \nabla w_l)\cdot (\nabla w_l - \nabla \hat g) 
 \\
  + \hat A(x,t,w_l, \nabla w_l)\cdot ( \nabla \hat g - \nabla u) - \hat A(x,t,w_l, \nabla u)\cdot (\nabla w_l - \nabla u)
  \\
  =: i_1 + i_2 - i_3.
\end{align*}
We have
\begin{align*}
 I_3 := \iint_{\Omega_T} i_3 \d x \d t &=  \sum^N_{j=1}\iint_{\Omega_T}\big(a_j(x,t, w_l) - a_j(x,t,u)\big)|\partial_j u|^{p_j-2}\partial_j u(\partial_j w_l - \partial_j u) \d x \d t
 \\
 & \quad + \iint_{\Omega_T} \hat A(x,t,u, \nabla u)\cdot (\nabla w_l - \nabla u) \d x \d t.
\end{align*}
The integrals in the sum are seen to vanish in the limit $l\to \infty$ due to H\"older's inequality, the boundedness of the sequence $(\partial_j w_l)^\infty_{l=1}$ in $L^{p_j}(\Omega_T)$, the continuity of $a_j(x,t,u)$ with respect to $u$, the pointwise convergence of $w_l$ to $u$ and the Dominated Convergence Theorem. The integral on the last row vanishes by the weak convergence of the derivatives $(\partial_j w_l)^\infty_{l=1}$. 
Thus,
\begin{align}\label{app_lim_I_3}
 \lim_{l\to \infty} I_3 = 0. 
\end{align}
By \eqref{app_vfield_weak_convg} we have
\begin{align}\label{app_lim_I_2}
 I_2 :=  \iint_{\Omega_T} i_2 \d x \d t \xrightarrow[l\to \infty]{} \iint_{\Omega_T} \hat{\mathcal{A}} \cdot ( \nabla \hat g - \nabla u) \d x \d t.
\end{align}
We use \eqref{Raviart_app_w_l} to write
\begin{align}\label{app_I_1-expr}
 \notag I_1 := \iint_{\Omega_T} i_1 \d x \d t &=  \tfrac12\norm{\hat u_0}_{L^2(\Omega)}^2  - \tfrac12 \norm{w_l(T)}_{L^2(\Omega)}^2 + \iint_{\Omega_T} f_l(w_l - \hat g)\d x \d t 
  - \int_\Omega \hat u_0 \hat{g}(0)\d x 
  \\
  &\quad + \int_\Omega w_l \hat{g}(x,T)\d x - \iint_{\Omega_T} w_l \partial_t \hat g \d x \d t.
\end{align}
We can calculate 
\begin{align}\label{app_f-limit}
  \iint_{\Omega_T} f_l(w_l - \hat g)\d x \d t = &  \iint_{\Omega_T} f (w_l - \hat g)\d x \d t +  \iint_{\Omega_T} (f_l - f)(w_l - \hat g)\d x \d t
  \\
 \notag \xrightarrow[l\to \infty]{} &\iint_{\Omega_T} f (u - \hat g)\d x \d t,
\end{align}
where we use the weak convergence established in \eqref{item:weakL-pbar-convg} to treat the first integral on the right-hand side on the first line. The second integral on the right-hand side vanishes in the limit due to the uniform bound \eqref{w_l--L-bar-p-bound}, H\"older's inequality and the Dominated Convergence Theorem. Since $(w_l(T))^\infty_{l=1}$ is bounded in $L^2(\Omega)$ we may, after passing to a subsequence, assume that $(w_l(T))^\infty_{l=1}$ converges weakly in $L^2(\Omega)$ to a limit function which may be identified as $u(T)$ by testing the equations for $u$ and $w_l$ in a suitable manner, as done in the main existence proof of this paper. The weak convergence then implies that 
\begin{align}\label{app_lim_w_l(T)}
 \liminf_{l\to \infty} \norm{w_l(T)}_{L^2(\Omega)}^2 \geq \norm{u(T)}_{L^2(\Omega)}^2.
\end{align}
Using \eqref{app_f-limit} and \eqref{app_lim_w_l(T)} in \eqref{app_I_1-expr} we end up with
\begin{align}\label{app_limsup_I_1}
 \limsup_{l\to\infty} I_1 &\leq \tfrac12\norm{\hat u_0}_{L^2(\Omega)}^2  - \tfrac12 \norm{u(T)}_{L^2(\Omega)}^2
 + \iint_{\Omega_T} f(u-\hat g)\d x \d t 
 \\
 \notag &\quad - \int_\Omega \hat u_0 \hat g(0)\d x + \int_\Omega u \hat g(x,T)\d x - \iint_{\Omega_T} u \partial_t \hat g \d x \d t
 \\
 \notag &= \iint_{\Omega_T} \hat{\mathcal{A}} \cdot \nabla (u - \hat g) \d x \d t,
\end{align}
where in the last step we also used \eqref{Raviart_app_u}. Combining \eqref{app_lim_I_3}, \eqref{app_lim_I_2} and \eqref{app_limsup_I_1} we have 
\begin{align*}
 \limsup_{l\to \infty} \iint_{\Omega_T} h_l \d x \d t = 0.
\end{align*}
Since $h_l$ is non-negative this means that $(h_l)^\infty_{l=1}$ converges to zero in $L^1(\Omega_T)$ and by passing to a subsequence we may assume that $h_l$ converges pointwise a.e.~to zero. By the definition of $h_l$ this means that $(\partial_j w_l)_{l = 1}^\infty$ converges pointwise a.e.~to $\partial_j u$. 
\end{proof}

\vspace{2mm}
\noindent \textbf{Proof of Theorem \ref{thm:app_existence_general_f}.} 
The pointwise convergence of $(\partial_j w_l)^\infty_{l=1}$ to $\partial_j u$ established in Lemma \ref{lem:app_ptwise_convg_of_grad}, the pointwise convergence of $w_l$ to $u$ and the continuity of $a_j(x,t,u)$ w.r.t~$u$ show that $(\hat A(x,t,w_l,\nabla w_l))^\infty_{l=1}$ converges pointwise to $\hat{\mathcal{A}}$. Thus, a standard application of Mazur's lemma as in the proof of our main existence result shows that
\begin{align*}
 \hat{\mathcal{A}} = \hat{A}(x, t, u, \nabla u).
\end{align*}
Combined with \eqref{eq:u_strange} this confirms that $u$ satisfies the equation in  \eqref{problem_with_general_f}. By Lemma \ref{lem:app_time_cont} the initial condition in \eqref{problem_with_general_f} is satisfied and by \eqref{eq:app_bdry_condition_satisfied}, the boundary condition in \eqref{problem_with_general_f} is also satisfied.
\qed



\end{document}